\documentclass[11pt,reqno]{amsart}
\pdfoutput=1
\usepackage[all]{xy}
\usepackage{amsfonts}
\usepackage{amssymb}
\usepackage{amssymb, amsmath, amsthm, pstricks}
\usepackage{mathabx}

\usepackage[all]{xy}
\usepackage{tikz}
\usetikzlibrary{positioning}

\makeatletter
\@namedef{subjclassname@2020}{\textup{2020} Mathematics Subject Classification}
\makeatother

\usepackage[latin1]{inputenc}

\usepackage[pagebackref=true]{hyperref}

\usepackage{amsmath,amssymb,amsthm, latexsym, amscd}
\usepackage{tikz-cd}

\usepackage{hyperref}

\usepackage{mathtools}

\numberwithin{equation}{section}

\newtheorem{theorem}{Theorem}[section]
\newtheorem{proposition}[theorem]{Proposition}
\newtheorem{corollary}[theorem]{Corollary}
\newtheorem{lemma}[theorem]{Lemma}

\theoremstyle{definition}
\newtheorem{definition}[theorem]{Definition}
\newtheorem{example}[theorem]{Example}
\newtheorem{remark}[theorem]{Remark}

\newtheorem{question}[theorem]{Question}

\def\<{\langle}
\def\>{\rangle}

\def\D{{\Delta}}

\def\N{{\mathbb N}}
\def\R{{\mathbb R}}
\def\Z{{\mathbb Z}}

\def\ent{{\rm ent}}

\def\diam{\mathop{\rm diam}\nolimits}

\def\T{{\mathbb T}}
\def\1{\mathbf 1}

\newcommand{\ie}{{\it i.e.}}
\newcommand{\eg}{{\it e.g.}}

\newcommand{\vol}{{\rm vol}}
\newcommand{\length}{{\rm length}}
\newcommand{\eent}{{\rm ent}}

\newcommand{\UW}{{\rm UW}}

\newcommand{\HH}{{\mathbb H}}

\newcommand{\Conv}{{\rm Conv}}
\newcommand{\Ric}{{\rm Ric}}
\newcommand{\st}{{\rm st}}
\newcommand{\CAT}{{\rm CAT}}

\def\N{{\mathbb N}}

\long\def\forget#1\forgotten{} %

\begin{document}

\title{Minimal volume entropy and fiber growth}

\author[I.~Babenko]{Ivan Babenko}
\author[S.~Sabourau]{St\'ephane Sabourau}

\thanks{Partially supported by the ANR project Min-Max (ANR-19-CE40-0014).}

\address{Universit\'e Montpellier II, CNRS UMR 5149,
Institut Montpelli\'erain Alexander Grothendieck,
Place Eug\`ene Bataillon, B\^at. 9, CC051, 34095 Montpellier CEDEX 5, France} 

\email{babenko@umontpellier.fr}

\address{\parbox{\linewidth}{Univ Paris Est Creteil, CNRS, LAMA, F-94010 Creteil, France \\
Univ Gustave Eiffel, LAMA, F-77447 Marne-la-Vall\'ee, France}}


\email{stephane.sabourau@u-pec.fr}



\begin{abstract}
This article deals with topological assumptions under which the minimal volume entropy of a closed manifold~$M$, and more generally of a finite simplicial complex~$X$, vanishes or is positive.
These topological conditions are expressed in terms of the growth of the fundamental group of the fibers of maps from a given finite simplicial complex~$X$ to lower dimensional simplicial complexes~$P$.
This leads to a complete characterization of spaces with positive minimal volume entropy for finite simplicial complexes whose fundamental group has uniform uniform exponential growth with no subgroup of intermediate growth.
As pointed out to us by V.~Kapovitch, these conditions are related to collapsing with Ricci curvature bounded below and lead to a refinement of Gromov's isolation theorem.
We also give examples of finite simplicial complexes with zero simplicial volume and arbitrarily large minimal volume entropy.
\end{abstract}

\forget
\begin{abstract}
We present two results which ensure that, under some topological assumptions, the minimal volume entropy of a closed manifold, and more generally of a finite simplicial complex, is positive.
The topological assumptions of the first result are expressed in terms of the exponential growth of the fundamental group of the fibers of maps from a given finite simplicial complex to simplicial complexes of lower dimension.
In order to complement this approach, we also present topological conditions expressed in similar terms which ensure that the minimal volume entropy vanishes.
The topological assumptions of the second result are related to the exponential growth of certain subgroups in the fundamental group of the finite simplicial complex and to the topology of the loop space of its classifying space.
Several examples are presented throughout the text.
\end{abstract}
\forgotten

\maketitle



\section{Introduction}

The notion of volume entropy has attracted a lot of attention since the early works of Efremovich~\cite{Efrem53}, \v{S}varc~\cite{Shvarts55} and Milnor~\cite{Milnor68}.
This Riemannian invariant describes the asymptotic geometry of the universal cover of a Riemannian manifold and is related to the growth of its fundamental group; see~\cite{Shvarts55} and~\cite{Milnor68}.
It is also connected to the dynamics of the geodesic flow.
More specifically, the volume entropy agrees with the topological entropy of the geodesic flow of a closed nonpositively curved manifold and provides a lower bound for it in general; see~\cite{Dinaburg71} and~\cite{Manning79}.
In this article, we study the minimal volume entropy of a closed manifold (and more generally of a finite simplicial complex), a topological invariant introduced by Gromov~\cite{gro82} related to the simplicial volume.
More precisely, we give topological conditions which ensure, in one case, that the minimal volume entropy of a finite simplicial complex is positive and, in the other case, that it vanishes.
Before stating our results, we need to introduce some definitions.
Unless stated otherwise, all spaces are path-connected.

\begin{definition}
The \emph{volume entropy} of a connected finite simplicial complex~$X$ with a piecewise Riemannian metric~$g$ is the exponential growth rate of the volume of balls in the universal cover of~$X$.
More precisely, it is defined as
\begin{equation}\label{eq:entropy.comp.space}
\ent(X, g) =\lim_{R \to \infty} \frac{1}{R} \log( \vol \, \widetilde{B}(R))
\end{equation}
where $\widetilde{B}(R)$ is a ball of radius~$R$ centered at any point in the universal cover of~$X$.
The limit exists and does not depend on the center of the ball. 
Observe that the volume entropy of a finite simplicial complex with a piecewise Riemannian metric is positive if and only if its fundamental group has exponential growth; see Definition~\ref{def:algent}.


The \emph{minimal volume entropy} of a connected finite simplicial $m$-complex~$X$, also known as \emph{asymptotic volume}, see~\cite{B93}, is defined as
\[
\omega(X) = \inf_g \, \ent(X,g) \, \vol(X,g)^\frac{1}{m}
\]
where $g$ runs over the space of all piecewise Riemannian metrics on~$X$.
This topological invariant is known to be a homotopic invariant for closed manifolds~$M$, see~\cite{B93}, and more generally, an invariant depending only on the image of the fundamental class of~$M$ under the classifying map, see~\cite{Brunnbauer08}.
The exact value of the minimal volume entropy (when nontrivial) of a closed manifold is only known in a few cases; see~\cite{katok}, \cite{BCG91}, \cite{Sambusetti99}, \cite{Sambusetti00}, \cite{CB03}, \cite{merlin}.
For instance, the minimal volume entropy of a closed $m$-manifold~$M$ which carries a hyperbolic metric is attained by the hyperbolic metric and is equal to $(m-1) \, \vol(M,{\rm hyp})^\frac{1}{m}$; see~\cite{katok} for $m=2$ and~\cite{BCG91} for~$m \geq 3$.


The \emph{simplicial volume} of a connected closed orientable $m$-manifold~$M$ is defined as
\[
\Vert M \Vert_\Delta = \inf \left\{ \sum_s |r_s| \mid \sum_s r_s \, \sigma_s \text{ real singular } m\text{-cycle representing } [M] \in H_m(M;\Z) \right\}.
\]
where $r_s \in \R$ and $\sigma_s:\Delta^m \to M$ is a singular $m$-simplex.
The definition extends to finite simplicial $m$-complexes~$X$ whose fundamental class is well-defined, that is, with $H_m(X;\Z) \simeq \Z$.
\end{definition}

The following inequality of Gromov~\cite[p.~37]{gro82} connects the minimal volume entropy of a connected closed manifold to its simplicial volume (see also~\cite{BK19} for a presentation of this result).
Namely, every connected closed orientable $m$-manifold~$M$ satisfies
\begin{equation} \label{eq:gro}
\omega(M)^m \geq c_m \, \Vert M \Vert_\Delta
\end{equation}
for some positive constant~$c_m$ depending only on~$m$.
Thus, every closed manifold with positive simplicial volume has positive minimal volume entropy.
In particular, the minimal volume entropy of a closed manifold which carries a negatively curved metric is positive; see~\cite{gro82}.
Other topological conditions ensuring the positivity of the minimal volume entropy have recently been obtained in~\cite{sab17} and extended in~\cite[Section~4]{BS2} or~\cite{minentB}; see~\cite{BCGS} for a presentation of numerous examples and cases where these conditions apply.
These conditions are related to the topology of the loop space of the manifold. 
In a different direction, the minimal volume entropy provides a lower bound both on the minimal volume, see~\cite{gro82}, and on the systolic volume of a closed manifold, see~\cite{sab} and~\cite{Brunnbauer08}.

\medskip

A natural question to ask in view of~\eqref{eq:gro} is whether every closed orientable manifold with zero simplicial volume has zero minimal volume entropy.
This is known to be true in dimension two~\cite{katok} and in dimension three~\cite{pieroni} (see also~\cite{AP03} combined with Perelman's resolution of Thurston's geometrization conjecture), where the cube of the minimal volume entropy is proportional to the simplicial volume.
In dimension four, the same is known to be true but only for closed orientable geometrizable manifolds; see~\cite{SS09}.
The techniques developed in this article allow us to provide a negative answer for finite simplicial complexes; see Proposition~\ref{prop:ex}.
The question for closed orientable manifolds remains open despite recent progress made with the introduction of the volume entropy semi-norm; see~\cite{BS}.
This geometric semi-norm in homology measures the minimal volume entropy of a real homology class throughout a stabilization process.
Namely, given a path-connected topological space~$X$, it is defined for every ${\bf a} \in H_m(X;\Z)$ as
\begin{equation} \label{eq:defE}
\Vert {\bf a} \Vert_E = \lim_{k \to \infty} \frac{\omega(k \, {\bf a})^m}{k}
\end{equation}
where $\omega({\bf a})$ is the infimum of the minimal relative volume entropy of the maps $f:M \to X$ from an orientable connected closed $m$-pseudomanifold~$M$ to~$X$ such that $f_*([M])={\bf a}$; see~\cite{BS} for a more precise definition.
The volume entropy semi-norm shares similar functorial features with the simplicial volume semi-norm.
Moreover, the two semi-norms are equivalent in every dimension.
That is,
\begin{equation} \label{eq:BS}
c_m \,  \Vert {\bf a} \Vert_\Delta \leq  \Vert {\bf a} \Vert_E \leq C_m \, \Vert {\bf a} \Vert_\Delta
\end{equation}
for some positive constants $c_m$ and~$C_m$ depending only on~$m$.
Thus, a closed manifold with zero simplicial volume has zero volume entropy semi-norm, but its minimal volume entropy may be nonzero \emph{a priori}.
See~\cite{BS} for further details.

\medskip

More generally, one may ask for a topological characterization of closed manifolds or simplicial complexes with positive minimal volume entropy.
Such a topological characterization holds for the systolic volume, a topological invariant sharing similar properties with the minimal volume entropy; see~\cite{B93}, \cite{B06}, \cite{B08}, \cite{Brunnbauer08}.
Namely, a closed $m$-manifold or simplicial $m$-complex has positive systolic volume if and only if it is essential (\ie, its classifying map cannot be homotoped into the $(m-1)$-skeleton of the target space); see~\cite{gro83} and~\cite{B93}.
Though this condition is necessary to ensure that a closed manifold or simplicial complex has positive minimal volume entropy, see~\cite{B93}, it is not sufficient.
Therefore, one should look for stronger or extra assumptions.

\medskip

In this article, we present topological conditions in this direction.
The first one implies that the minimal volume entropy of a given simplicial complex vanishes and the second one ensures it is positive.
Both these conditions are expressed in terms of the exponential/subexponential growth of the fundamental group of the fibers of maps between a given simplicial complex and simplicial complexes of lower dimension.
We will need the following notions.

\begin{definition} \label{def:algent}
Let $G$ be a finitely generated group and $S$ be a finite generating set of~$G$.
Denote by~$B_S(t) \subseteq G$ the ball centered at the identity element of~$G$ and of radius~$t$ for the word distance induced by~$S$.
The group~$G$ has \emph{exponential growth} if the exponential growth rate of the number of elements in~$B_S(t)$ defined as
\[
\ent(G,S) = \lim_{t \to \infty} \frac{1}{t} \log | B_S(t) |
\]
is nonzero for some (and so any) finite generating set~$S$.
(By convention, a non-finitely generated group has exponential growth.)
The group~$G$ has \emph{uniform exponential growth} at least~$h>0$ if the exponential growth rate of the number of elements in~$B_S(t)$ is at least~$h$ for every finite generating set~$S$.
That is, its \emph{algebraic entropy} satisfies
\[
\ent(G) = \inf_S \ent(G,S) \geq h.
\]

The group~$G$ is \emph{$\delta$-thick} if it has exponential growth and every finitely generated subgroup $H \leqslant G$ with exponential growth has uniform exponential growth at least~$h$.
It is \emph{thick} if it is $\delta$-thick for some~$\delta >0$.
This notion is also referred to as \emph{uniform uniform exponential growth} or \emph{locally uniform exponential growth} in the literature.
The class of thick groups is fairly large, for instance, generic finitely presented groups are thick; see Section~\ref{sec:exthick} for further examples.


The group~$G$ has \emph{subexponential growth} if it does not have exponential growth.
In this case, the \emph{subexponential growth rate} of~$G$ is defined as
\[
\nu(G) = \limsup_{t \to \infty} \frac{\log \log |B_S(t)|}{\log t}.
\]
Note that the subexponential growth rate does not depend on the chosen finite generating set~$S$.

The group~$G$ has \emph{polynomial growth} if for some (and so any) finite generating set, there exists a polynomial~$P$ such that
\[
|B_S(t)| \leq P(t)
\]
for every~$t \geq 0$.
By~\cite{gro81}, a finitely generated group has polynomial growth if and only if it is virtually nilpotent.

The group~$G$ has \emph{intermediate growth} if its growth is subexponential but not polynomial.
The first group of intermediate growth was constructed by Grigorchuk~\cite{grigorchuk} and~\cite{grigorchuk2}, answering a question raised by Milnor.
Still, it is an open problem whether \emph{finitely presented} groups of intermediate growth exist.

Examples of finitely generated groups of exponential growth which do not have uniform exponential growth were first constructed by Wilson~\cite{Wilson04}, answering a question asked in~\cite{GH97} and~\cite{gro99}.
Still, it is an open question whether all \emph{finitely presented} groups of exponential growth have uniform exponential growth.
\end{definition}


For our topological conditions, we consider connected finite simplicial $m$-complexes~$X$ along with simplicial maps $\pi:X \to P$ onto simplicial complexes~$P$ of dimension at most~$k<m$, where $m \geq 2$. 
We denote by $i_*:\pi_1(F_p) \to \pi_1(X)$ the homomorphism induced by the inclusion map $i:F_p \hookrightarrow X$ of a connected component~$F_p$ of a fiber~$\pi^{-1}(p)$ of~$\pi$.

\medskip

The first condition considered for~$X$ is the fiber $\pi_1$-growth collapsing assumption (or fiber collapsing assumption for short).

\medskip

\textbf{Fiber $\pi_1$-growth collapsing assumption (FCA).}
Let $X$ be a finite connected simplicial $m$-complex.
Suppose there exists a simplicial map \mbox{$\pi:X \to P$} onto a simplicial complex~$P$ of dimension at most~$k<m$ such that for every connected component~$F_p$ of every fiber~$\pi^{-1}(p)$ with $p \in P$, the finitely generated subgroup~$i_*[\pi_1(F_p)] \leqslant \pi_1(X)$ has subexponential growth.



\medskip

The fiber $\pi_1$-growth collapsing assumption \emph{with polynomial growth rate} is defined similarly with the condition that all the finitely generated subgroup~$i_*[\pi_1(F_p)] \leqslant \pi_1(X)$ have polynomial growth.

\medskip

Likewise, the fiber $\pi_1$-growth collapsing assumption \emph{with subexponential growth rate at most~$\nu$} is defined similarly with the condition that the subexponential growth rate of all the finitely generated subgroup~$i_*[\pi_1(F_p)] \leqslant \pi_1(X)$ is at most~$\nu$.

\medskip

In these definitions, it is enough to check the condition for every vertex~$p \in P$ (but we will not need this result).

\medskip

The following result shows that if the subexponential growth rate in the fiber collapsing assumption is small enough then the minimal volume entropy of~$X$ vanishes.


\begin{theorem} \label{theo:omega=0}
Let $X$ be a connected finite simplicial $m$-complex satisfying the fiber $\pi_1$-growth collapsing assumption with subexponential growth rate at most~$\nu$ onto a simplicial $k$-complex~$P$.
Suppose that $\nu < \frac{m-k}{m}$.
Then $X$ has zero minimal volume entropy, that is,
\[
\omega(X) = 0.
\]
\end{theorem}

In Section~\ref{sec:ex}, we give an example of a closed manifold satisfying the assumption of Theorem~\ref{theo:omega=0} with a fiber whose image of the fundamental group is a finitely generated group of intermediate growth (which coincides with the first Grigorchuk group).
Recall that it is an open question whether finitely presented groups of intermediate growth exist.



\medskip

Since the subexponential growth rate of a group with polynomial growth is zero, we immediately derive the following corollary.

\begin{corollary} \label{coro:polynomial}
Every connected finite simplicial complex satisfying the fiber $\pi_1$-growth collapsing assumption with polynomial growth rate has zero minimal volume entropy.
\end{corollary}



As an application of Kapovitch-Wilking's Generalized Margulis Lemma (Theorem~\ref{theo:KW}), see~\cite{KW} and also~\cite{courtois}, Vitali Kapovitch pointed out to us that collapsing with
Ricci curvature bounded below implies the fiber $\pi_1$-growth collapsing assumption; see Proposition~\ref{prop:Ricci} for a more general statement.
Combined with Corollary~\ref{coro:polynomial}, this immediately implies the following.

\begin{corollary} \label{coro:isolation} 
For every positive integer~$m$, there exists $v_m>0$ such that every closed Riemannian $m$-manifold~$M$ with $\Ric_M \geq -(m-1)$ and $\vol(M) \leq v_m$ has zero minimal volume entropy.
\end{corollary}

This statement can be seen as a refinement of Gromov's isolation theorem~\cite[\S0.5]{gro82}, which asserts that under the same assumption as Corollary~\ref{coro:isolation} the manifold~$M$ has zero simplicial volume.

\medskip

\forget
Since closed manifolds admitting an $F$-structure satisfy the fiber collapsing assumption with polynomial growth rate, see Definition~\ref{def:F-structure} for the definition of an $F$-structure, we obtain the following result.
Note that this result can be derived from Paternain and Petean's work on the connection between the topological entropy of the geodesic flow and $F$-structures; see~\cite[Theorem~A]{PP03}.

\begin{corollary} \label{coro:F-structure}
Every closed manifold admitting an $F$-structure (of possibly zero rank) has zero minimal volume entropy.
\end{corollary}
\forgotten

\forget
The following corollary presents the effect of the fiber collapsing assumption on the simplicial volume.
It is a direct consequence of Theorem~\ref{theo:omega=0} showing that $\omega(M)=0$ and the inequality~\eqref{eq:gro} providing an upper bound on $||M||_\Delta$ in terms of~$\omega(M)$.

\begin{corollary} \label{coro:zerovol}
Every closed manifold~$M$ satisfying the fiber collapsing assumption has zero simplicial volume, that is,
\[
||M||_\Delta =0.
\]
\end{corollary}

An alternate proof via amenable coverings which does not rely on Theorem~\ref{theo:omega=0} is given in Section~\ref{sec:Gromov}; see Remark~\ref{rem:amenable}.

\medskip
\forgotten

The second condition considered for~$X$ is the fiber $\pi_1$-growth non-collapsing assumption (or non-collapsing assumption for short).

\medskip

\textbf{Fiber $\pi_1$-growth non-collapsing assumption (FNCA).}
Let $X$ be a finite connected simplicial $m$-complex.
Suppose that for every simplicial map \mbox{$\pi:X \to P$} onto a simplicial complex~$P$ of dimension~$k<m$, there exists a connected component~$F_{p_0}$ of some fiber~$\pi^{-1}(p_0)$ with $p_0 \in P$ such that the finitely generated subgroup~$i_*[\pi_1(F_{p_0})] \leqslant \pi_1(X)$ has uniform exponential growth at least~$h$ for some $h=h(X)>0$ depending only on~$X$.

\medskip

This topological condition ensures that the minimal volume entropy of~$X$ does not vanish.

\begin{theorem} \label{theo:omega>0}
Let $m \geq 3$.
Every connected finite simplicial $m$-complex~$X$ with thick fundamental group satisfying the fiber $\pi_1$-growth non-collapsing assumption has positive minimal volume entropy, that is, 
\[
\omega(X) > 0.
\]
\end{theorem}

It follows that the simplicial complex~$X$ in Theorem~\ref{theo:omega>0} has small enough volume, its minimal volume entropy is bounded away from zero.
This result still holds true if the unit balls of~$X$ (instead of the whole simplicial complex~$X$) have small enough volume; see Remarks~\ref{rem:entUW} and~\ref{rem:entUW2}.

\medskip

As showed in Section~\ref{sec:exthick}, closed aspherical manifolds whose fundamental group is a non-elementary word hyperbolic group satisfy the conditions of Theorem~\ref{theo:omega>0}.

\medskip

Note that the fibers of the simplicial map $\pi:X \to P$ in the definition of the fiber collapsing and non-collapsing conditions can always be assumed to be connected; see Proposition~\ref{prop:connected}.

\medskip

The definitions of the fiber collapsing and fiber non-collapsing assumptions are exclusive but not complementary in general.
However, every simplicial complex with a thick fundamental group satisfies either the fiber collapsing assumption or the fiber non-collapsing assumption; see Proposition~\ref{prop:alternative}.
This leads to a complete characterization of spaces with positive minimal volume entropy for finite simplicial complexes whose fundamental group is thick with no subgroup of intermediate growth.

\begin{corollary}
Let $X$ be connected finite simplicial $m$-complex with $m \geq 3$ whose fundamental group is thick with no subgroup of intermediate growth.
Then, either $X$ satisfies the fiber collapsing assumption, in which case its minimal volume entropy is zero, or $X$ satisfies the fiber non-collapsing assumption, in which case its minimal volume entropy is positive.
\end{corollary}


We also give alternative formulations of both the fiber collapsing and non-collapsing assumptions in terms of open coverings of the simplicial complex~$X$, namely, the covering collapsing assumption (CCA) and the the covering non-collapsing assumption (CNCA); see Proposition~\ref{prop:cover} and Proposition~\ref{prop:exp.cover}.
This yields a result similar to Theorem~\ref{theo:omega>0} which also applies to simplicial complexes with non-thick fundamental group; see Theorem~\ref{theo:B.bis}.

\medskip

\forget
Though Theorem~\ref{theo:omega>0} and Theorem~\ref{theo:omega=0} are expressed in almost complementary terms, they do not provide a complete topological characterization of simplicial complexes with positive minimal volume entropy.
As an additional comment, let us mention that the distinction between exponential growth and uniform exponential growth for finitely generated groups is quite subtle.
Examples of finitely generated groups of exponential growth which do not have uniform exponential growth were first constructed by Wilson~\cite{Wilson04}, answering a question of Gromov.
Still, it is an open question whether all \emph{finitely presented} groups of exponential growth have uniform exponential growth.
\forgotten





The techniques developed in this article allow us to investigate the relationship between the minimal volume entropy and the simplicial volume of simplicial complexes whose fundamental class is well-defined.
In view of the lower and upper bounds~\eqref{eq:BS}, one can ask whether there is a complementary inequality to the bound~\eqref{eq:gro}.
Namely, does there exist a positive constant~$C_m$ such that
\[
\omega(M)^m \leq C_m \, \Vert M \Vert_\Delta
\]
for every connected closed orientable $m$-manifold~$M$?
The question also makes sense for every connected finite simplicial $m$-complex~$X$ whose fundamental class is well-defined.
Our next result provides a negative answer in this case.

\begin{proposition} \label{prop:ex}
There exists a sequence of connected finite simplicial complexes~$X_n$ with a well-defined fundamental class such that the simplicial volume of~$X_n$ vanishes for all $n \in \N$ and the minimal volume entropy of~$X_n$ tends to infinity.
\end{proposition}


We emphasize that both Theorem~\ref{theo:omega=0} and Theorem~\ref{theo:omega>0} hold for the class of finite simplicial complexes (including compact ${\rm CAT}(0)$ simplicial or cubical complexes) and not solely for closed manifolds.
This contrasts with all previous works, which focus on closed manifolds.
In particular, the topological conditions ensuring the positivity of the minimal volume entropy, see Theorem~\ref{theo:omega>0}, apply to simplicial complexes for which the simplicial volume is zero and the inequality~\eqref{eq:gro} does not readily extend.
This is exemplified by Proposition~\ref{prop:ex}. 

\medskip

Since a first version of this work appeared as the first part of our preprint~\cite{BS2} (before we extended it and decided to split it), the results established in this article have already  found applications in~\cite{BC} and~\cite{LM}. 

\medskip

\emph{Acknowledgment.} The second author would like to thank the Fields Institute and the Department of Mathematics at the University of Toronto for their hospitality while part of this work was completed. 
We express our gratitude to Rostislav Grigorchuk for multiple stimulating discussions and to Vitali Kapovitch for pointing out to us a connection to collapsing with Ricci curvature bounded from below.
Finally, we thank Corey Bregman and Matt Clay who pointed out a mistake in a previous version of this article and drew our attention on their recent work~\cite{BC}.

\section{Simplicial complexes with zero minimal volume entropy}

In this section, we first introduce the covering collapsing assumption and show that it is equivalent to the fiber growth collapsing assumption.
Then, we show the central result of this section, namely, the minimal volume entropy of a finite simplicial complex satisfying the fiber growth collapsing assumption with small subexponential growth rate vanishes.
Several examples of manifolds satisfying the fiber growth collapsing assumption are presented throughout this section.
We conclude this section with an extension of Gromov's isolation theorem.

\subsection{Covering collapsing assumption} \label{sec:zerovol}

\mbox { }

\medskip

We begin with the following definition.

\begin{definition} \label{def:subexp}
A path-connected open subset~$U$ of a path-connected topological space~$X$ has \emph{subexponential $\pi_1$-growth (resp. polynomial $\pi_1$-growth) in~$X$} if the subgroup $\Gamma_U:=i_*[\pi_1(U)]$ of~$\pi_1(X)$ has subexponential growth (resp. polynomial growth), where $i:U \hookrightarrow X$ is the inclusion map.
In this case, the subexponential $\pi_1$-growth rate of~$U$ in~$X$ is defined as the subexponential growth rate of~$\Gamma_U$.
\end{definition}

\textbf{Covering collapsing assumption (CCA).}
Let $X$ be a finite connected simplicial $m$-complex.
Suppose there exists a covering of~$X$ of multiplicity at most~$m$ by open subsets of subexponential $\pi_1$-growth in~$X$ (with subexponential growth rate at most~$\nu$ or polynomial growth rate).

\medskip

The following classical result implies that the notions of collapsing in terms of open coverings (CCA) or of fiber growth (FCA) are equivalent.


\begin{proposition} \label{prop:cover}
A connected finite simplicial $m$-complex~$X$ admits a covering of multiplicity~$k+1$ by open subsets of subexponential $\pi_1$-growth in~$X$ (with subexponential growth rate at most~$\nu$ or polynomial growth rate) if and only if there exists a simplicial map $\pi:X \to P$ onto a simplicial $k$-complex such that for every connected component~$F_p$ of every fiber~$\pi^{-1})(p)$, the subgroup $i_*[\pi_1(F_p)] \leqslant \pi_1(X)$ has subexponential growth (with subexponential growth rate at most~$\nu$ or polynomial growth rate).
%
\end{proposition}

\begin{proof}
Suppose that $X$ satisfies the fiber collapsing assumption.
Then there exists a simplicial map $\pi:X \to P$ onto a simplicial $k$-complex~$P$ such that for every connected component~$F_p$ of every fiber~$\pi^{-1}(p)$, where $p$ is a vertex of~$P$, the subgroup $i_*[\pi_1(F_p)]$ of~$\pi_1(X)$ has subexponential growth (resp. polynomial growth).
Since $P$ is a finite simplicial complex of dimension~$k$, the open covering formed by the open stars~$\textrm{st}(p) \subseteq P$ of the vertices~$p$ of~$P$ has multiplicity~$k+1$.
The connected components of the preimages~$\pi^{-1}(\textrm{st}(p)) \subseteq X$ of these open stars form an open covering of~$X$ with the same multiplicity~$k+1$ as the previous covering of~$P$.
Furthermore, the open subsets of this open covering of~$X$ strongly deformation retract onto the connected components~$F_p$ of the fibers~$\pi^{-1}(p)$.
In particular, they have subexponential $\pi_1$-growth in~$X$ with the same subexponential growth rate as the subgroups induced by the fibers (resp. polynomial growth). 
This proves the first implication. 

\medskip

\forget
Consider the first barycentric subdivision~$\textrm{sd}^1(P)$ of~$P$ and denote by~$\textrm{st}^1(p)$ the star of a vertex~$p$ for this subdivision.
The open subsets $\textrm{st}^1_\varepsilon(p)$ defined as the $\varepsilon$-neighborhoods of~$\textrm{st}^1(p)$ (with respect to the metric on~$P$ where every simplex of the subdivision is isometric to the standard simplex) for $p$ lying in the vertex set~$V(P)$ of~$P$ form an open covering of~$P$.
For $\varepsilon$ small enough, the multiplicity of this covering is equal to $k+1 \leq m$.
The connected components of the preimages~$\pi^{-1}(\textrm{st}^1_\varepsilon(p))$ of these open subsets form an open covering~$\{ U_i \}$ of~$X$ with the same multiplicity $k+1 \leq m$ as the previous covering of~$P$.
For $\varepsilon$ small enough, every open subset~$U_i$ strongly deformation retracts onto a connected component~$C_i$ of some fiber~$\pi^{-1}(p)$ with $p \in V(P)$.
Thus, the open subsets of the covering~$\{ U_i \}$ have subexponential $\pi_1$-growth in~$X$. 
This prove the first implication. \\
\forgotten

For the converse implication, let $\{ U_i \}_{i=0,\dots,s}$ be a covering of~$X$ of multiplicity~$k+1$ by open subsets of subexponential $\pi_1$-growth (resp. polynomial $\pi_1$-growth) in~$X$.
Take a partition of unity~$\{ \phi_i \}$ of~$X$, where each function~$\phi_i:X \to [0,1]$ has its support in~$U_i$.
Consider the map $\Phi:X \to \Delta^{s}$ defined by
\[
\Phi(x) = (\phi_0(x),\dots,\phi_s(x))
\]
in the barycentric coordinates of~$\Delta^{s}$.
The nerve~$P$ of the covering~$\{ U_i \}$ is a simplicial complex with one vertex~$v_i$  for each open set~$U_i$, where $v_{i_0},\dots,v_{i_n}$ span an $n$-simplex of~$P$ if and only if the intersection~$\cap_{j=1}^n U_{i_j}$ is nonempty.
By construction, the dimension of the nerve~$P$ is one less than the multiplicity of the covering~$\{ U_i \}$.
That is, $\dim P = k$.
We identify in a natural way the vertices~$\{ v_i \}$ of~$P$ with the vertices of~$\Delta^{s}$.
With this identification, the nerve~$P$ of~$X$ lies in~$\Delta^{s}$.
Furthermore, the image of~$\Phi$ lies in~$P$.
By~\cite[\S2.C]{hatcher}, subdividing~$X$ and~$P$ if necessary, we can approximate $\Phi:X \to P$ by a simplicial map $\pi:X \to P$ close to~$\Phi$ for the $C^0$-topology, whose normalized barycentric coordinates $\pi_i:X \to [0,1]$ have their support in~$U_i$.
Thus, every fiber~$\pi^{-1}(p)$ lies in one of the open subsets~$U_i$.
Therefore, for every connected component~$F_p$ of~$\pi^{-1}(p)$, the subgroup~$i_*[\pi_1(F_p)]$ lies in some subgroup~$i_*[\pi_1(U_i)]$.
Since the open subsets~$U_i$ have subexponential $\pi_1$-growth (resp. polynomial $\pi_1$-growth) in~$X$, the subgroups~$i_*[\pi_1(F_p)]$ have subexponential growth with a subexponential growth rate bounded by the one of the subsets of the open covering (resp. polynomial growth) and the simplicial complex~$X$ satisfies the fiber collapsing assumption as required.
\forget
Suppose that the covering~$\mathcal{U}$ is minimal (\ie, every proper subfamily of~$\mathcal{U}$ does not cover~$X$).
Take a partition of unity~$\{ \phi_i \}$ subordinate to the open covering~$\mathcal{U} = \{U_i \}$.
The open subset 
\[
V_i \{ x \in U_i \mid \phi_i(x) > \frac{1}{2s} \}
\]
form an open covering~$\mathcal{V}$ of~$X$, which is finer than~$\mathcal{U}$.
Without loss of generality, we can assume that the open subsets of~$\mathcal{V}$ are connected, otherwise take the finer covering composed of the connected components of the subsets of~$\mathcal{V}$.
We can also assume that the covering is minimal.
Note that the multiplicity of~$\mathcal{V}$ is bounded by the multiplicity of~$\mathcal{U}$.
Thus, the multiplicity of~$\mathcal{V}$ is at most~$n$.
Since the covering~$\mathcal{V}$ is finer than~$\mathcal{U}$, its subsets have subexponential $\pi_1$-growth in~$X$.
Consider a subdivision~$\Theta$ of~$X$ fine enough so that every simplex~$\tau$ of~$\Theta$ intersecting an open subset~$V_i$ of~$\mathcal{V}$ lies in the corresponding open subset~$U_i$ of~$\mathcal{U}$.
Take a partition of unity~$\{ \psi_i \}$ subordinate to the open covering~$\mathcal{V} = \{V_i \}$.
Consider the map $\Psi:X \to \Delta^{s-1}$ defined by
\[
\Psi(x) = (\psi_1(x),\cdots,\psi_s(x))
\]
in the barycentric coordinates of~$\Delta^{s-1}$.
The nerve~$P$ of the covering~$\{ V_i \}$ is a simplicial complex with one vertex~$v_i$  for each open set~$V_i$, where $v_{i_0},\cdots,v_{i_k}$ span a $k$-simplex of~$P$ if and only if the intersection~$\cap_{j=1}^k V_{i_j}$ is nonempty.
By construction, the dimension of the nerve~$P$ is one less than the multiplicity of the covering~$\mathcal{V}$.
That is, $\dim P \leq m-1$.

Let us construct a simplicial approximation of~$\Psi$.
\forgotten
\end{proof}

An illustration of the characterization of the fiber collapsing assumption in terms of open coverings is given by the following example.

\begin{example}
For $i=1,2$, let $M_i$ be a connected closed manifold of dimension~$m \geq 3$ with fundamental group~$\pi_1(M_i)$ of subexponential growth rate at most~$\nu < \frac{m-1}{m}$.
Let $N$ be a connected closed $n$-manifold embedded both in~$M_1$ and~$M_2$ with $n \leq m-3$.
Suppose that the embedding $N \subseteq M_i$ induces a $\pi_1$-monomorphism and that its normal fiber bundle $\mathcal{N}_i(N) \subseteq TM_i$ is trivial for $i=1,2$.
Define the $m$-manifold
\[
X = (M_1 \setminus U_1(N)) \mathop{\cup}_{N \times S^{m-n-1}} (M_2 \setminus U_2(U))
\]
where $U_i(N)$ is a small tubular neighborhood of~$N$ in~$M_i$.
By van Kampen's theorem, $\pi_1(M_i \setminus U_i(N))$ is isomorphic to~$\pi_1(M_i)$, and thus has subexponential growth rate at most~$\nu$.
Take a small tubular neighborhood~$U_i$ of~$M_i \setminus U_i(N)$ in~$X$ for $i=1,2$.
Since $U_i$ strongly deformation retracts onto~$M_i \setminus U_i(N)$, its fundamental group~$\pi_1(U_i)$ is isomorphic to~$\pi_1(M_i \setminus U_i(N))$.
This yields a covering of~$X$ of multiplicity two by open subsets~$U_1$ and~$U_2$ with subexponential $\pi_1$-growth at most~$\nu$ in~$X$.
According to Proposition~\ref{prop:cover}, the closed $m$-manifold~$X$ satisfies the fiber collapsing assumption.
Note however that the fundamental group of~$X$ has exponential growth in general.
This construction provides numerous examples of closed essential manifolds with a fundamental group of exponential growth and zero minimal volume entropy. 
For instance, when $N$ is reduced to a singleton, the manifold~$X$ is the connected sum~$M_1 \# M_2$ of~$M_1$ and~$M_2$.
This special case can also be recovered from~\cite[Theorem~2.8]{BS}.
\end{example}


\subsection{Connected and non-connected fibers}

\mbox { }

\medskip

The following result shows that we can assume that the fibers of the simplicial map $\pi:X \to P$ in the definition of the fiber collapsing and non-collapsing conditions are connected.

\begin{proposition} \label{prop:connected}
Let $\pi:X \to P$ be a simplicial map between two finite simplicial complexes.
Denote by~$k$ the dimension of~$P$.
Then there exists a surjective simplicial map $\bar{\pi}:X \to \bar{P}$ to a finite simplicial complex~$\bar{P}$ of dimension at most~$k$ such that the fibers of $\bar{\pi}:X \to \bar{P}$ agree with the connected components of the fibers of $\pi:X \to P$.
\end{proposition}

\begin{proof}
Without loss of generality, we can assume that the simplicial map $\pi:X \to P$ is onto.
Define $\bar{P} = X / \! \sim$ as the quotient space of~$X$, where $x \sim y$ if $x$ and~$y$ lie in the same connected component of a fiber of~$\pi:X \to P$.
Since the map $\pi:X \to P$ is simplicial, the quotient space~$\bar{P}$ is a simplicial complex of the same dimension as~$P$.
By construction, the map $\pi:X \to P$ factors out through a simplicial map $\bar{\pi}:X \to \bar{P}$ whose fibers agree with the connected components of the fibers of~$\pi:X \to P$.
\end{proof}


\subsection{Construction of a family of piecewise flat metrics} \label{sec:g_t} 

\mbox { }

\medskip

Let $\pi:X \to P$ be simplicial map from a connected finite simplicial $m$-complex~$X$ to a simplicial $k$-complex~$P$ with~$k<m$.
We will assume that the map $\pi:X \to P$ is onto and that its fibers~$F_p$ are connected; see Proposition~\ref{prop:connected}.

\medskip

The goal of this section is to construct a family of piecewise flat metrics~$g_t$ on~$X$ which collapses onto~$P$ (\ie, for which the map $\pi:X \to P$ is $1$-Lipschitz and the length of its fibers goes to zero).
The construction relies on some simplicial embeddings of~$X$ and~$P$ into an Euclidean space~$E$ of large dimension.

\medskip

Let $\Delta^s=\Delta^s(p_0,\dots,p_s)$ be the abstract $s$-simplex with the same vertices $p_0,\dots,p_s$ as~$P$.
Fix an $(s+1)$-dimensional Euclidean space~$H$ with an orthonormal basis $e_0,\dots,e_s$.
Identify the abstract $s$-simplex~$\Delta^s$ with the regular $s$-simplex of~$H$ with vertices $\frac{1}{\sqrt{2}} e_0,\dots,\frac{1}{\sqrt{2}} e_s$.
Define the subcomplex
\[
R_i = \pi^{-1}(p_i) \subseteq X.
\]
As previously, let $\Delta(R_i)$ be the abstract simplex with the same vertices as~$R_i$.
Denote by~$m_i$ the dimension of~$\Delta(R_i)$.
Fix an $(m_i+1)$-dimensional Euclidean space~$H_i$ with an orthonormal basis $e_0^i,\dots,e_{m_i}^i$.
Identify the abstract $m_i$-simplex~$\Delta(R_i)$ with the regular $m_i$-simplex of~$H_i$ with vertices $\frac{1}{\sqrt{2}} e^i_0,\dots,\frac{1}{\sqrt{2}} e^i_{m_i}$.

Consider the orthogonal sum
\begin{equation} \label{eq:E}
E = H \oplus H_0 \oplus \dots \oplus H_s.
\end{equation}
Denote by~$g_E$ the scalar product on~$E$.
There is a natural piecewise affine embedding $\chi:X \hookrightarrow E$ taking every vertex $v \in X$, identified with some element $\frac{1}{\sqrt{2}} e_j^i$ with $0 \leq i \leq s$ and $0 \leq j \leq m_i$, to
\[
\chi(v) = \tfrac{1}{\sqrt{2}} e_i + \tfrac{1}{\sqrt{2}} e_j^i.
\]
(Here, a piecewise affine embedding means an embedding whose restriction to each simplex is an affine map.)
Note that the distance between the images of any pair of vertices of~$X$ is bounded by~$\sqrt{2}$.
By construction, the whole space~$R_i$ is sent under $\chi:X \hookrightarrow E$ into the subspace~$H'_i = \frac{1}{\sqrt{2}} e_i + H_i$ orthogonal to~$H$, parallel to~$H_i$ and passing through~$\frac{1}{\sqrt{2}} e_i$.
By our choices of identification, the composition of~$\chi: X \hookrightarrow E$ with the orthogonal projection $p_H:E \to H$ onto~$H$ coincides with the simplicial map $\pi:X \to P$, that is,
\[
\pi = p_H \circ \chi.
\]

The piecewise flat metric on~$X$ induced by the piecewise affine embedding $\chi: X \hookrightarrow E$ can be deformed as follows.
Let $h_t:E \to E$ be the endomorphism of~$E$ preserving each factor of the decomposition~\eqref{eq:E} whose restriction to~$H$ is the identity map and restriction to each~$H_i$ is the homothety with coefficient~$t$.
For every $t \in (0,1]$, the map $\chi_t: X \hookrightarrow E$ defined as
\[
\chi_t = h_t \circ \chi
\]
is a piecewise affine embedding.
Note that $h_t$ preserves the subspaces~$H'_i$.
By construction, we still have
\[
\pi = p_H \circ \chi_t.
\]
Endow~$X$ with the piecewise flat metric~$g_t$ induced by the piecewise affine embedding~$\chi_t: X \hookrightarrow E$ defined as
\begin{equation} \label{eq:gt}
g_t = \chi_t^* (g_E).
\end{equation}
Endow also~$P$ with the natural piecewise flat metric~$g_P$ where all its simplices are isometric to the standard Euclidean simplex induced by the piecewise affine embedding~$P \subseteq H \subseteq E$.
The projection $p_H:E \to H$ is $1$-Lipschitz both for the metrics~$g_E$ and~$h_t^*(g_E)$ on~$E$, where $H$ is endowed with the restriction of~$g_t$ to~$H$.
It follows that $\pi=p_H \circ h_t \circ \chi:X \to P$ is $1$-Lipschitz.
Observe also that the $g_t$-length of every edge lying in some fiber~$\pi^{-1}(p_i) \subseteq X$ over a vertex $p_i \in P$ is equal to~$t$.
Since $P$ is a $k$-dimensional simplicial complex, we conclude that 
\begin{equation} \label{eq:volt}
\vol(X,g_t) = O(t^{m-k})
\end{equation}
as $t$ goes to zero.
Note also that for every simplex~$\Delta$ of~$X$, we have
\begin{equation} \label{eq:diamgt}
\diam(\Delta,g_t) \leq \sqrt{2}.
\end{equation}

\subsection{Construction of Lipschitz retractions around each fiber}

\mbox { }

\medskip

Using the same notations as in the previous section, we construct a Lipschitz retraction from a neighborhood of each fiber of~$\pi:X \to P$ above a vertex of~$P$ onto the fiber itself.
This is an important technical result which will be used in Section~\ref{sec:deforming} to deform paths of~$X$ into the $1$-skeleton of~$X$ without increasing their $g_t$-length too much (uniformly in~$t$).

\medskip

More precisely, we have

\begin{lemma} \label{lem:Xv}
There exist some constants $\tau_m \geq \frac{1}{2}$ and $\varepsilon_m, \sigma_m \in (0,1)$ with $\varepsilon_m \leq \tau_m$ depending only on~$m$ such that for every~$v \in P$, there exists a closed neighborhood~$X_v \subseteq X$ of~$\pi^{-1}(v)$ such that the following properties hold for every $t \in (0,1]$.
\begin{enumerate}
\item The subset~$X_v \subseteq X$ lies in the (open) star of~$\pi^{-1}(v)$ and contains all the points of~$X$ at $g_t$-distance at most~$\tau_m$ from~$\pi^{-1}(v)$. 
\item For every point~$z \in \partial X_v$, denote by~$\Delta_X$ the smallest simplex of~$X$ containing~$z$.
Pick a vertex~$z_- \in \Delta_X$ lying in~$\pi^{-1}(v)$ and a vertex~$z_+ \in \Delta_X$ not lying in~$\pi^{-1}(v)$ at minimal $g_t$-distance from~$z$.
Then, 
\begin{equation} \label{eq:z}
d_{g_t}(z,z_+) \leq d_{g_t}(z,z_-) - \varepsilon_m
\end{equation}
and
\begin{equation} \label{eq:zz}
d_{g_t}(z,z_+) + \sigma_m \leq \tau_m.
\end{equation}
\end{enumerate}

Furthermore, there exists $\kappa_m$-Lipschitz retraction
\[
\varrho_t:X_v \to \pi^{-1}(v)
\]
where $\kappa_m$ is a constant depending only on~$m$.
\end{lemma}

\begin{proof}
Say $v=p_0$.
Let $\Delta^q = \Delta_P^q$ be a $q$-simplex of~$P$ containing~$v$.
Recall that $\Delta^q$ lies in~$H$; see Section~\ref{sec:g_t}.
Denote by~$\Delta_v^{q-1}$ the $(q-1)$-face of~$\Delta^q$ opposite to~$v$.
Consider a $p$-simplex~$\Delta_X^p$ of~$X$ mapped onto~$\Delta_P^q$ under $\pi:X \to P$.
The intersection~$\pi^{-1}(v) \cap \Delta_X^p$ is a simplex of~$X$, whose dimension is denoted by~$r$.
By construction, the map $\pi:X \to P$ sends the $r$-simplex $\delta_0^r := \pi^{-1}(v) \cap \Delta_X^p$ of~$\Delta_X^p$ to~$v$.
Construct a retraction
\[
\bar{\varrho}_t:\Delta_X^p \setminus \pi^{-1}(\Delta_v^{q-1}) \to \delta_0^r
\]
onto~$\delta_0^r$ as follows.
First, embed $\Delta_X^p$ into the Euclidean space~$E$ through $\chi_t:X \hookrightarrow E$.
Under this identification, the image~$h_t(\delta_0^r)$ of~$\delta_0^r$ lies in the subspace~$H^v_0$ orthogonal to~$H$, parallel to~$H_0$ and passing through~$v$.
Then, take the orthogonal projection to~$H \oplus H_0$. 
Note that the image of~$\Delta_X^p$ under the composition of these maps agrees with the convex hull $\Conv(h_t(\delta_0^r) \cup \Delta_v^{q-1})$.
Thus, every point~$x \in \Delta_X^p \setminus \pi^{-1}(\Delta_v^{q-1})$ is sent to a point~$\bar{x} \in \Conv(h_t(\delta_0^r) \cup \Delta_v^{q-1})$.
Then, for every $\bar{x} \in \Conv(h_t(\delta_0^r) \cup \Delta_v^{q-1}) \setminus \Delta_v^{q-1}$ not lying in~$h_t(\delta_0^r)$, take the orthogonal projection $\bar{x}' \in \Delta^q$ of~$\bar{x}$ to~$\Delta^q$, send~$\bar{x}'$ to the point~$\bar{x}'' \in \Delta_v^{q-1}$ where the ray arising from~$v$ and passing through~$\bar{x}'$ meets~$\Delta_v^{q-1}$, and map~$\bar{x}$ to the point~$y' \in h_t(\delta_0^r)$ where the ray arising from~$\bar{x}''$ and passing through~$\bar{x}$ intersects~$h_t(\delta_0^r)$.
The map taking~$\bar{x}$ to~$y'$ extends by continuity into the identity map on~$h_t(\delta_0^r)$.
Finally, take the image~$y \in \delta_0^r$ of~$y'$ under the inverse map $\chi_t^{-1}: h_t(\delta_0^r) \to \delta_0^r$.
The resulting map 
$
\bar{\varrho}_t: \Delta_X^p \setminus \pi^{-1}(\Delta_v^{q-1}) \to \delta_0^r
$
sending~$x$ to~$y$ is a retraction onto~$\delta_0^r$.

\medskip

\forget
Consider the standard Euclidean $q$-simplex~$\Delta^q \subseteq \R^m$.
Denote by $v_0,\dots,v_q$ the vertices of~$\Delta^q$.
Fix some nonnegative integers $k_0,\dots,k_q$ with $q + \sum_{i=0}^q k_i = m$.
We refer to the~$k_i$ as \emph{dimensional weights}.
Fix $q+1$ pairwise orthogonal $k_i$-dimensional affine subspaces~$H_i \subseteq \R^m$ passing through~$v_i$ and orthogonal to the $q$-dimensional affine subspace~$H$ spanned by~$\Delta^q$.
We have the following orthogonal direct sum 
\[
\vec{H} \oplus \vec{H}_0 \oplus \cdots \oplus \vec{H}_q = \R^m.
\]
In each subspace~$H_i$, take a regular Euclidean $k_i$-simplex~$\delta_i^{k_i}$ with vertices $w_0^i,\dots,w_{k_i}^i$ centered at~$v_i$.
The $m+1$ vertices $w_0^i,\dots,w_{k_i}^i$ with $0 \leq i \leq q$ are affinely independent.
Their convex hull forms an $m$-simplex~$\Delta^m_1$ of~$\R^m$, that is, 
\[
\Delta^m_1 = \Conv(\Delta^q \cup (\cup_{i=0}^q \delta_i^{k_i})) \subseteq \R^m.
\]
More generally, we can consider the regular Euclidean $k_i$-simplices $t \delta_i^{k_i}$ of size~$t$ obtained as the image of~$\delta_i^{k_i}$ by the $t$-homothety centered at~$v_i$ with $t \in (0,1]$ and define 
\begin{equation} \label{eq:Dt}
\Delta^m_t = \Conv(\Delta^q \cup (\cup_{i=0}^q \, t \delta_i^{k_i})) \subseteq \Delta^m_1.
\end{equation}
As $t$ goes to zero, the simplex~$\Delta^m_t$ collapses onto~$\Delta^q$; see Figure~\ref{fig:collapsing} below.

\begin{figure}[htbp!]
\vspace{1.3cm}
\hspace{-4cm}
\def\svgwidth{4cm}
\input{collapsing.pdf_tex}
\caption{Collapsing~$\Delta_t^m$.} \label{fig:collapsing} 
\end{figure}

Fix a vertex~$v$ of the standard Euclidean $q$-simplex~$\Delta^q \subseteq \R^m$.
Say, $v=v_0$.
Let $\Delta_v^{q-1}$ be the $(q-1)$-face of~$\Delta^q$ opposite to~$v$.
Construct a retraction 
\[
\bar{\varrho}:\Delta_1^m \setminus \Delta_v^{q-1} \to \delta_0^k
\]
onto~$\delta_0^k$ with $k=k_0$ as follows.
First, take the orthogonal projection onto~$H \oplus H_0$ (along $\vec{H}_1 \oplus \dots \oplus \vec{H}_q$).
Note that the projection of~$\Delta_1^m$ agrees with $\Conv(\Delta^q \cup \delta_0^k)$.
Thus, every point $x \in \Delta_1^m \setminus \Delta_v^{q-1}$ is sent to a point~$\bar{x} \in \Conv(\Delta^q \cup \delta_v^{q-1})$.
Then, for every~$\bar{x} \in \Conv(\Delta^q \cup \delta_0^k) \setminus \Delta_v^{q-1}$ not lying in~$\delta_0^k$, take the orthogonal projection~$\bar{x}' \in \Delta^q$ of~$\bar{x}$ to~$\Delta^q$, send~$\bar{x}'$ to the point~$\bar{x}'' \in \Delta_v^{q-1}$ where the ray arising from~$v$ and passing through~$\bar{x}'$ meets~$\Delta_v^{q-1}$, and map~$\bar{x}$ to the point~$y \in \delta_0^k$ where the ray arising from~$\bar{x}''$ and passing through~$\bar{x}$ intersects~$\delta_0^k$.
The map taking~$\bar{x}$ to~$y$ extends by continuity into the identity map on~$\delta_0^k$.
The resulting map 
\[
\bar{\varrho}: \Delta_1^m \setminus \Delta_v^{q-1} \to \delta_0^k
\]
sending~$x$ to~$y$ is a retraction onto~$\delta_0^k$.
Furthermore, its restriction
\[
\bar{\varrho}_t: \Delta_t^m \setminus \Delta_v^{q-1} \to t \delta_0^k
\]
is surjective onto~$t \delta_0^k$.
\forgotten

The map $\bar{\varrho}_t: \Delta_X^p \setminus \pi^{-1}(\Delta_v^{q-1}) \to \delta_0^r$ is not Lipschitz as the Lipschitz constant at a point goes to infinity when the point moves to~$ \Delta_X^p \cap \pi^{-1}(\Delta_v^{q-1})$.
For the map to be Lipschitz, we need to restrict it to a domain away from~$\pi^{-1}(\Delta_v^{q-1}) \cap \Delta_X^p$.
In order to use the map as a building block to construct further maps on simplicial complexes, we also need to take domains that are coherent in terms of face inclusion.
Extend~$\Delta^q$ into a regular Euclidean $m$-simplex~$\Delta^m \subseteq H$, where $\Delta^q$ is a face of~$\Delta^m$.
The perpendicular bisector hyperplane of the segment joining the barycenters of~$\Delta^m$ and~$\Delta_v^m$ intersects~$\Delta^q$ along a subspace~$\mathcal{H}$.
Let $\tau_{q,m} = d(v,\mathcal{H})$ be the distance from~$v$ to~$\mathcal{H}$ in~$\Delta^q$.
Observe that the sequence~$\tau_{q,m}$ is decreasing in~$q$. 
In particular, 
\[
\tau_{q,m} \geq \tau_m:=\tau_{m,m}.
\]
Note also that $\tau_m \geq \frac{1}{2}$.
See Figure~\ref{fig:construction} below.


\begin{figure}[htbp!]
\vspace{0.5cm}
\hspace{0cm}
\def\svgwidth{5cm}
\begingroup%
  \makeatletter%
  \providecommand\color[2][]{%
    \errmessage{(Inkscape) Color is used for the text in Inkscape, but the package 'color.sty' is not loaded}%
    \renewcommand\color[2][]{}%
  }%
  \providecommand\transparent[1]{%
    \errmessage{(Inkscape) Transparency is used (non-zero) for the text in Inkscape, but the package 'transparent.sty' is not loaded}%
    \renewcommand\transparent[1]{}%
  }%
  \providecommand\rotatebox[2]{#2}%
  \newcommand*\fsize{\dimexpr\f@size pt\relax}%
  \newcommand*\lineheight[1]{\fontsize{\fsize}{#1\fsize}\selectfont}%
  \ifx\svgwidth\undefined%
    \setlength{\unitlength}{181.99146661bp}%
    \ifx\svgscale\undefined%
      \relax%
    \else%
      \setlength{\unitlength}{\unitlength * \real{\svgscale}}%
    \fi%
  \else%
    \setlength{\unitlength}{\svgwidth}%
  \fi%
  \global\let\svgwidth\undefined%
  \global\let\svgscale\undefined%
  \makeatother%
  \begin{picture}(1,0.66056304)%
    \lineheight{1}%
    \setlength\tabcolsep{0pt}%
    \put(0.71883093,0.40916082){\color[rgb]{0,0,0}\makebox(0,0)[lt]{\lineheight{1.25}\smash{\begin{tabular}[t]{l}$\tau_{q,m}$\end{tabular}}}}%
    \put(0.55840776,0.04282638){\color[rgb]{0,0,0}\makebox(0,0)[lt]{\lineheight{1.25}\smash{\begin{tabular}[t]{l}$\Delta^q$\end{tabular}}}}%
    \put(0.0612345,0.04618187){\color[rgb]{0,0,0}\makebox(0,0)[lt]{\lineheight{1.25}\smash{\begin{tabular}[t]{l}$\mathcal{H}$\end{tabular}}}}%
    \put(0.25170423,0.53881148){\color[rgb]{0,0,0}\makebox(0,0)[lt]{\lineheight{1.25}\smash{\begin{tabular}[t]{l}$\Delta^{m}$\end{tabular}}}}%
    \put(0,0){\includegraphics[width=\unitlength,page=1]{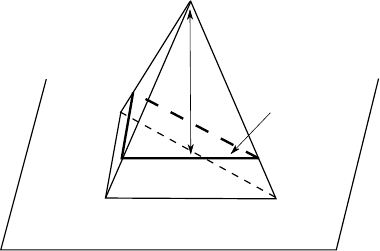}}%
    \put(1.08611307,0.37218462){\color[rgb]{0,0,0}\makebox(0,0)[lt]{\lineheight{1.25}\smash{\begin{tabular}[t]{l}$\mathcal{H}'$\end{tabular}}}}%
    \put(0.48556098,0.71402671){\color[rgb]{0,0,0}\makebox(0,0)[lt]{\lineheight{1.25}\smash{\begin{tabular}[t]{l}$v$\end{tabular}}}}%
  \end{picture}%
\endgroup%

\caption{Construction of~$\mathcal{H}$.} \label{fig:construction} 
\end{figure}

Consider the domain~$\Delta^q(v)$ of~$\Delta^q$ containing~$v$ delimited by~$\mathcal{H}$.
The restriction
\[
\varrho_t: \pi^{-1}(\Delta^q(v)) \cap \Delta_X^p \to \delta_0^r
\]
of~$\bar{\varrho}_t$ is $\kappa_m$-Lipschitz for some constant~$\kappa_m \geq 1$ depending only on~$m$. 
Note that this construction is coherent. 
That is, if $\Delta_P$ and~$\Delta'_P$ are two simplices of~$P$ containing~$v$, and $\Delta_X$ and~$\Delta'_X$ are two simplices of~$X$ mapped onto~$\Delta_P$ and~$\Delta'_P$ under $\pi:X \to P$, then the retractions~$\varrho_t$ defined on $\pi^{-1}(\Delta_P(v)) \cap \Delta_X$ and $\pi^{-1}(\Delta'_P(v)) \cap \Delta'_X$ coincide with the intersection of their domains of definition.
This will allow us to put together the retractions~$\varrho_t$.

Given a point~$z$ of~$\Delta_X^p$ lying in~$\pi^{-1}(\mathcal{H})$, let $z_-$ be a vertex of~$\Delta_X^p$ lying in~$\delta_0^r$ and $z_+$ be a vertex of~$\Delta_X^p$ not lying in~$\delta_0^r$ at minimal $g_t$-distance from~$z$.
Recall that $\Delta_X^p$ collapses onto~$\Delta_P^q$ in~$E$ as $t$ goes to zero.
By our choice of~$\mathcal{H}$, there exist $\varepsilon_m, \sigma_m \in (0,1)$ depending only on~$m$ such that 
\[
d_{g_t}(z,z_+) \leq d_{g_t}(z,z_-) - \varepsilon_m
\]
and
\[
d_{g_t}(z,z_+) + \sigma_m \leq \tau_m.
\]
We can further assume that $\varepsilon_m \leq \tau_m$.

Now, define
\begin{equation} \label{eq:Pv}
P_v = \cup \Delta_P^q(v) \subseteq P
\end{equation}
as the union over all the closed domains~$\Delta_P^q(v) \subseteq \Delta_P^q$, where $\Delta_P^q$ is a simplex of~$P$ of any dimension~$q$ containing~$v$.
Denote also 
\begin{equation} \label{eq:Xv}
X_v = \pi^{-1}(P_v) \subseteq X.
\end{equation}
By construction, the subset $X_v \subseteq X$ is a closed neighborhood of~$\pi^{-1}(v)$, lying in the (open) star of~$\pi^{-1}(v)$ and containing all the points of~$X$ at $g_t$-distance at most~$\tau_m$ from~$\pi^{-1}(v)$.

Putting together the retractions $\varrho_t : \pi^{-1}(\Delta^q(v)) \cap \Delta_X^p \to \delta_0^r$ where $\Delta_X^p$ is a simplex of~$X_v$ projecting to a simplex~$\Delta_P^q$ of~$P$ containing~$v$ and $\delta_0^r= \pi^{-1}(v) \cap \Delta_X^p$, we obtain a $\kappa_m$-Lipschitz retraction of~$X_v$ onto~$\pi^{-1}(v)$, still denoted by
\[
\varrho_t:X_v \to \pi^{-1}(v).
\]
\end{proof}

\forget
\medskip

Consider the domain~$\HH$ defined as the half-space of~$\R^m$ containing~$v$ delimited by the hyperplane containing~$\mathcal{H}$ orthogonal to~$H$.
Note that the barycenter of~$\Delta_1^m$ (and so of~$\Delta_t^m$) lies in the interior of~$\HH$.
The restriction
\[
\varrho: \Delta_1^m \cap \HH \to \delta_0^k
\]
of~$\bar{\varrho}$ is $\kappa_m$-Lipschitz for some constant~$\kappa_m \geq 1$ depending only on~$m$ (which can be expressed in terms of~$\tau_{q,m}$).
And so is the restriction of~$\bar{\varrho}_t$ denoted by
\[
\varrho_t: \Delta_t^m \cap \HH \to t \delta_0^k.
\]
We will also need to consider the region~$\Delta^q(v)$ of~$\Delta^q$ containing~$v$ delimited by~$\mathcal{H}$.
That is, 
\begin{equation} \label{eq:Dv}
\Delta^q(v) = \Delta^q \cap \HH.
\end{equation}

Consider a simplicial map $\pi:X \to P$ to a finite simplicial $(m-1)$-complex~$P$ as in the fiber collapsing assumption.
By Proposition~\ref{prop:connected}, we can assume that the simplicial map $\pi:X \to P$ is onto and that its fibers~$F_p$ are connected.

We want to construct a piecewise flat metric~$g_t$ on~$X$ for $t \in (0,1]$ as follows.
Every $p$-simplex~$\Delta_X^p$ of~$X$ maps onto a $q$-simplex~$\Delta_P^q$ of~$P$ under~$\pi$.
We identify~$\Delta_P^q$ with the standard Euclidean $q$-simplex $\Delta^q \subseteq \R^m$ with $v_0,\dots,v_q$.
Let $k_i +1$ be the number of vertices of~$\Delta_X^p$ sent to~$v_i$.
We identify~$\Delta_X^p$ with the Euclidean $p$-simplex~$\Delta_t^p$ defined in~\eqref{eq:Dt} where the dimensional weights are equal to~$k_i$.
Note that this construction is coherent. 
That is, if $\Delta_X'$ is a face of~$\Delta_X$, then the Euclidean metric defined on~$\Delta_X'$ agrees with the restriction of the Euclidean metric on~$\Delta_X$  to~$\Delta_X'$.
This allows us to define a piecewise flat metric~$g_t$ on the whole simplicial complex~$X$.
\forgotten

\subsection{Deforming arcs into edge-arcs} \label{sec:deforming}

\mbox { }

\medskip

Considering the family of piecewise flat metrics~$g_t$ on~$X$ defined in~\eqref{eq:gt}, we show the following result about the deformation of arcs of~$X$ into its $1$-skeleton.
This result will allow us to apply combinatorial techniques to count homotopy classes in Section~\ref{sec:edge-loop}.

\begin{proposition} \label{prop:deforming}
Let $X$ be a connected finite simplicial $m$-complex endowed with the piecewise flat metric~$g_t$ defined in~\eqref{eq:gt}.
Then, every arc~$\gamma$ of~$X$ joining two vertices can be deformed into an arc~$\gamma_e$ lying in the $1$-skeleton of~$X$ (\ie, $\gamma_e$ is an edge-arc), while keeping its endpoints fixed, with
\begin{equation} \label{eq:deforming}
\length_{g_t}(\gamma_e) \leq C_m \, \length_{g_t}(\gamma)
\end{equation}
for every $t \in (0,1]$, where $C_m$ is a positive constant depending only on~$m$.
\end{proposition}

\begin{proof}
Let us start with a simple observation.
Every arc of a regular Euclidean simplex~$\Delta^m$ with endpoints on~$\partial \Delta^m$ can be deformed into an arc of~$\partial \Delta^m$ with the same endpoints at the cost of multiplying its length by a factor bounded by a constant~$\lambda_m$ depending only on~$m$.
Applying this observation successively on the simplices of the skeleta of~$X$, we deduce by induction that the inequality~\eqref{eq:deforming} holds with $C_m = \lambda'_m$ for~$t=1$, where $\lambda'_m = \prod_{i=2}^m \lambda_i$, and, more generally, when every simplex of~$X$ is isometric to a regular Euclidean simplex of the same size.

\medskip

Endow~$P$ with its natural piecewise flat metric where all simplices are isometric to the standard Euclidean simplex of the same dimension.
Denote by~$v$ the image of the starting point of~$\gamma$ by $\pi:X \to P$.
Note that $v$ is a vertex of~$P$.
Consider the domains~$P_v$ and~$X_v$ introduced in~\eqref{eq:Pv} and~\eqref{eq:Xv}.
For every $q$-simplex $\Delta^q \subseteq P_v$ containing~$v$, the distance between~$v$ and its opposite side in~$\Delta^q(v)$ is at least~$\tau_m$.
Since the map $\pi:X_v \to P_v$ is $1$-Lipschitz, we deduce that if $\gamma$ leaves~$X_v$ then its $g_t$-length is greater than~$\tau_m$.

Let us argue by induction on the integer $n \geq 0$ such that
\[
n \varepsilon_m \leq \length_{g_t}(\gamma) < (n+1) \varepsilon_m
\]
where $\varepsilon_m$ is given by Lemma~\ref{lem:Xv}.
The value of~$C_m$ in~\eqref{eq:deforming} can be taken to be equal to 
$
C_m = 12 \tfrac{\lambda'_m \kappa_m}{\sigma_m}
$,
where $\kappa_m$ and~$\sigma_m$ are given by Lemma~\ref{lem:Xv}, and $\lambda'_m$ is defined above.

Suppose that $\gamma$ lies in~$X_v$.
(This is the case for instance if $\length_{g_t}(\gamma) < \tau_m$ and in particular if $n=0$.)
The image~$\gamma'$ of~$\gamma$ under the $\kappa_m$-Lipschitz retraction $\varrho_t:X_v \to \pi^{-1}(v)$ satisfies 
\[
\length_{g_t}(\gamma') \leq \kappa_m \, \length_{g_t}(\gamma).
\]
By construction, the fiber~$\pi^{-1}(v)$ is a simplicial complex of dimension at most~$m$ composed of regular Euclidean simplices of size~$t$.
As observed at the beginning of the proof, the arc~$\gamma'$ lying in~$\pi^{-1}(v)$ can be deformed into an arc~$\gamma_e$ lying in the $1$-skeleton of~$\pi^{-1}(v)$, and so of~$X$, with the same endpoints multiplying its length by a factor bounded by at most~$\lambda'_m$.
This concludes the proof of the proposition in this case with $C_m=\kappa_m \lambda'_m$.

Suppose that $\gamma$ leaves~$X_v$.
Denote by~$z$ the first point where $\gamma$ leaves~$X_v$.
The point~$z$ splits~$\gamma$ into two subarcs, $\gamma'$ and~$\gamma''$, with~$\gamma' \subseteq X_v$.
Let $\Delta_X$ be the smallest simplex of~$X$ containing $v$ and~$z$.
Pick a vertex~$z_-$ of~$\Delta_X$ lying in~$\pi^{-1}(v)$ and a vertex~$z_+$ of~$\Delta_X$ not lying in~$\pi^{-1}(v)$ at minimal $g_t$-distance from~$z$.
By Lemma~\ref{lem:Xv}.\eqref{eq:z}, we have
\begin{equation} \label{eq:zzz}
d_{g_t}(z,z_+) \leq d_{g_t}(z,z_-) - \varepsilon_m \leq \length_{g_t}(\gamma') - \varepsilon_m.
\end{equation}
Since $z$ and~$z_{\pm}$ lie in the same simplex~$\Delta_X$, the arc~$\gamma$ is homotopic to~$\alpha'  \cup [z_-,z_+] \cup \alpha''$, where the two arcs
\[
\alpha' = \gamma' \cup [z,z_-] \quad \mbox { and } \quad \alpha'' = [z_+,z] \cup \gamma''
\]
start and end at vertices of~$X$.
As previously observed, we have $\length_{g_t}(\gamma') \geq \tau_m$.
Recall also that $\diam_{g_t}(\Delta_X) \leq \sqrt{2}$; see~\eqref{eq:diamgt}.
Thus,
\[
\length_{g_t}(\alpha') \leq \length_{g_t}(\gamma') + \sqrt{2} \leq \left(1+\tfrac{\sqrt{2}}{\tau_m} \right) \, \length_{g_t}(\gamma')
\]
for $t \in (0,1]$.
The arc~$\alpha'$ lies in~$X_v$ and is sent to an arc of~$\pi^{-1}(v)$ with the same endpoints under the $\kappa_m$-Lipschitz retraction~$\varrho_t:X_v \to \pi^{-1}(v)$.
In turn, this arc can be deformed into an arc~$\alpha_e'$ lying in the $1$-skeleton of~$X$ with the same endpoints 
with
\begin{align}
\length_{g_t}(\alpha_e') & \leq \lambda'_m \kappa_m \, \length_{g_t}(\alpha') \nonumber \\
 & \leq \lambda'_m \kappa_m \left( 1 + \tfrac{\sqrt{2}}{\tau_m} \right) \, \length_{g_t}(\gamma'). \label{eq:a'}
\end{align}
Now, by~\eqref{eq:zzz}, we have
\begin{align*}
\length_{g_t}(\alpha'') & \leq \length_{g_t}(\gamma'') + d_{g_t}(z,z_+) \\
 & \leq \length_{g_t}(\gamma) - \varepsilon_m.
\end{align*}
By induction, the arc~$\alpha''$ can be deformed into an edge-arc~$\alpha''_e$ with the same endpoints with
\begin{align}
\length_{g_t}(\alpha''_e) & \leq C_m \, \length_{g_t}(\alpha'') \nonumber \\
 & \leq C_m \, \length_{g_t}(\gamma'') + C_m \, d_{g_t}(z,z_+). \label{eq:a''}
\end{align}
As a result of~\eqref{eq:a'} and~\eqref{eq:a''}, the arc~$\gamma$ can be deformed into the edge-arc $\gamma_e = \alpha_e' \cup [z_-,z_+] \cup \alpha''_e$, where
\[
\length_{g_t}(\gamma_e) \leq \lambda'_m \kappa_m \left( 1 + \tfrac{\sqrt{2}}{\tau_m} \right) \, \length_{g_t}(\gamma') + \sqrt{2} + C_m \, \length_{g_t}(\gamma'') + C_m \, d_{g_t}(z,z_+).
\]
In order to have $\length_{g_t}(\gamma_e) \leq C_m \, \length_{g_t}(\gamma)$, it is enough to have
\[
A_m \, \length_{g_t}(\gamma') + \sqrt{2} + C_m \, d_{g_t}(z,z_+) \leq C_m \, \length_{g_t}(\gamma')
\]
where $A_m = \lambda'_m \kappa_m \left( 1 + \tfrac{\sqrt{2}}{\tau_m} \right) \leq 4 \lambda'_m \kappa_m$ (recall that $\tau_m \geq \frac{1}{2}$).
That is,
\[
\frac{C_m \, d(z,z_+) + \sqrt{2}}{C_m - A_m} \leq  \length_{g_t}(\gamma').
\]
Recall that $d_{g_t}(z,z_+) + \sigma_m \leq \tau_m$; see~Lemma~\ref{lem:Xv}.\eqref{eq:zz}.
Thus, for $C_m$ large enough (\eg, $C_m \geq 12 \frac{\lambda'_m \kappa_m}{\sigma_m} \geq \frac{(1+ \sqrt{2} +\sigma_m) A_m}{\sigma_m}$), we have
\[
\frac{C_m \, d_{g_t}(z,z_+) + \sqrt{2}}{C_m - A_m} \leq d_{g_t}(z,z_+) + \sigma_m \leq \tau_m \leq  \length_{g_t}(\gamma')
\]
as desired.
\end{proof}

\forget
\begin{proof}[Proof of Theorem~\ref{theo:A}]
Consider a simplicial map $\pi:X \to P$ to a finite simplicial $(m-1)$-complex~$P$ as in the fiber collapsing assumption.
By Proposition~\ref{prop:connected}, we can assume that the simplicial map $\pi:X \to P$ is onto and that its fibers~$F_p$ are connected.

Let us introduce a couple of definitions.
An edge of~$X$ is said to be \emph{long} if it is sent to an edge of~$P$ by the simplicial map~$\pi:X \to P$.
It is said to be \emph{short} otherwise (in which case, it is sent to a vertex of~$P$).
Denote also by~$n_e$ the number of edges of~$X$.

We want to construct a piecewise flat metric~$g_t$ on~$X$ for $t \in (0,1]$ as follows.
Every $p$-simplex~$\Delta_X^p$ of~$X$ maps to a $q$-simplex~$\Delta_P^q$ of~$P$ under~$\pi$.
Denote by $v_0,\cdots,v_q$ the vertices of the standard Euclidean $q$-simplex~$\Delta^q \subseteq \R^N$ with $N \gg m$.
Let $k_i+1$ be the number of vertices of~$\Delta_X^p$ sent to~$v_i$.
Fix $q+1$ pairwise orthogonal affine subspaces $H_i \subseteq \R^N$ of dimension~$k_i$ passing through~$v_i$ and orthogonal to the affine subspace~$H$ spanned by~$\Delta^q$.
In each subspace~$H_i$, take a regular Euclidean $k_i$-simplex of size~$t$ with vertices $x_0^i,\dots,x_{k_i}^i$ centered at~$v_i$.
Define the metric~$g_t$ on~$\Delta_X^p$ by identifying~$\Delta_X^p$ with the Euclidean $p$-simplex with vertices $\{x_0^i,\dots,x_{k_i}^i \mid 0 \leq i \leq q \}$.
Note that this construction is coherent. 
That is, if $\Delta_X'$ is a face of~$\Delta_X$, then the metric defined on~$\Delta_X'$ agrees with the restriction to~$\Delta_X'$ of the metric on~$\Delta_X$.
This allows us to define the metric~$g_t$ on~$X$.

By construction, every short edge of~$X$ is of length~$t$ and every long edge of~$X$ is of length~$\simeq 1$ when $t$ goes to zero.
Furthermore,
\begin{equation} \label{eq:volt}
\vol(X,g_t) = O(t)
\end{equation}
as $t$ goes to zero.

\begin{proposition}
Every arc~$\gamma$ of~$X$ joining two vertices can be deformed into an arc~$\gamma'$ lying in the $1$-skeleton of~$X$ while keeping the endpoints fixed with
\[
\length_{g_t}(\gamma') \leq C_m \, \length_{g_t}(\gamma')
\]
for every $t \in (0,1]$, where $C_m$ is a positive constant depending only on~$m$.
\end{proposition}

\begin{proof}
Fix a vertex~$v=v_0$ of the standard Euclidean $q$-simplex~$\Delta^q \subseteq \R^N$.
Consider the regular Euclidean $k$-simplex $\Delta_t^k \subseteq H_0$ of size~$t$ with vertices $x_0,\dots,x_k$ centered at~$v$ as above.

Construct a map $\Xi_v \to \Delta_t^k$ as follows.
First, take the orthogonal projection to~$H \oplus H_0$ (along $H_1 \oplus \dots \oplus H_q$).
Note that the projection of~$\Delta_X^p$ agrees with $\Delta_X \cap (H \oplus H_0) = {\rm Conv}(\Delta_t^k \cup \Delta^q)$.
Then for every $x \in {\rm Conv}(\Delta_t^k \cup \Delta^q)$ not lying in the $(q-1)$-face~$\Delta^{q-1}_v$ of~$\Delta^q$ opposite to~$v$, take the orthogonal projection~$x' \in \Delta_v^{q-1}$ of~$x$ to the subspace ${\rm Vect}(\Delta_v^{q-1})$ spanned by~$\Delta_v^{q-1}$ and consider the intersection point~$\bar{x} \in \Delta_t^k$ between~$\Delta_t^k$ and the line passing through~$x$ and~$x'$.
The resulting map ${\rm Conv}(\Delta_t^k \cup \Delta^q) \setminus \Delta^{q-1}_v \to \Delta_t^k$ sending~$x$ to~$\bar{x}$ is a piecewise linear retraction onto~$\Delta_t^k$.

\end{proof}
\forgotten




\subsection{Edge-loop entropy} \label{sec:edge-loop}

\mbox { }

\medskip

In this section, we introduce the edge-loop entropy -- a discrete substitute for the volume entropy -- and show that the growth of the edge-loop entropy of~$(X,g_t)$ is controlled as $t$ goes to zero.

\begin{definition} \label{def:entN}
Let $X$ be a connected finite simplicial complex with a piecewise Riemannian metric~$g$.
The volume entropy of~$(X,g)$, see~\eqref{eq:entropy.comp.space}, can also be defined as the exponential growth rate of the number of homotopy classes induced by loops of length at most~$T$.
Namely,
\begin{equation} \label{eq:entN}
\ent(X,g) = \lim_{T \to \infty} \frac{1}{T} \log \mathcal{N}(X,g; T)
\end{equation}
where $\mathcal{N}(X,g;T) = {\rm card} \{ [\gamma] \in \pi_1(X,\star) \mid \gamma \mbox{ loop of } g\mbox{-length at most } T \}$.
See~\cite[Lemma~2.3]{sab} for instance, for a proof of this classical result.

It will be convenient to consider a similar notion for edge-loops.
Define the edge-loop entropy of~$(X,g)$ as
\[
\ent_e(X,g) = \lim_{T \to \infty} \frac{1}{T} \log \mathcal{N}_e(X,g;T)
\]
where $\mathcal{N}_e(X,g;T) = {\rm card} \{ [\gamma] \in \pi_1(X,\star) \mid \gamma \mbox{ edge-loop of } g\mbox{-length at most } T \}$.
Clearly, one has $\ent_e(X,g) \leq \ent(X,g)$.

Let $A$ be a subcomplex of~$X$ with basepoint~$a$.
We also define
\[
\mathcal{N}(A \subseteq (X,g); T) = {\rm card} \! \left\{ [\gamma] \in \pi_1(X,a) \mid \gamma \subseteq A \mbox{ and } \length_g(\gamma) \leq T \right\}
\]
as the number of homotopy classes (in~$X$) of loops of~$A$ based at~$a$ of $g$-length at most~$T$.
\end{definition}

\forget
Clearly, $\ent_e(X,g) \leq \ent(X,g)$.
For the simplicial $m$-complex~$X$ satisfying the fiber collapsing assumption and the piecewise flat metrics~$g_t$ considered in the previous sections, a reverse inequality holds true.
Namely, by Proposition~\ref{prop:deforming}, there exists $C_m >0$ such that 
\begin{equation} \label{eq:entropies}
\ent(X,g_t) \leq C_m \, \ent_e(X,g_t)
\end{equation}
for every $t \in (0,1]$.

\medskip
\forgotten

The edge-loop entropy of~$(X,g_t)$ can be bounded as follows.

\begin{proposition} \label{prop:log}
Suppose that the subexponential growth rate of all the subgroups~$i_*[\pi_1(F_p)]$ of~$\pi_1(X)$ is at most~$\nu$, where $F_p=\pi^{-1}(p)$ is a (connected) fiber of $\pi:X \to P$ and $i:F_p \hookrightarrow X$ is the inclusion map.
Then
\begin{equation} \label{eq:log}
\eent_e(X,g_t) = O \left( \tfrac{1}{t^\nu} \right)
\end{equation}
as $t$ goes to zero.
\end{proposition}

\begin{proof}
Let us introduce a couple of definitions.
An edge of~$X$ is said to be \emph{long} if it is sent to an edge of~$P$ by the simplicial map~$\pi:X \to P$.
It is said to be \emph{short} otherwise (in which case, it is sent to a vertex of~$P$).
By construction, every long edge of~$X$ is of length~$\sqrt{1+t^2}$ and every short edge of~$X$ is of length~$t$.
Denote also by~$n_e$ the number of edges of~$X$.

Observe that $g_t = t^2 g_1$ on every (connected) fiber~$F_p=\pi^{-1}(p)$ of~$\pi:X \to P$.
Hence,
\[
\diam(F_p,g_t) = t \cdot \diam(F_p,g_1) \xrightarrow[t \to 0]{} 0.
\]
Thus, by taking $t$ small enough, we can assume that $\diam(F_p,g_t) < \frac{1}{2}$ for every vertex $p \in P$.

Let us estimate the number of homotopy classes of edge-loops in~$X$ of \mbox{$g_t$-length} at most~$T$.
Every edge-loop~$\gamma$ in~$X$ of $g_t$-length at most~$T$ decomposes as
\begin{equation} \label{eq:decomp}
\gamma = \alpha_1 \cup \beta_1 \cup \alpha_2 \cup \cdots \cup \beta_N
\end{equation}
where $\alpha_i$ is a long edge of~$X$ and $\beta_i$ is a possibly constant edge-path lying in a (connected) fiber $F_i=\pi^{-1}(p_i)$ of~$\pi:X \to P$ over a vertex~$p_i \in P$, which joins the terminal endpoint of~$\alpha_i$ to the initial endpoint of~$\alpha_{i+1}$.

Fix a basepoint $a_i \in F_i$.
Denote by~$\ell_i$ the $g_t$-length of~$\beta_i$.
Let~$\bar{\beta}_i$ be the loop of~$F_i$ based at~$a_i$ obtained by connecting the endpoints $x_i$ and~$y_i$ of~$\beta_i$ to the basepoint~$a_i$ along two paths of~$F_i$ of $g_t$-length at most $\diam(F_i,g_t) < \frac{1}{2}$.
The number $\mathcal{N}^e_{{x_i},{y_i}}(F_i \subseteq (X,g_t);\ell_i)$ of homotopy classes (relative to the endpoints) in~$X$ of edge-paths in~$F_i$ with endpoints $x_i$ and~$y_i$, and $g_t$-length at most~$\ell_i$ is bounded by the number of homotopy classes in~$X$ of loops in~$F_i$ based at~$a_i$ of $g_t$-length at most $\ell_i + 2 \, \diam(F_i,g_t)$.
Thus, 
\begin{align}
\mathcal{N}^e_{{x_i},{y_i}}(F_i \subseteq (X,g_t);\ell_i) & \leq \mathcal{N}(F_i \subseteq (X,g_t); \ell_i + 2 \, \diam(F_i,g_t)) \nonumber  \\
 & \leq \mathcal{N} \left( F_i \subseteq (X,g_1); \tfrac{\ell_i + 1}{t} \right) \label{eq:N1}
\end{align}
since $g_t = t^2 g_1$ on the fiber~$F_i$, where we refer to Definition~\ref{def:entN} for the definition of~$\mathcal{N}$.

By assumption, the subgroups $i_*[\pi_1(F_p)] \leqslant \pi_1(X)$ have a subexponential growth at most~$\nu$ and the same holds for $\mathcal{N}(F_p \subseteq (X,g_1);T)$; see~\cite{manning05}.
More specifically, there exists a function $Q(T)=A \exp(T^\nu)$ with $A>0$ such that
\begin{equation} \label{eq:Q}
\mathcal{N}(F_p \subseteq (X,g_1);T) \leq Q(T)
\end{equation}
for every vertex $p \in P$ and every $T \geq 0$.

It follows from~\eqref{eq:N1} and~\eqref{eq:Q} that the number of homotopy classes in~$X$ induced by the different possibilities for the edge-path~$\beta_i$ of length~$\ell_i$ is at most
\[
\mathcal{N}^e_{{x_i},{y_i}}(F_i \subseteq (X,g_t);\ell_i) \leq Q \left( \tfrac{\ell_i +1}{t} \right)
\]
where $\ell_i$ is the $g_t$-length of~$\beta_i$.

Now, there are at most $n_e$ possible choices for each long edge~$\alpha_i$.
(Recall that $n_e$ is the number of edges of~$X$.)
Hence, the number of homotopy classes of edge-loops in~$X$ of $g_t$-length at most~$T$ which decomposes as in~\eqref{eq:decomp} with $\beta_i$ of~$g_t$-length $\ell_i \leq \theta_i$, where $\theta_i = \left \lceil{\ell_i}\right \rceil$, is bounded by
\[
n_e^N \,  \prod_{i=1}^N Q \left( \tfrac{\theta_i +1}{t} \right).
\]
Since every edge~$\alpha_i$ is of $g_t$-length at least~$1$, we have $N \leq T$ and $\sum_{i=1}^N \ell_i \leq T-N$.
Since $\theta_i = \left \lceil{\ell_i}\right \rceil$, we also have $\sum_{i=1}^N \theta_i \leq T$.
Therefore, the number~$\mathcal{N}_e(X,g_t;T)$ of homotopy classes of edge-loops in~$X$ of $g_t$-length at most~$T$ is bounded by
\begin{equation} \label{eq:double}
\mathcal{N}_e(X,g_t;T) \leq \sum_{N \leq \left \lfloor{T}\right \rfloor} \sum_{(\theta_i)_N \leq \left \lfloor{T}\right \rfloor} n_e^N \, \prod_{i=1}^N Q \left( \tfrac{\theta_i +1}{t} \right)
\end{equation}
where the second sum is over all $N$-tuples $(\theta_1,\dots,\theta_N)$ of positive integers such that \mbox{$\sum_{i=1}^N \theta_i \leq \left \lfloor{T}\right \rfloor$}.

The double sum~\eqref{eq:double} has at most $T \, 2^T$ terms (the first sum has $\left \lfloor{T}\right \rfloor$ terms and the second sum has $2^{\left \lfloor{T}\right \rfloor-1}$ terms given by the distinct decomposition of the integer~$\left \lfloor{T}\right \rfloor$).
Consider the largest term of~\eqref{eq:double} attained by some $N \leq T$ and $(\theta_i)_N \leq T$.
We have
\begin{align}
\mathcal{N}_e(X,g_t;T) & \leq T \, 2^T n_e^T \, \prod_{i=1}^N Q \left( \tfrac{\theta_i +1}{t} \right) \label{eq:max} \\
 & \leq T \, 2^T n_e^T \, A^T \exp \left( \frac{1}{t^\nu} \sum_{i=1}^N (\theta_i +1)^\nu \right). \nonumber
\end{align}
Applying H\"older's inequality to the sum $\sum_{i=1}^N (\theta_i +1)^\nu$ with $p=\frac{1}{1-\nu}$ and $q=\frac{1}{\nu}$, we obtain
\[
\sum_{i=1}^N (\theta_i +1)^\nu \leq \left( \sum_{i=1}^N 1^p \right)^{1/p} \cdot \left( \sum_{i=1}^N (\theta_i +1) \right)^{1/q} \leq T^{1-\nu} \cdot 2^\nu T^\nu \leq 2T
\]
since $\nu q =1$, $N \leq T$ and $\sum_{i=1}^N (\theta_i +1) \leq \sum_{i=1}^N \theta_i +N \leq 2T$.
Hence,
\[
\mathcal{N}_e(X,g_t;T) \leq T \, 2^T \, n_e^T \, A^T \, \exp \left( \frac{2T}{t^\nu} \right).
\]
This implies that
\[
\eent_e(X,g_t) \leq \log(2n_eA) + \frac{2}{t^\nu}.
\]
\forget
Let $\theta_i' = \theta_i +1$.
Note that $\sum_{i=1}^N \theta_i' \leq T$.
Using the expression of~$Q$ and the concavity of the nondecreasing function $\log(1+\cdot)$, we obtain
\begin{align}
\log \left( \prod_{i=1}^N Q \left( \tfrac{\theta_i'}{t} \right) \right) 
 & \leq T \, \log(a) + d \, \sum_{i=1}^N \log \left( 1+\tfrac{\theta_i'}{t} \right) \nonumber \\
 & \leq T \, \log(a) + d \, N \, \log \left( 1+\tfrac{T}{Nt} \right). \label{eq:L1}
\end{align}
Introduce $f_t(x) = x \, \log(1+\frac{1}{xt})$ with $x \in [0,1]$.
For $t \leq \frac{1}{e-1}$, we have 
\[
f'_t(x) = \log(1+\tfrac{1}{xt}) - \tfrac{1}{xt+1} \geq \log(1+\tfrac{1}{t}) -1 \geq 0.
\]
Thus, for $t$ small enough, we deduce that 
\begin{equation} \label{eq:L2}
\tfrac{N}{T} \, \log \left( 1+\tfrac{T}{Nt} \right) = f_t(\tfrac{N}{T}) \leq f_t(1) = \log \left( 1+\tfrac{1}{t} \right).
\end{equation}
Taking the log in~\eqref{eq:Ne}, dividing by~$T$ and letting~$T$ go to infinity, we obtain from~\eqref{eq:L1} and~\eqref{eq:L2} that
\[
\eent_e(X,g_t) = O \left( \log \left( \tfrac{1}{t} \right) \right)
\]
as $t$ goes to zero.
\forgotten
\end{proof}

\begin{remark}
If $X$ satisfies the fiber collapsing assumption with polynomial growth rate, we can derive a stronger bound than~\eqref{eq:log}.
Namely, the edge-loop entropy of~$(X,g_t)$ has a logarithmic growth when $t$ goes to zero, that is,
\[
\eent_e(X,g_t) = O \left( \log \left( \tfrac{1}{t} \right) \right).
\]
The argument is similar to the proof of Proposition~\ref{prop:log} until the inequality~\eqref{eq:max}, except that $Q$ should be replaced by a polynomial of the form $Q(T)=a(T+1)^d$ with~$a>0$.
Now, using the expression of~$Q$, the concavity of the nondecreasing function $\log(1+\cdot)$, and the inequalities $N \leq T$ and $\sum_{i=1}^N (\theta_i+1) \leq 2T$, we obtain
\begin{align}
\log \left( \prod_{i=1}^N Q \left( \tfrac{\theta_i+1}{t} \right) \right) 
 & \leq T \, \log(a) + d \, \sum_{i=1}^N \log \left( 1+\tfrac{\theta_i+1}{t} \right) \nonumber \\
 & \leq T \, \log(a) + d \, N \, \log \left( 1+\tfrac{2T}{Nt} \right). \label{eq:L1}
\end{align}
Introduce $f_t(x) = x \, \log(1+\frac{1}{xt})$ with $x \in [0,1]$.
For $t \leq \frac{1}{e-1}$, we have 
\[
f'_t(x) = \log(1+\tfrac{1}{xt}) - \tfrac{1}{xt+1} \geq \log(1+\tfrac{1}{t}) -1 \geq 0.
\]
Thus, for $x=\frac{N}{2T}$ and $t$ small enough, we deduce that 
\begin{equation} \label{eq:L2}
\tfrac{1}{2} \cdot \tfrac{N}{T} \, \log \left( 1+\tfrac{2T}{Nt} \right) = f_t(\tfrac{N}{2T}) \leq f_t(1) = \log \left( 1+\tfrac{1}{t} \right).
\end{equation}
Taking the log in~\eqref{eq:max}, dividing by~$T$ and letting~$T$ go to infinity, we obtain from~\eqref{eq:L1} and~\eqref{eq:L2} that
\[
\eent_e(X,g_t) = O \left( \log \left( \tfrac{1}{t} \right) \right)
\]
as $t$ goes to zero.
\end{remark}

\subsection{Fiber collapsing assumption and zero minimal volume entropy}
\mbox { }

\medskip

We show the following result (stated in the introduction as Theorem~\ref{theo:omega=0}).

\begin{theorem} \label{theo:A}
Let~$X$ be a connected finite simplicial $m$-complex.
Suppose there exists a simplicial map $\pi:X \to P$ to a simplicial $k$-complex~$P$ with~$k<m$ such that for every connected component~$F_p$ of every fiber~$\pi^{-1}(p)$ with $p \in P$, the finitely generated subgroup~$i_*[\pi_1(F_p)]$ of~$\pi_1(X)$ has subexponential growth rate at most~$\nu$.
Suppose that $\nu < \frac{m-k}{m}$.
Then $X$ has zero minimal volume entropy.
\end{theorem}




\begin{proof}
By Proposition~\ref{prop:connected}, we can assume that the simplicial map $\pi:X \to P$ in Theorem~\ref{theo:A} is onto and that its fibers~$F_p$ are connected.
Consider the family of piecewise flat metrics~$g_t$ on~$X$ defined in Section~\ref{sec:g_t}.
Recall that $\ent_e(X,g_t) \leq \ent(X,g_t)$; see Definition~\ref{def:entN}.
By Proposition~\ref{prop:deforming}, a reverse inequality holds true.
Namely, there exists $C_m >0$ such that 
\begin{equation} \label{eq:entropies}
\ent(X,g_t) \leq C_m \, \ent_e(X,g_t)
\end{equation}
for every $t \in (0,1]$.
By~\eqref{eq:volt} and~\eqref{eq:log}, we deduce that
\[
\eent(X,g_t) \, \vol(X,g_t)^\frac{1}{m} = O \left( t^{\frac{m-k}{m} - \nu} \right).
\]
Since $\nu < \frac{m-k}{m}$, we conclude that $\eent(X,g_t) \, \vol(X,g_t)^\frac{1}{m}$ converges to zero as $t$ goes to zero.
\end{proof}

Combining Theorem~\ref{theo:A} and Proposition~\ref{prop:cover}, we immediately derive the following result, which can also be expressed in terms of covering collapsing assumption.

\begin{corollary}\label{coro:collapsing}
Every connected finite simplicial $m$-complex~$X$ which admits a covering of multiplicity~$k+1$ by open subsets of subexponential $\pi_1$-growth in~$X$ with subexponential growth rate at most $\nu < \frac{m-k}{m}$ has zero minimal volume entropy. 
\end{corollary}


We conclude with an application.
Let us recall the definition of an $F$-structure, first introduced by Cheeger-Gromov in a different context; see~\cite{ChG} and \cite{ChGbis}.

\begin{definition} \label{def:F-structure}
A closed manifold~$M$ admits an \emph{$F$-structure} if there are a locally finite open covering~$\{U_i\}$ of~$M$, finite normal covers~$\pi_i:\tilde{U}_i \to U_i$ and effective smooth actions of tori~$\T^{k_i}$ on~$\tilde{U}_i$ which extend the action of the deck transformation group such that if $U_i$ and~$U_j$ intersect each other, then $\pi_i^{-1}(U_i \cap U_j)$ and $\pi_j^{-1}(U_i \cap U_j)$ have a common cover space on which the lifting actions of~$\T^{k_i}$ and~$\T^{k_j}$ commute.
We also assume that some orbits have positive dimension.
See~\cite{ChG} or \cite{ChGbis} for a more precise definition.
The \emph{rank} of an $F$-structure is the minimal dimension of the orbits.
\end{definition}

\forget
\begin{definition} \label{def:F-structure}
A closed manifold~$M$ admits an \emph{$F$-structure} if the following conditions are satisfied:
\begin{itemize}
\item there is a locally finite open covering~$\{U_i\}$ of~$M$ with finite normal covers~$\pi_i:\tilde{U}_i \to U_i$;
\item there is an effective smooth action of a torus~$\T^{k_i}$ on~$\tilde{U}_i$ which extends the deck transformations;
\item if $U_i$ and~$U_j$ intersect each other, then $\pi_i^{-1}(U_i \cap U_j)$ and $\pi_j^{-1}(U_i \cap U_j)$ have a common cover space on which the lifting actions of~$\T^{k_i}$ and~$\T^{k_j}$ commute.
\end{itemize}
\end{definition}
\forgotten

As an application of Corollary~\ref{coro:polynomial}, we derive the following result, which is also a consequence of Paternain and Petean's work on the connection between the topological entropy of the geodesic flow and $F$-structures; see~\cite[Theorem~A]{PP03}.

\begin{corollary}
Every closed manifold admitting an $F$-structure (of possibly zero rank) has zero minimal volume entropy.
\end{corollary}

\begin{proof}
By the Slice Theorem and its consequences, see~\cite[Appendix~B]{GGK}, we derive the following properties.
The orbits of the $F$-structure of a closed $m$-manifold~$M$ partition the manifold into closed submanifolds covered by tori; see also~\cite{ChG} and~\cite{PP03}.
The trivial orbits form a submanifold of codimension at least one (at least two if the manifold is orientable) and the orbit space is an orbifold of dimension at most~$m-1$.
Now, since the fibers of the natural projection from~$M$ to the orbit space have virtually abelian fundamental groups (and virtually abelian groups have polynomial growth by~\cite{gro81}), the manifold~$M$ satisfies the fiber collapsing assumption with polynomial growth rate and has zero minimal volume entropy by Corollary~\ref{coro:polynomial}.
\end{proof}

\subsection{Examples of manifolds satisfying the fiber collapsing assumption} \label{sec:ex}

\mbox { }

\medskip

In this section, we construct a closed orientable manifold with fundamental group of exponential growth satisfying the fiber collapsing assumption with fibers of subexponential growth which do not have polynomial growth.
Furthermore, this example satisfies the condition on the subexponential growth rate of the subgroups~$i_*[\pi_1(F_p)]$ of Theorem~\ref{theo:A} (which implies that their minimal volume entropy is zero).

\medskip

The first Grigorchuk group~$G$ was defined in~\cite{grigorchuk1}.
It is the first example of a finitely generated group of intermediate growth, that is, its growth is subexponential but not polynomial; see~\cite{grigorchuk} and~\cite{grigorchuk2}.
The exact value of the subexponential growth rate of~$G$ has recently been computed in~\cite{EZ20}.
It is roughly equal to
\[
\nu(G) \simeq 0.7674 \in [\tfrac{3}{4},\tfrac{4}{5}].
\]
The group~$G$ is a finitely generated recursively presented group -- a description of its presentation can be found in~\cite{lys} -- but it is not finitely presented.
It is an open question whether finitely presented groups of intermediate growth exist.
By Higman's embedding theorem~\cite{hig61}, the group~$G$ can be embedded into a finitely presented group.
A concrete realization of such an embedding is given in~\cite[Theorem~1]{grigorchuk2}.
The construction goes as follows.

Consider the group~$\bar{G}$ given by the following presentation:
\begin{multline} \label{eq:presGbar}
\bar{G} = \langle a,c,d,u \mid a^2 = c^2 = d^2 = (ad)^4 = (adacac)^4 = e ;\\
 u^{-1} a u = aca, u^{-1}c u=dc, u^{-1}d u = c \rangle.
\end{multline}
The group~$\bar{G}$ contains the first Grigorchuk group~$G$. 
More precisely, the group~$\bar{G}$ is an HNN-extension of~$G$:
\[
\bar{G} = \langle G,u \mid u^{-1} x u = \sigma(x) \mbox{ for every } x \in G \rangle.
\]
where $\sigma: G \to G$ is a monomorphism. 
The subgroup~$G \leqslant \bar{G}$ is generated by~$a$, $c$ and~$d$.
Note that $\bar{G}$ contains a free subsemigroup with two generators, and therefore has exponential growth.

\medskip

Let us construct an orientable closed $5$-dimensional manifold~$M$ with $\pi_1(M)=\bar{G}$ as follows.
Define
\begin{equation} \label{eq:connectedsum}
N = (\R P^5)_a \# (\R P^5)_c \# (\R P^5)_d \# (S^1 \times S^4)_u
\end{equation}
where the indices $a$, $c$, $d$ and~$u$ correspond to the generators of~$G$.
Note that $\R P^5$ is orientable and so is~$N$.
Take five loops~$\gamma_1,\dots,\gamma_5$ in the homotopy classes $(ad)^4$, $(adacac)^4$, $u^{-1} a u aca$, $u^{-1}c udc$ and $u^{-1}d uc$ of $\pi_1(N) = \Z_2 * \Z_2 * \Z_2 * \Z$.
Since $N$ is orientable, the normal bundles of~$\gamma_1,\dots,\gamma_5$ are trivial.
Placing the curves in generic position, we can assume that the loops $\gamma_1,\dots,\gamma_5$ are smooth simple closed curves which do not intersect each other.
Denote by~$M$ the orientable closed manifold obtained from~$N$ by spherical surgeries of type~$(1,4)$ along~$\gamma_1,\dots,\gamma_5$.
See~\cite{Milnor61} for an account on spherical surgeries.
Since spherical surgeries of type~$(1,4)$ correspond to attaching index~$2$ handles, the fundamental group of~$M$ is given by the presentation~\eqref{eq:presGbar}.
That is, $\pi_1(M) = \bar{G}$.

Let us construct a piecewise linear map $\pi:M \to S^1$ with subexponential growth fibers.
Consider the natural map $N \to S^1$ which takes the terms~$(\R P^5)_a \# (\R P^5)_c \# (\R P^5)_d$ in the connected sum~\eqref{eq:connectedsum} to a point $p_0 \in S^1$ and projects the last term~$(S^1 \times S^4)_u$ to the $S^1$-factor of the product.
By the expression of the relations of the presentation~\eqref{eq:presGbar} of~$\bar{G}$, the images by $N \to S^1$ of the loops $\gamma_1,\dots,\gamma_5$ are contractible in~$S^1$.
Thus, the map $N \to S^1$ extends to the handles of~$M$, which yields a map $M \to S^1$.
Deforming the map, if necessary, by sending the complement of a tubular neighborhood of a regular fiber~$F$ of~$M \to S^1$ to a point, we can assume that the map $M \to S^1$ is smooth with a unique critical value~$p_0 \in S^1$ and that the inverse image $\pi^{-1}(S^1 \setminus \{ p_0 \})$ has a product structure~$(0,1) \times F$ whose vertical slices coincide with the fibers of~$M \to S^1$.
We can further deform $M \to S^1$ into a piecewise linear map $\pi:M \to S^1$ by taking fine enough triangulations of~$M$ and~$S^1$, and by applying the simplical approximation theorem, without changing the topology of the fibers above~$S^1 \setminus \{ p_0\}$.

Let us show that $\ker \pi_* = G$, where $\pi_*:\pi_1(M) \to \pi_1(S^1)$ is the $\pi_1$-homomorphism induced by~$\pi:M \to S^1$.
Since the subgroup~$G \leqslant \bar{G}$ is generated by~$a$, $c$ and~$d$, the inclusion~$G \leqslant \ker \pi_*$ is obvious.
For the reverse inequality, observe that every element $w \in \ker \pi_*$ can be represented by a word in the letters $a$, $b$, $d$ and~$u$ with a minimal number of occurrences of~$u^{\pm 1}$.
By construction, $\pi_*(a) = \pi_*(c) = \pi_*(d) = 0$ and $\pi_*(u)$ is a generator of~$\pi_1(S^1)$.
Thus, the word~$w$ has as many~$u$'s as~$u^{-1}$'s.
If the word~$w$ contains a letter~$u$ or~$u^{-1}$, then it contains a subword~$uw'u^{-1}$ or~$u^{-1}w'u$, where $w'$ is a word in~$a$, $c$ and~$d$ (without~$u$).
According to the presentation~\eqref{eq:presGbar}, these subwords can be replaced with subwords in the letters~$a$, $b$, $d$ (without~$u$) in the representation of~$w$, which contradicts the choice of the word representing~$w$.
Thus, $w$ lies in the subgroup~$G$ of~$\bar{G}$ generated by $a$, $c$ and~$d$.
That is, $\ker \pi_* \leqslant G$.
Hence, $\ker \pi_* =G$.

Now, since $i_*[\pi_1(F_{p_0})]$ is a subgroup of~$\ker \pi_*$ containing the generators $a$, $c$ and~$d$ of~$G$, we derive that $i_*[\pi_1(F_{p_0})] = \ker \pi_* = G$.
All the other fibers~$F_p \simeq F$ with $p \in S^1$ different from~$p_0$ can be deformed into~$F_{p_0}$.
More precisely, there is a homotopy $h_t:F_p \to M$ starting at the inclusion map $i:F_p \hookrightarrow M$ and ending in~$F_{p_0}$ (\ie, $h_1:F_p \to F_{p_0}$).
This implies that $i_*[\pi_1(F_p)]$ is a subgroup of~$i_*[\pi_1(F_{p_0})] = G$.
Since $G$ has subexponential growth, the image~$i_*[\pi_1(F_p)]$ of the fundamental group of every fiber~$F_p$ of $\pi:M \to S^1$ has also subexponential growth, where~$p \in S^1$.

Since $\nu(G) < \frac{m-k}{m} = \frac{4}{5}$ (with $m=5$ and~$k=1$), the orientable closed $5$-dimensional manifold~$M$ satisfies the fiber collapsing assumption of Theorem~\ref{theo:A}.

\begin{remark}
This example shows that the effect of the collapsing can be due to fiber subgroups of intermediate growth (which are not finitely presented) and not merely of polynomial growth.
\end{remark}

\begin{remark}
Anticipating on the notion of amenable group, see Definition~\ref{def:amenable}, observe that the group~$\bar{G}$ is amenable; see~\cite{grigorchuk2}.
Therefore, by Gromov's vanishing simplicial volume theorem (see Theorem~\ref{theo:vanishing}), every manifold with fundamental group~$\bar{G}$ has zero simplicial volume.
\end{remark}

\begin{remark}
One can show that the manifold~$M$ is essential.
(This is not direct and requires some work.)
An easier way to obtain an essential manifold~$M'$ is to modify our construction by taking the connected sum of~$M$ with a nilmanifold, say~$\T^m$.
In this case, we collapse~$M' = \T^m \# M$ to the graph~$P=[0,1] \cup_{\{1\} = p_1} S^1$ so that the preimage of~$p_1 \neq p_0$ is the attaching sphere of the connected sum, the torus~$\T^m \setminus B^m$ with a ball removed is sent to~$[0,1]$ and the term~$M \setminus B^m$ is sent to~$S^1$ as before.
The manifold~$M'$ still satisfies the fiber collapsing assumption of Theorem~\ref{theo:A} with the map $\pi:M' \to P$, and the image~$i_*[\pi_1(F'_{p_0})]$ of the fundamental group of the fiber~$F'_{p_0}$ of $\pi:M' \to P$ still agrees with the group~$G$ of intermediate growth.
\end{remark}

\forget
Fix $m \geq 4$.
Let us construct a closed $m$-manifold~$N$ with $\pi_1(N)=K$ as follows.
Define
\[
N' = (\R P^m)_a \# (\R P^m)_c \# (\R P^m)_d
\]
where the indices $a$, $c$ and~$d$ correspond to the generators of~$G$.
Take two closed curves~$\gamma_1$ and~$\gamma_2$ in the homotopy classes $(ad)^4$ and $(adacac)^4$ of $\pi_1(N') = \Z_2 * \Z_2 * \Z_2$.
Placing the curves in generic position if necessary, we can assume that $\gamma_1$ and~$\gamma_2$ are smooth simple closed curves which do not intersect.
Denote by~$N$ the closed manifold obtained from~$N'$ by spherical surgeries of type~$(1,m-1)$ along~$\gamma_1$ and~$\gamma_2$.
Since spherical surgeries of type~$(1,m-1)$ correspond to attaching index~$2$ handles, the fundamental group of~$N$ is given by the presentation~\eqref{eq:presG}, \ie, $\pi_1(N) = K$.
Note that there is a natural embedding $X_K \hookrightarrow N$ of~$X_K$ into~$N$.

\medskip

Let us construct a closed $m$-manifold~$M$ with $\pi_1(N)=\bar{G}$ as follows.
Define
\begin{equation} \label{eq:connectedsum}
M' = N \# (S^1 \times S^{m-1})_u
\end{equation}
where the index~$u$ corresponds to a new generator of the fundamental group.
Denote by $\pi':M' \to S^1$ the natural map which takes the first term~$N$ of the connected sum~\eqref{eq:connectedsum} to a point $p_0 \in S^1$ and projects the second term~$S^1 \times S^{m-1}$ to the $S^1$-factor of the product.
The map~$\pi':M' \to S^1$ is smooth with a unique critical value~$p_0 \in S^1$.
In this construction, we perform the connected sum~\eqref{eq:connectedsum} away from the $2$-complex~$X_K \subseteq N$.
Thus, $X_K \subseteq M'$ and $\pi'(X_K)=p_0$.
Observe that $\pi'_*(a) = \pi'_*(c) = \pi'_*(d) = 0$ and $\pi'_*(u)$ is a generator of~$\pi_1(S^1)$, where $\pi'_*:\pi_1(M') \to \pi_1(S^1)$ is the $\pi_1$-homomorphism induced by~$\pi'$.
Take three closed curves $\gamma_3$, $\gamma_4$ and~$\gamma_5$ of~$M'$ in the homotopy classes $u^{-1} a u aca$ , $u^{-1}c u cd$ and $u^{-1}d u c$ of~$\pi_1(M')$.
As previously, we can assume that the curves are smooth and simple with no intersection.
We can also assume that they do not intersect the $2$-complex~$X_K \subseteq M'$.
Denote by~$M$ the manifold obtained from~$M'$ by spherical surgeries of type~$(1,m-1)$ along~$\gamma_3$, $\gamma_4$ and~$\gamma_5$.
By construction, the fundamental group of~$M$ is given by the presentation~\eqref{eq:presG}.
That is, $\pi_1(M)=\bar{G}$.

\medskip

Let us construct a piecewise linear map $\pi:M \to S^1$ with subexponential growth fibers.
By definition of the curves, the images~$\pi'(\gamma_i)$ with $i=3,4,5$ are contractible in~$M'$.
Therefore, the map $\pi':M' \to S^1$ extends to the handles attached to~$M'$ by surgery, which gives rise to a map $\pi:M \to S^1$.
Moving all the critical values of~$\pi:M \to S^1$ to~$p_0$, we can assume that the only critical value of~$\pi:M \to S^1$ is~$p_0$ and that $\pi^{-1}(p) \simeq S^{m-1}$ for every $p \in S^1$ different from~$p_0$.
We can further assume that $\pi:M \to S^1$ is piecewise linear by taking fine enough triangulations of~$M$ and~$S^1$, and by applying the simplical approximation theorem.
This does not change the subgroups~$\pi_1(F_p) \leqslant \pi_1(M)$, where $F_p = \pi^{-1}(p)$ and $p \in S^1$.
By construction, $\pi_*(a) = \pi_*(c) = \pi_*(d) = 0$ and $\pi_*(u)$ is a generator of~$\pi_1(S^1)$, where $\pi_*:\pi_1(M) \to \pi_1(S^1)$ is the homomorphism induced by~$\pi$.
Since $X_K \subseteq F_{p_0}$ and $i_*[\pi_1(F_{p_0})]$ is a subgroup of~$\ker \pi_*$, we derive that $i_*[\pi_1(F_{p_0})] = \ker \pi_* = G$.
For all the other fibers~$F_p$ with $p \in S^1$ different from~$p_0$, we have $i_*[\pi_1(F_p)] = 0$.
Thus, the fundamental groups of the fibers of $\pi:M \to P$ have subexponential growth.

\medskip

For $m=5$ (and $k=1)$, we have $\nu(G) < \frac{m-k}{m} = \frac{4}{5}$.
In this case, the manifold~$M$ is orientable and satisfies the fiber collapsing assumption of Theorem~\ref{theo:A}.

\medskip

When $m=4$ (and $k=1$), we have $\nu(G) > \frac{m-k}{m} = \frac{3}{4}$.
\forgotten

\forget
In dimension~$4$, the construction does not go through as the loop~$\gamma_4$ representing $u^{-1}cu dc$ is orientation-reversing.
Still, we can define~$N$ as the connected sum~$\#_{i=1}^4 (S^1 \times S^4)_i$.
Performing spherical surgeries along eight loops~$\gamma_1,\dots,\gamma_8$ representing the eight relations of the presentation~\eqref{eq:presGbar} of~$\bar{G}$ as previously, we obtain an orientable closed $4$-manifold~$M$ and a piecewise linear map~$\pi:M \to S^1$ with $i_*[\pi_1(F_{p_0})] = G$ and $i_*[\pi_1(F_p)] \leqslant G$ for every $p \in S^1$.
In this case ($m=4$ and $k=1$), we have $\nu(G) > \frac{m-k}{m} = \frac{3}{4}$.
Thus, the manifold~$M$ does not satisfy the fiber collapsing assumption of Theorem~\ref{theo:A} with the simplicial map $\pi:M \to S^1$.

We show in this case that the natural method presented in the proof of Theorem~\ref{theo:A} consisting of shrinking the fibers of the map $\pi:M \to S^1$ does not readily apply, even though the images of the fundamental groups of the fibers of~$\pi:M \to S^1$ have subexponential growth.
Endow~$M$ with a piecewise flat metric where all simplices of~$M$ are isometric to the standard Euclidean simplex with the same dimension.
Since the subexponential growth rate of \mbox{$i_*[\pi_1(F_{p_0})] = G$} is greater than~$\frac{3}{4}$, the number of homotopy classes of~$F_{p_0}$ which can be represented by a loop based at a vertex~$\star \in F_{p_0}$ of length at most~$T$ is at least $C \, \exp(T^{\frac{3}{4}})$ for some positive constant~$C$.
Fix a loop~$\beta$ of~$M$ based at~$\star$ representing~$u \in \pi_1(M) = \bar{G}$ and denote by~$\ell$ its length.
Modify the metric on~$M$ by shrinking the length of the fibers of~$\pi:M \to S^1$ by a factor~$t$ as in the proof of Theorem~\ref{theo:A} and denote by~$g_t$ the metric so-obtained.
Note that $\vol(M,g_t) \geq C' \, t$ for some positive constant~$C'$.
Observe also that the number of homotopy classes of the fiber~$F_{p_0}$ which can be represented by loops based at~$\star$ of $g_t$-length at most~$T$ is at least $C \, \exp \left( ( \frac{T}{t} )^{\frac{3}{4}} \right)$. 
Fix $T=n(\ell +1)$.
Consider all the loops $\gamma = \alpha_1 \beta \cdots \alpha_n \beta$ of~$M$ based at~$\star$ of $g_t$-length at most~$T$ which can be constructed by taking a loop~$\alpha_1$ in~$F_{p_0}$ of $g_t$-length at most~$1$, followed by the loop~$\beta$ of $g_t$-length~$\ell$, then by a loop~$\alpha_2$ in~$F_{p_0}$ of $g_t$-length at most~$1$, then by~$\beta$ and so on.
We obtain at least $C^n \, \exp\left( \frac{n}{t^{\frac{3}{4}}} \right) = C^{\frac{T}{\ell +1}} \, \exp \left( \frac{T}{(\ell +1) t^{\frac{3}{4}}} \right)$ homotopy classes.
Thus, $\ent(M,g_t) \gtrsim \frac{1}{t^{\frac{3}{4}}}$.
Hence, $\ent(M,g_t)^4 \, \vol(M,g_t) \gtrsim \frac{1}{t^2}$ which tends to infinity as $t$ goes to zero. 

\begin{remark}
Note that the orientable closed $4$-manifold~$M$ constructed above satisfies the fiber collapsing assumption of Theorem~\ref{theo:A} with~$k=2$.
Indeed, one can show that there exists a map $M \to K$ from~$M$ to a $2$-dimensional polyhedron~$K$ which induces a $\pi_1$-isomorphism whose fundamental groups of the fibers have a trivial image in~$\pi_1(M)$.
Therefore, the minimal volume entropy of~$M$ vanishes.
\end{remark}

???\\forget
\begin{remark}
The group~$\bar{G}$ is amenable; see~\cite{grigorchuk2}.
Thus, by Gromov's vanishing simplicial volume theorem, see Theorem~\ref{theo:vanishing}, every manifold with fundamental group~$\bar{G}$ has zero simplicial volume.
\end{remark}

\begin{remark}
It is not clear whether $M$ is essential.
This seems to be a delicate question.
Indeed, the Grigorchuk group~$G$ is a $2$-group.
Therefore, all its homology groups are torsion groups.
Since $\bar{G}$ is an HNN-extension of~$G$, the same holds for the homology groups of~$\bar{G}$ of dimension at least~$2$.
Studying the homology classes with coefficient in~$\Z_2$ of this group is even more complicated.
At any rate, we can modify our construction to obtain an essential manifold~$M'$ by taking the connected sum of~$M$ with a nilmanifold, say~$\T^m$.
In this case, we collapse~$M' = \T^m \# M$ to the graph~$P=[0,1] \cup_{\{1\} = p_1} S^1$ so that the preimage of~$p_1 \neq p_0$ is the attaching sphere of the connected sum, the torus~$\T^m$ is sent to~$[0,1]$ and the term~$M$ is sent to~$S^1$ as before.
\end{remark}

\begin{remark}
To ensure that our manifolds are essential, we can modify our construction by taking the connected sum of~$M$ with a nilmanifold, say~$\T^m$, which yields an essential manifold~$M'$.
In this case, we collapse~$M'= \T^m \# M$ to the graph~$P=[0,1] \cup_{\{1\} = p_1} S^1$ so that the preimage of~$p_1 \neq p_0$ is the attaching sphere of the connected sum, the torus~$\T^m$ is sent to~$[0,1]$ and the term~$M$ is sent to~$S^1$ as before.
\end{remark}
???\\forgotten

???\\forget
\begin{example}
Finitely generated groups with intermediate growth (\ie, groups with subexponential but non-polynomial growth) have first been constructed by R.~Grigorchuk; see~\cite{grigorchuk}.
Though it is still an open question whether finitely presented groups with intermediate growth exist or not, we will consider a finitely presented group~$G$ of subexponential growth with, say, $N_G(T) \gtrsim \exp(\sqrt{T})$, where $N_G(T)$ is the number of elements of~$G$ with word length at most~$T$ with respect to a fixed finite generating set of~$G$.
(Several finitely generated groups with various growth have been constructed since the first Grigorchuk group.)
In this case, the subexponential growth rate of~$G$ is at least~$\frac{1}{2}$.
Let $N$ be a closed triangulated $m$-manifold with fundamental group~$G$, for some $m \geq 3$.
Endow $N$ with a piecewise flat metric where all simplices of~$N$ are isometric to the standard Euclidean simplex with the same dimension.
The number of homotopy classes of~$N$ which can be represented by a loop of length at most~$T$ is at least $C \, \exp(\sqrt{T})$ for some positive constant~$C$.
Consider the closed $m$-manifold $M = N \# (S^1 \times S^{m-1})$ obtained by attaching a handle~$S^1 \times S^{m-1}$ to~$N$.
There exists a simplicial map $\pi:M \to P=S^1 \times D^{m-2}$ to a simplicial $(m-1)$-complex~$P$, which shrinks~$N$ to a vertex~$\star \in P$ and collapses the $S^{m-1}$-factor of the handle to the $(m-2)$-dimensional disk~$D^{m-2}$ so that most of the fibers of the simplicial map~$\pi:M \to P$ are $1$-dimensional.
By construction, the fundamental groups of the fibers of~$\pi:M \to P$ have subexponential growth.
Modify the metric on~$M$ by shrinking the length of the fibers of~$\pi:M \to P$ by a factor~$t$ as in the proof of Theorem~\ref{theo:A}. and denote by~$g_t$ the metric so-obtained.
Note that $\vol(M,g_t) \geq C' \, t$ for some positive constant~$C'$.
Observe also that the number of homotopy classes of the fiber~$F_\star = \pi^{-1}(\star)$ which can be represented by loops of $g_t$-length at most~$T$ is at least $C \, \exp \left( \sqrt{\frac{T}{t}} \right)$. 
Fix $T=2n$.
Consider all the loops $\gamma = \alpha_1 \beta \cdots \alpha_n \beta$ of~$M$ of $g_t$-length at most~$T$ which can be constructed by taking a loop~$\alpha_1$ in~$F_\star$ of $g_t$-length~$1$, followed by a loop~$\beta$ of $g_t$-length~$1$ going around the handle~$S^1 \times S^{m-1}$, then by a loop~$\alpha_2$ in~$F_\star$ of $g_t$-length~$1$, then by~$\beta$ and so on.
We obtain at least $C^n \, \exp\left( \frac{n}{\sqrt{t}} \right) = C^{\frac{T}{2}} \, \exp \left( \frac{T}{2 \sqrt{t}} \right)$ homotopy classes.
Thus, $\ent(M,g_t) \gtrsim \frac{1}{\sqrt{t}}$.
Hence, $\ent(M,g_t)^m \, \vol(M,g_t) \gtrsim \frac{1}{t^{\frac{m}{2}-1}}$ which tends to infinity as $t$ goes to zero.
\end{example}
\forgotten


\subsection{Fiber collapsing assumption and zero simplicial volume} \label{sec:Gromov}

\mbox { }

\medskip

Drawing a parallel with the simplicial volume through Gromov's vanishing simplicial volume theorem, we show that a manifold satisfying the fiber collapsing assumption has zero simplicial volume.

\begin{definition} \label{def:amenable}
A group~$G$ is \emph{amenable} if it admits a finitely-additive left-invariant probability measure.
A path-connected open subset~$U$ of a path-connected topological space~$X$ is \emph{amenable in~$X$} if $i_*[\pi_1(U)]$ is an amenable subgroup of~$\pi_1(X)$, where \mbox{$i:U \hookrightarrow X$} is the inclusion map.
\end{definition}

Gromov's vanishing simplicial volume theorem can be stated as follows.

\begin{theorem}[\cite{gro82}, see also~\cite{Ivanov80}] \label{theo:vanishing}
Let $M$ be a connected closed $m$-manifold.
Suppose that $M$ admits a covering by amenable open subsets of multiplicity at most~$m$.
Then 
\[ 
\Vert M \Vert_{\D} =0.
\]

In particular, the simplicial volume of a connected closed manifold with amenable fundamental group is zero.
\end{theorem}

The characterization of the fiber collapsing assumption in terms of coverings allows us to derive the following result about the effect of the fiber collapsing assumption on the simplicial volume.
Note that, contrarily to Theorem~\ref{theo:A}, there is no hypothesis about how the subexponential growth rate compares to the dimensions.

\begin{proposition} \label{prop:0simplicial}
Every closed $m$-manifold~$M$ satisfying the fiber collapsing assumption has zero simplicial volume. 
\end{proposition}

\begin{proof}
Recall that every finitely generated group with subexponential growth is amenable; see~\cite{AS57} or~\cite[Theorem~6.11.12]{CSC} for instance.
Thus, every open subset $U \subseteq M$ with subexponential $\pi_1$-growth in~$M$, see Definition~\ref{def:subexp}, is amenable in~$M$.
By Proposition~\ref{prop:cover}, the manifold~$M$ admits a covering of multiplicity at most~$m$ by open subsets of subexponential $\pi_1$-growth in~$M$, and so by amenable open subsets.
It follows from Theorem~\ref{theo:vanishing} that $M$ has zero simplicial volume.
\end{proof}

\forget
\begin{remark} \label{rem:amenable}
Recall that every finitely generated group with subexponential growth is amenable; see~\cite{AS57} or~\cite[Theorem~6.11.12]{CSC} for instance.
Thus, every open subset $U \subseteq X$ with subexponential $\pi_1$-growth in~$X$, see Definition~\ref{def:subexp}, is amenable in~$X$.
This observation allows us to present an alternate proof of Corollary~\ref{coro:zerovol} which asserts that connected closed $m$-manifolds satisfying the fiber collapsing assumption have zero simplicial volume.
Indeed, by Proposition~\ref{prop:cover}, such a manifold~$M$ admits a covering of multiplicity at most~$m$ by open subsets of subexponential $\pi_1$-growth in~$M$, and so by amenable open subsets.
It follows from Theorem~\ref{theo:vanishing} that $M$ has zero simplicial volume.
\end{remark}
\forgotten

\forget
\begin{question}
Though every finitely generated group with subexponential growth is amenable, the converse is false, even among finitely presented groups.
For instance, the following finitely presented group
\[
G = \langle a,t \mid a^{t^2} = a^2, [[[a,t^{-1}],a],a]=1 \rangle
\]
is amenable and has exponential growth; see~\cite{BV05}.
(Here, $x^y=y^{-1}xy$ is the conjugate of~$x$ under~$y$.)
According to Theorem~\ref{theo:vanishing}, every closed manifold with fundamental group isomorphic to~$G$ has zero simplicial volume.
Is the minimal volume entropy of a simplicial complex with fundamental group isomorphic to~$G$ always zero?
\end{question}
\forgotten

\forget

\mbox { }

\medskip

Corollary~\ref{coro:zerovol} claims that closed manifolds satisfying the fiber collapsing assumption have zero simplicial volume.
We give an alternate proof of this result which does not rely on Theorem~\ref{theo:A}.

\begin{definition}
A group~$G$ is \emph{amenable} if it admits a finitely-additive left-invariant probability measure.
A connected open subset~$U$ of a manifold~$M$ is \emph{amenable} if $i_*[\pi_1(U)]$ is an amenable subgroup of~$\pi_1(M)$, where \mbox{$i:U \hookrightarrow M$} is the inclusion map.
\end{definition}

Gromov's vanishing simplicial volume theorem can be stated as follows.

\begin{theorem}[\cite{gro82}, see also~\cite{Ivanov80}] \label{theo:vanishing}
Let $M$ be a closed $m$-manifold.
Suppose that $M$ admits a covering by amenable open subset of multiplicity at most~$m$.
Then 
\[ 
\Vert M \Vert_{\D} =0.
\]
In particular, the simplicial volume of a closed manifold with amenable fundamental group is zero.
\end{theorem}

We can now proceed with an alternate proof that the fiber collapsing assumption implies the vanishing of the simplicial volume; see Corollary~\ref{coro:zerovol}.

\begin{proof}[Second proof of Corollary~\ref{coro:zerovol}]
Consider a simplicial map $\pi:M \to P$ onto a finite simplicial $(m-1)$-complex~$P$ as in the fiber collapsing assumption.
Every point $p \in P$ has a small open neighborhood~$B_p \subseteq P$ whose preimage~$U_p=\pi^{-1}(B_p)$ is homotopy equivalent to the fiber $F_p=\pi^{-1}(p)$.
By assumption, it follows that the subgroups~$\Gamma_p=i_*[\pi_1(U_p)] \leqslant \pi_1(M)$ have subexponential growth and so are amenable.

Since $P$ is a finite simplicial complex of dimension~$m-1$, 
we can choose the covering of open stars of vertices of $P$. 
This covering is of multiplicity $m$ and for each vertex $p \in P$, the corresponding subset $\{U_p\}$ is homotopy equivalent to $\pi^{-1}(p)$.
Since the open subsets~$U_p$ are amenable, it follows from Gromov's vanishing simplicial volume theorem, see Theorem~\ref{theo:vanishing}, that the simplicial volume of~$M$ is zero.
\end{proof}

\begin{proof}[Second proof]
Theorem~\ref{theo:zerovol}  is also a direct consequence of Theorem~\ref{theo:zerovol} showing that $\omega(M)=0$ and the inequality~\eqref{eq:gro} providing an upper bound on $||M||_\Delta$ in terms of~$\omega(M)$.
\end{proof}
\forgotten

\subsection{Collapsing with Ricci curvature bounded below}

\mbox{ }

\medskip

In this section, we show that the collapsing of manifolds with Ricci curvature bounded below is connected to the fiber collapsing assumption.

\medskip

\forget
Let us introduce the following notion studied in~\cite{CC97}.

\begin{definition} \label{def:collapsingRicci}
A sequence of closed Riemannian $m$-manifolds~$M_k$ \emph{collapses with Ricci curvature bounded below} if 
\begin{itemize}
\item $\Ric(M_k) \geq -(m-1)$;
\item the sequence~$M_k$ converges to a compact metric space~$X$ for the Gromov-Hausdorff topology;
\item $\lim_{k \to \infty} \vol(M_k) =0$ or equivalently $\dim_H X \leq m-1$, where $\dim_H X$ is the Hausdorff dimension of~$X$.
\end{itemize}
See~\cite{CC97} for a description of the structure of the limit space~$X$.
\end{definition}
\forgotten

Recall the following result of V. Kapovitch and B. Wilking. 

\begin{theorem}[Generalized Margulis Lemma, see~\cite{KW} and also~\cite{courtois}] \label{theo:KW}
For every positive integer~$m$, there exist two constants $\varepsilon_m \in (0,1)$ and~$C_m>0$ such that for every complete Riemannian $m$-manifold~$M$ with $\Ric_M \geq -(m-1)$, the image of the natural homomorphism
\begin{equation} \label{eq:KW}
\pi_1(B(x,\varepsilon_m)) \to \pi_1(B(x,1))
\end{equation}
induced by the inclusion contains a nilpotent subgroup of index at most~$C_m$.

In particular, the image of~\eqref{eq:KW} is virtually nilpotent and so has polynomial growth.
\end{theorem}

As an application of this theorem, Vitali Kapovitch pointed out to us that collapsing with Ricci curvature bounded below (studied by Cheeger and Colding in~\cite{CC97}) implies the fiber collapsing assumption.
More precisely, we have the following result.

\begin{proposition} \label{prop:Ricci} 
For every positive integer~$m$, there exists $v_m>0$ such that every closed Riemannian $m$-manifold~$M$ with $\Ric_M \geq -(m-1)$ and $\vol(M) \leq v_m$ satisfies the fiber collapsing assumption with polynomial growth rate.

In this case, the manifold~$M$ has zero minimal volume entropy.
\end{proposition}

\begin{proof}
Let $\varepsilon_m \in (0,1)$ be the constant in the Generalized Margulis Lemma; see Theorem~\ref{theo:KW}.
By the nerve construction of~\cite[\S3.4]{gro82}, if every ball of radius~$\frac{\varepsilon_m}{4}$ in~$M$ has volume at most~$v_m$ with $v_m>0$ small enough (in particular, if $\vol(M) \leq v_m$) then there exists a continuous map $f:M \to P$ to a finite simplicial complex~$P$ of dimension at most~$m-1$ such that for every $p \in P$, the fiber~$f^{-1}(p)$ lies in some ball of radius~$\varepsilon_m$ in~$M$; see~\cite[Corollary, p.~52]{gro82}.
By the last statement of Theorem~\ref{theo:KW}, the subgroup~$i_*[\pi_1(F_p)] \leqslant \pi_1(M)$, where $i:F_p \hookrightarrow M$ is the inclusion map of a connected component~$F_p$ of~$f^{-1}(p)$, has polynomial growth (recall that a subgroup or a quotient of a virtually nilpotent group is virtually nilpotent).
Thus, the manifold~$M$ satisfies the fiber collapsing assumption with polynomial growth rate.
By Corollary~\ref{coro:polynomial}, it follows that $M$ has zero minimal volume entropy.
\end{proof}

\begin{remark}
This is a refinement of Gromov's isolation theorem~\cite[\S0.5]{gro82} which asserts that every manifold~$M$ in Proposition~\ref{prop:Ricci} has zero simplicial volume.
\end{remark}

\forget
\begin{theorem} \label{theo:Ricci}
Let $M_k$ be a sequence of closed Riemannian $m$-manifolds collapsing with Ricci curvature bounded below.
Then the manifolds~$M_k$ satisfy the fiber collapsing assumption for $k$ large enough.
\end{theorem}

\begin{proof}
By definition, the sequence~$M_k$ converges to a compact metric space~$X$ for the Gromov-Hausdorff topology.
That is, there exists a metric space~$\mathcal{X}$ containing an isometric copy of the spaces~$M_k$ and of the limit space~$X$ such that the sequence~$M_k$ converges to~$X$ for the Hausdorff distance in~$\mathcal{X}$.
Since the topological dimension of a metric space is bounded by its Hausdorff dimension, see~\cite[VII.4 and V.8]{dim}, the limit metric space~$X$ admits a covering of multiplicity at most~$m$ by open subsets~$U_i$ of arbitrarily small diameter $d<\frac{\varepsilon_m}{2}$, where $\varepsilon_m$ is the constant in the Generalized Margulis Lemma; see Theorem~\ref{theo:KW}.
Every open subset $U_i \subseteq X$ extends to an open subset $\mathcal{U}_i \subseteq \mathcal{X}$ containing~$U_i$ as follows
\[
\mathcal{U}_i = \bigcup_{x \in U_i} B_{\mathcal{X}}(x,r_i(x))
\]
where $r_i(x)=\frac{1}{2} d_X(x,X \setminus U_i) \leq \frac{1}{2} \diam(U_i)$.

Now, observe that $\mathcal{U}_i$ and~$\mathcal{U}_j$ intersect each other in~$\mathcal{X}$ if and only if $U_i$ and~$U_j$ intersect each other in~$X$.
Indeed, by contradiction, suppose that $U_i$ and~$U_j$ are disjoint and that $\mathcal{U}_i$ and~$\mathcal{U}_j$ intersect at~$y \in \mathcal{X}$.
By definition, there exist $x_i \in U_i$ and~$x_j \in U_j$ such that 
\[
d_{\mathcal{X}}(y,x_i) < \frac{1}{2} d_X(x_i,X \setminus U_i) \mbox{ and } d_{\mathcal{X}}(y,x_j) < \frac{1}{2} d_X(x_j,X \setminus U_j).
\]
Note that $d_X(x_i,X \setminus U_i) \leq d_X(x_i,x_j)$ and $d_X(x_j,X \setminus U_j) \leq d_X(x_i,x_j)$ since $x_j \in X \setminus U_i$ and $x_i \in X \setminus U_j$.
Thus,
\[
d_{\mathcal{X}}(x_i,x_j) \leq d_{\mathcal{X}}(x_i,y) + d_{\mathcal{X}}(y,x_j) < d_X(x_i,x_j).
\]
Hence a contradiction since $X$ is isometrically embedded into~$\mathcal{X}$.

Thus, the coverings~$(U_i)$ and~$(\mathcal{U}_i)$ have the same multiplicity bounded by~$m-1$.
By construction, the diameter of every open set~$\mathcal{U}_i$ is at most~$2 \diam(U_i) < \varepsilon_m$, and so is the diameter of every open subset~$\mathcal{U}_i \cap M_k \subseteq M_k$ (recall that $M_k$ is isometrically embedded into~$\mathcal{X}$).
We deduce from the last statement of Theorem~\ref{theo:KW} that every open subset~$\mathcal{U}_i \cap M_k$ has polynomial growth in~$M_k$.
Thus, the open subsets~$(\mathcal{U}_i \cap M_k)_i$ form a covering of~$M_k$ of multiplicity~$m-1$ by open subsets of polynomial growth in~$M_k$.
That is, the manifolds~$M_k$ satisfy the fiber collapsing assumption for $k$ large enough; see Proposition~\ref{prop:cover}.
\end{proof}
\forgotten

\forget
The following corollary was also communicated to us by Vitali Kapovitch.

\begin{corollary}
A closed aspherical manifold~$M$ whose fundamental group is non-elementary word hyperbolic cannot collapse with Ricci curvature bounded below.
\end{corollary}

\begin{proof}
We argue by contradiction.
Consider a sequence of Riemannian metrics~$g_k$ on~$M$ such that $M_k=(M,g_k)$ collapses with Ricci curvature bounded below.
By Theorem~\ref{theo:Ricci}, the manifold~$M_k$ satisfies the fiber collapsing assumption for $k$ large enough.
On the other hand, by the last statement of Proposition~\ref{prop:hyp}, the manifold~$M$ satisfies the fiber non-collapsing assumption.
Hence a contradiction.
\end{proof}
\forgotten

\section{Simplicial complexes with positive minimal volume entropy}

In this section, we introduce the covering non-collapsing assumption and show that it is equivalent to the fiber growth non-collapsing assumption when the fundamental group is thick,
Then, relying on the notion of Urysohn width, we show that the minimal volume entropy of simplicial complexes satisfying the covering non-collapsing assumption and some mild combinatorial conditions is positive.
We also establish a similar result for simplicial complexes satisfying the more manageable fiber growth non-collapsing assumption, without the combinatorial conditions, when the fundamental group is thick.
Finally, we construct simplicial complexes with zero simplicial volume and arbitrarily large minimal volume entropy.


\subsection{Covering non-collapsing assumption} \label{sec:positivevol}

\mbox { }

\medskip



As in Section~\ref{sec:zerovol}, we begin with some definitions.

\begin{definition} \label{def:exp.cover}
A covering~$\mathcal{U} = \{U_i \}$ of a path-connected topological space~$X$ by path-connected open subsets has \emph{uniform exponential $\pi_1$-growth} at least~$h$ if for at least one open subset~$U$ of~$\mathcal{U}$, the subgroup $\Gamma_{U}:=i_*[\pi_1(U)]$ of~$\pi_1(X)$ has uniform exponential growth at least~$h$, where \mbox{$i:U \hookrightarrow X$} is the inclusion map.
\end{definition}

\textbf{Covering non-collapsing assumption (CNCA).}
Let $X$ be a finite connected simplicial $m$-complex.
Suppose that every finite open covering of~$X$ of multiplicity at most~$m$ has uniform exponential $\pi_1$-growth at least~$h$, for some $h=h(X) >0$ depending only on~$X$ (and not on the open covering).


\medskip


\forget
{\color{red}(compare with \cite{BC})}

\begin{definition} \label{def:condition.U}
A finitely generated group~$G$ has \emph{Property~${\bf U}$} if
\begin{enumerate}
\item there exists a constant $\underline{h}=\underline{h}(G)>0$ such that every finitely generated subgroup~$H \leqslant G$ of exponential growth has uniform exponential growth at least~$\underline{h}$, \ie, $\ent(H) \geq \underline{h}$; \label{U1}
\item every finitely generated subgroup~$H \leqslant G$ of subexponential growth has zero subexponential growth rate, \ie, $\nu(G)=0$.
\end{enumerate}
By extension, a finite simplicial complex~$X$ has Property~${\bf U}$ if its fundamental group~$\pi_1(X)$ has Property~${\bf U}$.
\end{definition}
\forgotten




Contrarily to the collapsing case, see Proposition~\ref{prop:cover}, the equivalence between the various non-collapsing assumptions holds only for thick groups.



\begin{proposition} \label{prop:exp.cover}
Let $X$ be a connected finite simplicial $m$-complex.
\begin{enumerate}
\item If $X$ satisfies the covering non-collapsing assumption with constant~$h$ then $X$ satisfies the fiber non-collapsing assumption with the same constant~$h$. \label{ec1}
\item Suppose that $\pi_1(X)$ is $\delta$-thick.
If $X$ satisfies the fiber non-collapsing assumption then $X$ satisfies the covering non-collapsing assumption with constant~$\delta$. \label{ec2}
\end{enumerate}
\end{proposition}
\begin{proof}
We argue as in the proof of Proposition~\ref{prop:cover}.

Let $\pi:X \to P$ be a simplicial map onto a simplicial complex~$P$ of dimension~$k < m$.
By Proposition~\ref{prop:connected}, we can assume that the fibers of~$\pi:X \to P$ are connected.
Since $P$ is a finite simplicial complex of dimension~$k$, the covering of~$P$ formed of the open stars ${\rm st}(p) \subseteq P$ of the vertices~$p$ of~$P$ has multiplicity~$k+1$.
The preimages $\pi^{-1}({\rm st}(p)) \subseteq X$ of these open stars form an open covering~$\mathcal{U}$ of~$X$ with the same multiplicity $k+1 \leq m$ as the previous covering of~$P$.
Since $X$ satisfies the covering non-collapsing assumption, there exists an open subset~$U_0$ of~$\mathcal{U}$ such that the subgroup~$\Gamma_{U_0} \leqslant \pi_1(X)$ has uniform exponential growth at least~$h$.
By construction of~$\mathcal{U}$, the open subset~$U_0$ strongly deformation retracts onto a fiber~$F_{p_0}=\pi^{-1}(p_0)$.
It follows that the subgroup~$\Gamma_{p_0} = i_*[\pi_1(F_{p_0})]$ is isomorphic to~$\Gamma_{U_0}$ and has also uniform exponential growth at least~$h$.
This proves the point~\eqref{ec1}.

\medskip

Let $\mathcal{U}=\{U_i\}$ be a finite open covering of~$X$ of multiplicity at most~$m$.
Consider a simplicial map $\pi:X \to P$ onto the nerve~$P$ of the covering~$\mathcal{U}$ constructed from a partition of unity subordinate to~$\mathcal{U}$ as in the proof of Proposition~\ref{prop:cover}.
By construction, the normalized barycentric coordinates $\pi_i:X \to [0,1]$ have their support in~$U_i$.
In particular, every fiber~$F_p=\pi^{-1}(p)$ over a point~$p \in P$ lies in some open subset~$U_i$.
Since $X$ satisfies the fiber non-collapsing assumption, there exists a fiber~$F_{p_0}$, contained in some open subset~$U_{i_0}$, such that the subgroup~$\Gamma_{p_0}$ has (uniform) exponential growth.
Since $F_{p_0} \subseteq U_{i_0}$, we have $\Gamma_{p_0} \leqslant \Gamma_{U_{i_0}}$ and the subgroup $\Gamma_{U_{i_0}} \leqslant \pi_1(X)$ has also exponential growth.
Since $\pi_1(X)$ is $\delta$-thick, it follows that $\Gamma_{U_{i_0}}$ has uniform exponential growth at least~$\delta$.
This proves the point~\eqref{ec2}.
\forget
{\color{red}
1Let $P$ be a finite simplicial complexe of dimension $k<m$. Consider a map $f : X \longrightarrow P$ 
which can be supposed simplicial. Denote $\{p_i\}$ the set of vertices of $P$ 
and let $\mathcal{V} = \{V_i \}$ be the covering of $P$ by open stars of the vertices. Put $U_i = f^{-1}(V_i)$
so we have an open covering $\mathcal{U} = \{U_i \} $ of $X$ of multiplicity $k+1 \leq m$. 
If $X$ satisfies  the covering non-collapsing assumption the existe $i_0$ such that the sub-group 
$\Gamma_{U_{i_0}}$ has uniforme exponential grouwth at least $h$. The construction of $\mathcal{U}$
implies that $U_i$ are homotopy equivalent to $f^{-1}(p_i)$ for all $i$. So we see that the subgroup
$\Gamma_{f^{-1}(p_i)}$ has uniform exponential $\pi_1$-growth at least~$h$ as well.
This implies that $X$ satisfies the fiber non-collapsing assumption with the same constant $h$.

2. Suppose now that $\pi_1(X)$ satisfies the property ${\bf U}$ with some constant $l$ 
and $X$ complies with the covering non-collapsing assumption. Consider some finite covering
$\mathcal{U} = \{U_i \}$  of $X$ of mutiplicity less or equal than $m$. Let $P = \mathcal{N}(\mathcal{U})$
be the nerve of $\mathcal{U}$ and $\Phi : X \longrightarrow P$ be some classifying map constructed
by a partition of unity subordinate to  $\mathcal{U}$. 

Denote $\{p_i\}$ the vertices of $P$ wthich are in one-to-one correspondence with $\{U_i \}$. 
As $\dim(P) < m$ and $X$ satisfies the fiber non-collapsing assumption there exists a vertex $p_{i_0}$
such that $\Gamma_{f^{-1}(p_{i_0})}$ has exponential grouwth. Because 
$f^{-1}(p_{i_0}) \subset U_{i_0}$ and then  $\Gamma_{f^{-1}(p_{i_0})} <  \Gamma_{U_{i_0}}$
the last group, $\Gamma_{U_{i_0}}$, has exponential grouth as well. The ${\bf U}$-hypothesis implies now
that $\Gamma_{U_{i_0}}$ has uniforme exponential grouwth at least $l$, so the result.
}
\forgotten
\end{proof}

\begin{remark}
If $\pi_1(X)$ is $\delta$-thick, the notions of non-collapsing in terms of open coverings (CNCA) and of fiber growth (FNCA) are equivalent.
Furthermore, the constant~$h$ in the definitions of the non-collapsing assumptions satisfies $h \geq \delta$, but a priori, this inequality can be strict.
\end{remark}


The collapsing and non-collapsing assumptions, whether in terms of open coverings or fiber growth, are not complementary in general.
However, they are complementary for simplicial complexes with thick fundamental groups; compare with~\cite[Lemma~3.8]{BC}.

\begin{proposition} \label{prop:alternative}
Let $X$ be a connected finite simplicial $m$-complex with thick fundamental group.
Then $X$ satisfies either the covering collapsing assumption, or the covering non-collapsing assumption.

Similarly, $X$ satisfies either the fiber collapsing assumption, or the fiber non-collapsing assumption.
\end{proposition}

\begin{proof}
Suppose that $X$ does not satisfy the covering collapsing assumption.
Let $\mathcal{U}$ be an open covering of~$X$ of multiplicity at most~$m$.
There is a subset~$U$ of~$\mathcal{U}$ such that the subgroup~$\Gamma_U := i_*[\pi_1(U)]$ has exponential growth.
Since $\pi_1(X)$ is thick, the subgroup~$\Gamma_{U}$ has uniform exponential growth.
Therefore, $X$ satisfies  the covering non-collapsing assumption.

For the second statement, either we argue similarly, or we use the fact that FCA $\Leftrightarrow$ CCA and FNCA $\Leftrightarrow$ CNCA when $\pi_1(X)$ is thick.
\end{proof}

\subsection{Examples of thick groups and non-collapsing simplicial complexes} \label{sec:exthick}

\mbox { }

\medskip

Let us give some examples of $\delta$-thick groups:

\begin{enumerate}
\item $G$ is a group whose $2$-generated subgroups are free, with $\delta=\log(3)$. 
Examples of such groups can be found in~\cite{guba}, \cite{bumagin} and~\cite{AO96}.
Generically, all finitely presented groups satisfy this property; see~\cite{AO96}.
\item $G$ is a torsion-free non-elementary word hyperbolic group with $\delta=\delta(G)$ depending on~$G$; see~\cite{delzant}.
\item $G$ is a discrete subgroup of the isometry group of an $m$-dimensional Cartan-Hadamard manifold of pinched sectional curvature $-a^2 \leq K \leq -1$, with $\delta=\delta(m,a)$ depending only on~$m$ and~$a$; see~\cite{BCG11}. 
More generally, $G$ is a discrete subgroup of the isometry group of a geodesic Gromov hyperbolic space with bounded geometry; see~\cite{BCGS} and~\cite{BF}.
\item $G$ has exponential growth (\ie, non virtually abelian in this case) and acts freely on a ${\rm CAT}(0)$ cube complex of dimension two or three, with $\delta>0$ depending only on the dimension (\eg, $\delta=\frac{1}{10} \log(2)$ in the $2$-dimensional case); see~\cite{KS19} and~\cite{GJN}.
\item $G$ has exponential growth (\ie, non virtually abelian in this case) and acts freely on a $\CAT(0)$ cube $m$-complex with isolated flats or freely and weakly properly discontinuously on a Gromov hyperbolic $\CAT(0)$ cube $m$-complex, with $\delta=\delta_m$ depending only on~$m$; see~\cite{GJN}. \label{IF}
%
\item $G$ is a triangle-free Artin group or the Higman group, with $\delta=\sqrt[600]{2}$; see~\cite{GJN}.
\item $G$ is the mapping class group of a compact orientable surface~$S$, with $\delta=\delta_S$ depending on~$S$; see~\cite{man10}.
\end{enumerate}
Of course, any subgroup with exponential growth of a $\delta$-thick group is $\delta$-thick.


\medskip

\forget

An additional example is given by the following lemma which follows from~\cite{delzant}.

\begin{proposition} \label{prop:thick}
Every torsion-free non elementary hyperbolic group~$G$ is thick.
\end{proposition}

\begin{proof}
By~\cite{delzant}, there exists an integer~$n_G$ such that for every $x,y \in G$ which do not commute and for every $n \geq n_G$, the subgroup~$\langle x^n,y^n \rangle \leqslant G$ generated by~$x^n$ and~$y^n$ is free of rank~$2$.

Let $H \leqslant G$ be a subgroup generated by a finite set~$S$ of~$G$.
Assume that $H$ has exponential growth.
Take $x,y \in S$ which do not commute.
(Such pair of elements exists otherwise the subgroup~$H$ would be abelian.)
We have $B_{S_0}\left(\frac{t}{n_G}\right) \subseteq B_S(t)$, where $S_0= \{ x^{\pm n_G}, y^{\pm n_G} \}$.
Here, $B_S(t)$ is formed of the elements of~$H$ at distance at most~$t$ from the identity element with respect to the word distance induced by~$S$.
Taking the exponential growth rate of these balls, we obtain $\ent(H,S) \geq \frac{1}{n_G} \log(3)$, where $\log(3)$ is the exponential growth rate of the free group of rank~$2$ generated by~$S_0$ with respect to the word distance induced by~$S_0$.
This implies that the subgroup~$H$ has uniform exponential growth at least~$h(G) = \frac{1}{n_G} \log(3)$.
Hence the result.
\end{proof}

\forgotten

The following result provides examples of simplicial complexes satisfying the covering/fiber non-collapsing assumption.

\begin{proposition} \label{prop:hyp}
Let $X$ be a finite aspherical simplicial $m$-complex with $H_m(X;\R)$ nontrivial, where $m \geq 2$. 
Suppose the fundamental group of~$X$ is a non-elementary word hyperbolic group.
Then $X$ satisfies the covering non-collapsing assumption (and thus the fiber non-collapsing assumption).

In particular, every closed orientable aspherical manifold whose fundamental group is a non-elementary word hyperbolic group satisfies the covering non-collapsing assumption (and thus the fiber non-collapsing assumption).
\end{proposition}

\begin{proof}
First observe that since $X$ is aspherical, its fundamental group~$\pi_1(X)$ is torsion-free, otherwise there would exist a finite-dimensional aspherical space with a finite fundamental group, which is impossible; see~\cite[Proposition~2.45]{hatcher}.
%
%
Suppose $X$ does not satisfy the covering non-collapsing assumption.
Since $\pi_1(X)$ is a thick group, it follows from Proposition~\ref{prop:alternative} that $X$ satisfies the covering collapsing assumption.
That is, there is a covering of~$X$ of multiplicity~$\leq m$ by open subsets of subexponential $\pi_1$-growth.
In particular, the open subsets of this covering are amenable in~$X$; see Definition~\ref{def:amenable}.
According to the generalization given by~\cite[Theorem~9.2]{Ivanov80} (also proved via different approaches in~\cite{FM} and~\cite{LS}) of Gromov's vanishing simplicial volume theorem, see~Theorem~\ref{theo:vanishing}, the canonical homomorphism $H_b^m(X;\R) \to H^m(X;\R)$ between bounded cohomology and singular cohomology vanishes.
By~\cite{mineyev}, the canonical homomorphism $H_b^m(X;\R) \to H^m(X;\R)$ is also surjective.
Hence, $H^m(X;\R)$ is trivial, which leads to a contradiction.
Indeed, by assumption, $H_m(X;\R)$ is nontrivial, and by the universal coefficient theorem for cohomology, $H^m(X;\R) = {\rm Hom}(H_m(X;\R),\R)$ is also nontrivial.
Therefore, $X$ satisfies the covering non-collapsing assumption and so the fiber non-collapsing assumption by Proposition~\ref{prop:exp.cover}.
\end{proof}

In connection with Proposition~\ref{prop:0simplicial}, one can ask the following question.

\begin{question}
Does every closed orientable manifold~$M$ satisfying the fiber non-collapsing assumption have positive simplicial volume?
Otherwise, find examples of closed orientable manifolds with zero simplicial volume satisfying the fiber non-collapsing assumption.
\end{question}

\subsection{Urysohn width and volume}
\mbox { }

\medskip

Let us go over the notion of Urysohn width in metric geometry; see~\cite{gro88} for further context.


\begin{definition} \label{def:UW}
The \emph{Urysohn $q$-width} of a compact metric space~$X$, denoted by~$\UW_q(X)$, is defined as the least real~$w >0$ such that there exists a finite covering~$\mathcal{U}$ of~$X$ of multiplicity at most~$q+1$ by (path-connected) open subsets~$U$ of diameter less than~$w$ in~$X$.
That is,
\[
\UW_q(X) = \inf_{\substack{U \in \, \mathcal{U} \\ m(\mathcal{U}) \leq q+1}} \diam_X(U).
\]
For a simplicial $m$-complex~$X$, we will simply write $\UW(X)$ for~$\UW_{m-1}(X)$.
\forget
The \emph{Urysohn $q$-width} of a compact metric space~$X$, denoted by~$\UW_q(X)$, is defined as the least real~$\delta >0$ such that there is a continuous map $\pi:X \to P$ from~$X$ to a simplicial $q$-complex~$P$, where all the fibers~$\pi^{-1}(p)$ have diameter at most~$\delta$ in~$X$.
That~is,
\begin{equation} \label{eq:width}
\UW_q(X) = \adjustlimits  \inf_{\pi:X \to P} \sup_{p \, \in P} \, \diam_X [\pi^{-1}(p)]
\end{equation}
where $\pi:X \to P$ runs over all continuous map from~$X$ to a simplicial $q$-complex~$P$ and $p$ runs over all points of~$P$.
Note that the simplicial complex~$P$ may vary with~$\pi:X \to P$.
For a simplicial $m$-complex~$X$, we will simply write $\UW(X)$ for~$\UW_{m-1}(X)$.
\forgotten
\end{definition}

The Urysohn width can also be interpreted in terms of fiber diameter; see~\cite[Lemma~0.8]{Guth17} for instance.

\begin{proposition} \label{prop:UW}
A compact metric space~$X$ has Urysohn $q$-width less than~$w$ if and only if there exists a continuous map $\pi:X \to P$ from~$X$ to a simplicial $q$-complex~$P$, where all the fibers~$\pi^{-1}(p)$ have diameter at most~$w$ in~$X$.
That~is,
\begin{equation} \label{eq:width}
\UW_q(X) = \adjustlimits  \inf_{\pi:X \to P} \sup_{p \, \in P} \, \diam_X [\pi^{-1}(p)]
\end{equation}
where $\pi:X \to P$ runs over all continuous map from~$X$ to a simplicial $q$-complex~$P$ and $p$ runs over all points of~$P$.
Note that the simplicial complex~$P$ may vary with~$\pi:X \to P$.
\end{proposition}

In the case of simplicial complexes, we can further require extra structural properties on the map~$\pi:X \to P$ in the previous proposition.

\begin{proposition} \label{prop:width}
Let $X$ be a finite simplicial complex with a piecewise Riemannian metric.
Subdividing~$X$ if necessary, we can assume that the maps~$\pi:X \to P$ in the relation~\eqref{eq:width} are surjective and simplicial, and that their fibers are connected.
\end{proposition}

\begin{proof}
Suppose $\UW_q(X) < w$.
By definition, there is a finite open covering~$\mathcal{U} = \{ U_i \}_{i=1,\cdots, s}$ of~$X$ of multiplicity~$q+1$ and diameter less than~$w$.
Consider the natural map $\Phi:X \to P \subseteq \Delta^{s-1}$ to the nerve~$P$ of~$\mathcal{U}$ given by a partition of unity of the covering.
As in the proof of Proposition~\ref{prop:cover}, subdividing~$X$ and~$P$, we can approximate $\Phi:X \to P$ by a simplicial map $\pi:X \to P$ close to~$\Phi$ for the $C^0$-topology, whose normalized barycentric coordinates $\pi_i:X \to [0,1]$ have their support in~$U_i$; see~\cite[\S2.C]{hatcher}.
Thus, every fiber~$\pi^{-1}(p)$ lies in one of the open sets~$U_i$.
Therefore, $\diam_X [\pi^{-1}(p)] < w$.
As a result, we can assume that the map $\pi:X \to P$ is simplicial in Proposition~\ref{prop:UW}; see~\eqref{eq:width}.
Now, by Proposition~\ref{prop:connected}, we can replace $\pi:X \to P$ with a surjective simplicial map $\bar{\pi}:X \to \bar{P}$ onto a simplicial complex~$\bar{P}$ of dimension at most~$q$, whose fibers are connected and of diameter less than~$w$.
\forget
Suppose $\UW_q(X) < \delta$.
By~???, there is a finite open covering~$\{ U_i \}_{i=1,\cdots, s}$ of~$X$ of multiplicity~$q+1$ and diameter less than~$\delta$.
Recall that the nerve~$P$ of the covering~$\{ U_i \}$ is a simplicial complex with one vertex~$v_i$  for each open set~$U_i$, where $v_{i_0},\cdots,v_{i_k}$ span a $k$-simplex of~$P$ if and only if the intersection~$\cap_{j=1}^k U_{i_j}$ is nonempty.
The nerve~$P$ lies in~$[0,1]^s$, where the vertex~$v_i$ corresponds to the unit vector~$e_i$ of the canonical basis of~$\R^s$.
Consider a partition of unity~$\{ \phi_i \}$ of~$X$, where each function $\phi_i:X \to [0,1]$ has its support in~$U_i$.
By construction, the map $\phi:X \to [0,1]^s$ with coordinate functions~$\phi_i:X \to [0,1]$ has its image in~$P$.
By~\cite[\S2.C]{hatcher}, subdividing~$X$ and~$P$, we can approximate $\phi:X \to P$ by a simplicial map $\pi:X \to P$ close to~$\phi$ for the $C^0$-topology, whose coordinate functions $\pi_i:X \to [0,1]$ have also their support in~$U_i$.
Thus, every fiber~$\pi^{-1}(p)$ lies in one of the open sets~$U_i$.
Therefore, $\diam_X \pi^{-1}(p) < \delta$.
As a result, we can assume that the map $\pi:X \to P$ is simplicial in the definition of the Urysohn width; see~\eqref{eq:width}.
Now, by Proposition~\ref{prop:connected}, we can replace $\pi:X \to P$ with a simplicial map $\bar{\pi}:X \to \bar{P}$ to a simplicial complex~$\bar{P}$ of dimension at most~$q$, whose fibers are connected and of diameter less than~$\delta$.
\forgotten
\end{proof}

We will need the following recent result of Liokumovich-Lishak-Nabutovsky-Rotman~\cite{LLNR}, extending a theorem of L.~Guth~\cite{Guth17}.
The proof of this result was later on simplified by P.~Papasoglu~\cite{Pap}; see also~\cite{Nab}.

\begin{theorem}[\cite{Guth17}, \cite{LLNR}, \cite{Pap}, \cite{Nab}] \label{theo:width}
Let $X$ be a finite simplicial $m$-complex with a piecewise Riemannian metric.
Then
\[
\vol(X) \geq C_m \, \UW(X)^m
\]
where $C_m$ is an explicit positive constant depending only on~$m$.

More generally, if for some $R>0$, every ball~$B(R) \subseteq X$ of radius~$R$ has volume at most~$C_m \, R^m$ then
\[
\UW(X) \leq R.
\]
\end{theorem}

A more general statement involving the lower dimensional widths and the Hausdorff content of balls holds true; see \cite{LLNR}, \cite{Pap}, \cite{Nab}.

\subsection{Modified Urysohn width and regular simplicial complexes}

\begin{definition} \label{def:UW'}
Let $X$ be a length metric space and~$A \subseteq X$ be a path-connected subset of~$X$.
The \emph{intrinsic distance} between any pair of points of~$A$ is defined as the infimum length of paths of~$A$ between this pair of points.
The \emph{intrinsic diameter} of~$A$, denoted by~$\diam^+(A)$, is the diameter of~$A$ with respect to the intrinsic metric of~$A$.

The \emph{modified Urysohn $q$-width} of~$X$, denoted by~$\UW^+_q(X)$, is defined as the least real $w>0$ such that there exists a finite covering of~$X$ of multiplicity at most~$q+1$ by (path-connected) open subsets of intrinsic diameter less than~$w$ (compare with Definition~\ref{def:UW}).

As previously, for a simplicial $m$-complex~$X$, we will simply write $\UW^+(X)$ for~$\UW^+_{m-1}(X)$.
\end{definition}

Since the intrinsic diameter of an open subset of~$X$ is greater or equal to its extrinsic diameter, we have 
\[
\UW_q(X) \leq \UW^+_q(X).
\]
Let us show that a reverse inequality holds up to a factor two under some combinatorial conditions.

\begin{definition}
Let $X$ be a simplicial complex.
A $k$-simplex $\Delta^k \subseteq X$ is \emph{isolated} if it is not the face of a $(k+1)$-simplex of~$X$.
The simplicial complex~$X$ is \emph{$k$-regular} if its simplices of dimension at most~$k$ are not isolated.
\end{definition}

\begin{proposition} \label{prop:UW+2UW}
Let $X$ be a $2$-regular finite simplicial $m$-complex without locally separating vertices with $m \geq 3$ endowed with a piecewise Riemannian metric.
Then
\[
\UW^+_q(X)  \leq 2 \, \UW_q(X)
\]
for every $q \in \{2,\dots,m-1\}$.
\end{proposition}

\begin{proof}
Fix $\varepsilon >0$.
By Proposition~\ref{prop:width}, subdividing~$X$ if necessary, there exists a surjective simplicial map $\pi:X \to P$ from~$X$ onto a simplicial $q$-complex~$P$ whose fibers are connected and satisfy
\begin{equation} \label{eq:UWq+epsilon}
\diam_X[\pi^{-1}(p)] < \UW_q(X) + \varepsilon
\end{equation}
for every $p \in P$.

Denote by $\Theta(P)$ the triangulation of~$P$ and by~$\Theta^n(P)$ its $n$-th barycentric subdivision (the integer~$n$ will be set later).
Let $\{p_i\}$ be the vertices of~$\Theta^{n-1}(P)$.
The closed stars $\st(p_i) \subseteq P$ of~$p_i$ in the triangulation~$\Theta^n(P)$ form a finite covering of~$P$ of multiplicity~$q+1$.
Note that the points of~$P$ of maximal multiplicity~$q+1$ are exactly the (iso)-barycenters of the $q$-simplices of the triangulation~$\Theta^{n-1}(P)$.

Consider the covering~$\{F_i\}$ of~$X$ by the polyhedral closed subsets 
\[
F_i= \pi^{-1}(\st(p_i)) \subseteq X.
\]
This covering is of multiplicity~$q+1$ and the points of~$X$ of maximal multiplicity~$q+1$ are exactly the points lying in the fibers of the barycenters of the $q$-simplices of~$\Theta^{n-1}(P)$.
Observe that for $n$ large enough, we have
\begin{align*}
\diam_X(F_i) & < \diam_X[\pi^{-1}(p_i)] + \varepsilon \\
 & < \UW_q(X) + 2\varepsilon
\end{align*}
where the second inequality comes from~\eqref{eq:UWq+epsilon}.

Take an $\varepsilon$-dense net~$\{ x^i_j \mid j \in J_i \}$ in each polyhedral subset~$F_i$ with respect to its intrinsic metric.
We can further assume that the points~$x^i_j$ are not vertices of~$X$.
Connect every pair of points~$x^i_j$ and~$x^i_{j'}$ with a length-minimizing geodesic~$\gamma^i_{j,j'}$ of~$X$.
Clearly,
\[
\length(\gamma^i_{j,j'}) \leq \diam_X(F_i) < \UW_q(X) + 2 \varepsilon.
\]
Define
\[
F_i^+ = F_i \bigcup \left( \bigcup_{j \neq j'} \gamma^i_{j,j'} \right)
\]
as the union of~$F_i$ with these geodesics.
By construction, the subsets~$F_i^+$ form a closed covering of~$X$ with intrinsic diameter
\begin{equation} \label{eq:Ui+}
\diam^+(F_i^+) < 2 \, \UW_q(X)+6 \varepsilon.
\end{equation}

Since the vertices of~$X$ are not locally separating, we can slightly move the curves~$\gamma^i_{j,j'}$ without increasing their length too much (keeping the intrinsic diameter bound~\eqref{eq:Ui+}) so that the curves ~$\gamma^i_{j,j'}$ avoid the vertices of~$X$. 
Since the simplices of~$X$ of dimension~$1$ and~$2$ are not isolated, we can also slightly move the curves~$\gamma^i_{j,j'}$ without increasing their length too much so that the curves~$\gamma^i_{j,j'}$ are pairwise disjoint and avoid the fibers over the barycenters of~$\Theta^{n-1}(P)$ corresponding to the points of maximal multiplicity~$q+1$ of the covering~$\{ \st(p_i) \}$.
Note that these fibers are of codimension~$q \geq 2$ in each simplex of~$X$ they intersect.
We can even assume that the curves~$\gamma^i_{j,j'}$ are piecewise linear.
Despite the risk of confusion, we still denote by~$F_i^+$ the union of~$F_i$ with the curves~$\gamma^i_{j,j'}$ thus-modified.

Now, recall that the covering~$\{F_i\}$ is of multiplicity~$q+1$.  
Since the curves~$\gamma^i_{j,j'}$ are disjoint, the only way for the multiplicity of~$\{F_i^+\}$ to be greater than~$q+1$ is if some curve~$\gamma^{i_0}_{j,j'}$ intersects a region of multiplicity~$q+1$ of~$\{ F_i \mid i \neq i_0 \}$.
That is, if $\gamma^{i_0}_{j,j'}$ intersects a region of maximal multiplicity of~$\{F_i\}$, given by the fibers of the barycenters of~$\Theta^{n-1}(P)$.
This is excluded after the previous curve deformation.
Hence, the closed covering~$\{F_i^+\}$ has multiplicity~$q+1$ and satisfies the intrinsic diameter bound~\eqref{eq:Ui+}.

By taking small enough open neighborhoods of the~$F_i^+$, we obtain an open covering of~$X$ with the same properties.
Subdividing~$X$ even further and slightly moving the curves~$\gamma^i_{j,j'}$ if necessary, we can assume that this open covering of~$X$ is given by the open stars of the~$F_i^+$.
This shows that $\UW^+_q(X) \leq 2 \, \UW_q(X) + 6 \varepsilon$.
Hence the proposition.
\end{proof}

\begin{remark} \label{rem:UW+2UW}
The end of Proposition~\ref{prop:UW+2UW} shows that there is a finite covering of~$X$ of multiplicity at most~$q+1$ by open \emph{simplicial} subsets of intrinsic diameter less than $2 \, \UW_q(X) + 6 \varepsilon$.
\end{remark}

\subsection{Diameter and uniform group growth}
\mbox { }

\medskip

Let us present the following classical result relating the diameter and the volume entropy of a space, similar in spirit to the \v{S}varc-Milnor lemma; see~\cite[\S5.16]{gro99}.
We refer to Definition~\ref{def:algent} and Definition~\ref{def:entN} for the basic definitions.

\forget
\begin{definition} \label{def:Gamma-i}
Let $U$ be a connected open subset in a connected finite simplicial complex~$X$.
Denote by $\Gamma_U$ the subgroup of~$\pi_1(X)$ given by the image of~$\pi_1(U)$ under the group homomorphism induced by the inclusion map $i:U \hookrightarrow X$.
That is,
\[
\Gamma_U = i_*[\pi_1(U)].
\]
\end{definition}

We have the following classical diameter-entropy bound similar in spirit to the \v{S}varc-Milnor lemma; see~\cite[\S5.16]{gro99}.
\forgotten

\begin{proposition}\label{prop:diam-ent}
Let $U$ be a connected open simplicial subset in a connected finite simplicial complex~$X$ with a piecewise Riemannian metric.
Then
\[
\diam^+(U) \cdot  \ent(X) \geq \frac{1}{2} \, \ent(\Gamma_U)
\]
where $\Gamma_U := i_*[\pi_1(U)]$ is the image of~$\pi_1(U)$ under the group homomorphism induced by the inclusion map $i:U \hookrightarrow X$.
\end{proposition}

\begin{proof}
The proof of this result is classical; see~\cite[Proposition~3.22]{gro99} for the details.
Since $U$ is a simplicial subset of a finite simplicial complex, its fundamental group~$\pi_1(U)$ is finitely generated and so is~$\Gamma_U$.
Fix~$\varepsilon >0$.
Take a system of loops of~$U$ with basepoint~$x_0$ whose homotopy classes in~$X$ form a finite generating set of~$\Gamma_U = i_*[\pi_1(U,x_0)] \leqslant \pi_1(X,x_0)$.
Decompose these loops into segments of length less than~$\varepsilon$ and connect the endpoints of these segments to~$x_0$ with almost-minimizing arcs of~$U$.
The triangular loops~$\gamma_ i \subseteq U$ thus-formed induce a finite generating set~$S$ of~$\Gamma_U$ in homotopy with 
\[
\length(\gamma_i) < 2 \, \diam^+(U) + \varepsilon.
\]
Clearly, every homotopy class $\alpha \in \Gamma_U$ can be represented by a loop~$\gamma \subseteq U$ based at~$x_0$ of length at most 
\[
(2 \, \diam^+(U) + \varepsilon) \cdot d_S(e,\alpha)
\]
where $d_S$ is the word distance on~$\Gamma_U$ induced by~$S$.
Thus, the number~$\mathcal{N}(X;T)$ of homotopy classes represented by loops based at~$x_0$ of length at most~$T$, see Definition~\ref{def:entN}, satisfies
\[
\mathcal{N}(X;T) \geq  {\rm card} \left\{ \alpha \in\Gamma_U \mid d_S(e,\alpha) \leq \frac{T}{2 \, \diam^+(U) + \varepsilon} \right\}
\]
It follows from~\eqref{eq:entN} that
\[
\eent(X) \geq \frac{1}{2 \, \diam^+(U) + \varepsilon} \, \ent(\Gamma_U,S)
\]
for every $\varepsilon >0$.
Hence the result.
\end{proof}

\subsection{Covering non-collapsing assumption and minimal volume entropy}
\mbox{ } 

\medskip

We can now prove the following result complementing Corollary~\ref{coro:collapsing} under some mild combinatorial assumptions.

\begin{theorem} \label{theo:B.bis}
Every connected finite $2$-regular simplicial $m$-complex~$X$ without locally separating points
and with $m \geq 3$ satisfying the covering non-collapsing assumption has positive minimal volume entropy.

More precisely,
\[
\omega(X) \geq C_m' \, h(X)
\]
where $h(X)$ is the constant in the covering non-collapsing assumption on~$X$ and $C'_m$ is an explicit positive constant depending only on~$m$.
\end{theorem}

\begin{proof}
By Proposition~\ref{prop:UW+2UW} and Remark~\ref{rem:UW+2UW}, for every~$\varepsilon >0$, there exists an open simplicial covering~$\mathcal{U}=\{U_i\}$ of~$X$ of multiplicity at most~$m$ with 
\[
\diam^+(U_i) < 2\UW(X) + \varepsilon.
\]
By the covering non-collapsing assumption, there is an open simplicial subset $U_{i_0}$ of~$\mathcal{U}$ such that the subgroup~$\Gamma_{U_{i_0}}=i_*[\pi_1(U_{i_0})]$ has uniform exponential growth at least~$h(X)$.
It follows from Proposition~\ref{prop:diam-ent} that 
\[
\frac{1}{2} h(X) \leq \frac{1}{2} \ent(\Gamma_{i_0}) \leq \diam^+(U_{i_0}) \cdot \ent(X) \leq (2 \UW(X) + \varepsilon) \cdot \ent(X).
\]
Letting $\varepsilon$ go to zero, we obtain
\begin{equation} \label{eq:entUW}
\ent(X) \cdot \UW(X) \geq \frac{1}{4} h(X)
\end{equation}
By Theorem~\ref{theo:width}, this yields
\[
\ent(X) \cdot \vol(X)^{\frac{1}{m}} \geq C'_m \, h(X)
\]
with $C'_m = \frac{1}{4} C_m^{\frac{1}{m}}$.
\end{proof}

\begin{remark} \label{rem:entUW}
If the simplicial complex~$X$ in Theorem~\ref{theo:B.bis} has small enough volume, its minimal volume entropy is bounded away from zero.
This result still holds true if the unit balls of~$X$ (instead of the whole simplicial complex~$X$) have small enough volume.
Indeed, in this case, we have $\UW(X) \leq 1$ by Theorem~\ref{theo:width}, and the lower bound~\eqref{eq:entUW} leads to $\ent(X) \geq \frac{1}{4} h(x)$.
\end{remark}

\begin{remark}
When $\pi_1(X)$ is thick, we can replace the covering non-collapsing assumption in Theorem~\ref{theo:B.bis} with the fiber non-collapsing assumption by Proposition~\ref{prop:exp.cover}.
In this case, we will see in Theorem~\ref{theo:B} that we can drop the extra combinatorial assumptions.
\end{remark}


\subsection{Handling non-regular simplicial complexes} \label{sec:non-regular}

\mbox{ }

\medskip

In this section, we start with a simplicial complex satisfying the FNCA and replace it with a $2$-regular simplicial complex without locally separating vertices preserving the FNCA with the same constant.
Our goal is to drop the extra combinatorial assumptions in Theorem~\ref{theo:B.bis} for simplicial complexes (with a thick fundamental group) satisfying the FNCA; see Theorem~\ref{theo:B}.

\medskip

Recall that a finite connected simplicial $m$-complex~$X$ satisfies the FNCA if there exists $h(X)>0$ such that for every simplicial map $\pi:X \to P$ onto a simplicial complex~$P$ of dimension~$k < m$, there exists a connected component~$F_{p_0}$ of some fiber~$\pi^{-1}(p_0)$ with $p_0 \in P$ such that the finitely generated subgroup~$i_*[\pi_1(F_{p_0})] \leqslant \pi_1(X)$ has uniform exponential growth at least~$h(X)$.

\medskip

Let $X$ be a finite simplicial $m$-complex with~$m \geq 3$.
Define an extension
\begin{equation} \label{eq:hatX}
\widehat{X} = X \bigcup_i \Delta_i^3
\end{equation}
of~$X$ by attaching a $3$-simplex~$\Delta_i^3$ along every isolated edge~$\Delta_i^1$ or triangle~$\Delta_i^2$ of~$X$ so that the resulting simplicial $m$-complex~$\widehat{X}$ is $2$-regular.
Note that the inclusion $X \hookrightarrow \widehat{X}$ is a $\pi_1$-isomorphism.

\medskip

Replacing~$X$ with the $2$-regular simplicial complex~$\widehat{X}$ does not alter the fiber non-collapsing assumption.

\begin{lemma} \label{lem:hat}
Let $X$ be a finite simplicial $m$-complex with~$m \geq 3$.
If $X$ satisfies the FNCA with constant at least~$h$, then $\widehat{X}$ also satisfies the FNCA with constant at least~$h$.
\end{lemma}

\begin{proof}
Let $\widehat{\pi}:\widehat{X} \to P$ be a simplicial map onto a simplicial $q$-complex~$P$ with $q < m$.
Denote by~$\pi:X \to P$ the restriction of~$\widehat{\pi}:\widehat{X} \to P$ to~$X$.
For every vertex $p \in P$, the $\widehat{\pi}$-fiber over~$p$ decomposes as
\[
\widehat{\pi}^{-1}(p) = \pi^{-1}(p) \bigcup_i \left( \widehat{\pi}^{-1}(p) \cap \Delta_i^3 \right)
\]
where $\Delta_i^3$ runs over the $3$-simplices of~$\widehat{X} \setminus X$.
Since the map $\widehat{\pi}:\widehat{X} \to P$ is simplicial, every block $\widehat{\pi}^{-1}(p) \cap \Delta_i^3$ in the previous decomposition is a $k$-face of~$\Delta_i^3$ with $0 \leq k \leq 3$.
If $\widehat{\pi}^{-1}(p) \cap \Delta_i^3$ is disjoint from~$\pi^{-1}(p)$, then $\widehat{\pi}^{-1}(p) \cap \Delta_i^3$ is a contractible connected component of~$\widehat{\pi}^{-1}(p)$.
If $\widehat{\pi}^{-1}(p) \cap \Delta_i^3$ intersects~$\pi^{-1}(p)$ along a vertex, an edge or a triangle, then $\widehat{\pi}^{-1}(p) \cap \Delta_i^3$ deformation retracts onto this vertex, edge or triangle.
Therefore, every connected component~$\widehat{F}_p$ of~$\widehat{\pi}^{-1}(p)$ is either contractible or deformation retracts onto a connected component~$F_p$ of~$\pi^{-1}(p)$.
In the latter case, the subgroups~$i_*[\pi_1(F_p)] \leqslant \pi_1(X)$ and~$i_*[\pi_1(\widehat{F}_p)] \leqslant \pi_1(\widehat{X})$ have the same growth.
Hence the result.
\end{proof}

We can split simplicial complexes at their locally separating vertices as follows.

\begin{definition} \label{def:Xstar}
Let $X$ be a finite simplicial complex.
Denote by~$X^\star$ the finite simplicial complex obtained by locally disconnecting~$X$ at its locally separating vertices.
This construction comes with a natural simplicial map
\begin{equation} \label{eq:j}
j:X^\star \to X
\end{equation}
injective away from the vertices of~$X^\star$ with 
\[
X = X^\star / \! \sim
\]
where $x_1 \sim x_2$ if $j(x_1)=j(x_2)$.
Observe that the map $j:X^\star \to X$ is $\pi_1$-injective on each connected component of~$X^\star$.
\end{definition}

Splitting a simplicial complex at its locally separating vertices does not alter the fiber non-collapsing assumption either.

\begin{lemma} \label{lem:star}
Let $X$ be a finite simplicial $m$-complex with $m \geq 2$.
Denote by~$X^\star$ the finite simplicial $m$-complex obtained by locally disconnecting~$X$ at its locally separating vertices.
If $X$ satisfies the FNCA with constant at least~$h$, then $X^\star$ also satisfies the FNCA with constant at least~$h$.
\end{lemma}

\forget
\begin{proof}
Suppose that $X$ satisfies the FNCA with constant at least~$h$.
Without loss of generality, we can assume that $X$ is connected.

Let $x$ be a separating vertex of~$X$.
We can split~$X$ at~$x$ into two simplicial complexes~$X_1$ and~$X_2$ with basepoints~$x_1 \in X_1$ and~$x_2 \in X_2$ such that
\[
X = X_1 \vee X_2 = (X_1 \sqcup X_2)/ \! \sim
\]
where $x_1 \sim x_2$ are identified with~$x$.
We can further perform this splitting so that $x_1$ is not a locally separating vertex of~$X_1$.
Note that $x_2$ may still be a locally separating vertex of~$X_2$.
By van Kampen's theorem, we have 
\[
\pi_1(X,x) \simeq \pi_1(X_1,x_1) * \pi_1(X_2,x_2).
\]

Let $\mathcal{V}_i = \{ V_{i,j} \mid j \in J_i \}$ be an open covering of~$X_i$ of multiplicity at most~$m$ with $V_{i,j}$ connected.
Switching the indices if necessary, we can assume that $x_1 \in V_{1,1}$ and $x_2 \in V_{2,1}$.
Fix three small enough contractible open balls $B_1 \subseteq B_2 \subseteq B_3 \subseteq X_1$ around~$x_1 \in X_1$ whose closures~$\widebar{B}_i$ are still contractible and lie in ~$V_{1,1}$.

Define $U_{1,j} \subseteq X_1 \subseteq X = X_1 \vee X_2$ with $j \in J_1$ as
\[
U_{1,j} =
\begin{cases}
V_{1,1} \setminus \widebar{B}_1 & \text{if } j=1 \\
V_{1,j} \setminus \widebar{B}_3 & \text{otherwise.}
\end{cases}
\]
Define also $U_{2,j} \subseteq X = X_1 \vee X_2$ with $j \in J_2$ as
\[
U_{2,j} =
\begin{cases}
V_{2,1} \cup B_2 & \text{if } j=1 \\
V_{2,j} \cup B_1 & \text{if } j \neq 1 \text{ and } x_2 \in V_{2,j} \\
V_{2,j}  & \text{otherwise.}
\end{cases}
\]
By construction, the subsets~$U_{i,j}$ are connected and form an open covering~$\mathcal{U}$ of~$X=X_1 \vee X_2$ of multiplicity at most~$m$ with
\[
i_*[\pi_1(U_{i,j})] \simeq i_*[\pi_1(V_{i,j})]
\]
by contractibility of~$\widebar{B}_i$.
Since $X$ satisfies the FNCA with constant at least~$h$, the same goes with the disjoint union~$X_1 \sqcup X_2$.
Repeating this process over and over with the remaining separating vertices, we can assume that $X$ has no separating vertices anymore.

\medskip

Let $x$ be a locally separating vertex of~$X$.
We can split~$X$ at~$x$ into a simplicial $m$-complex~$Y$ with two vertices~$y_1$ and~$y_2$ such that
\[
X = Y/ \! \sim
\]
where $y_1 \sim y_2$ are identified with~$x$.
We can further perform this splitting so that $y_1$ is not a locally separating vertex of~$Y$.
Note that $y_2$ may still be a locally separating vertex of~$Y$.
By van Kampen's theorem, we have
\[
\pi_1(X) \simeq \pi_1(Y) * \Z.
\]

Let $\mathcal{V} = \{ V_j \mid 1 \leq j \leq k \}$ be an open covering of~$Y$ of multiplicity~$m$ with $V_j$ connected.
Switching the indices if necessary, we can assume that $y_1 \in V_1$ and $y_2 \in V_2$.
Fix three small enough contractible open balls $B_1 \subseteq B_2 \subseteq B_3 \subseteq Y$ around~$y_1 \in Y$ whose closures~$\widebar{B}_i$ are still contractible and lie in~$V_1$.

Define $U_i \subseteq X=Y/ \! \sim$ with $1 \leq i \leq k$ as follows.

\noindent If $y_2 \notin V_2$, then 
\[
U_1 = V_1 \setminus \widebar{B}_2, \quad U_2 = V_2 \cup B_3, \quad
U_{i>2} = 
\begin{cases}
V_i \setminus \widebar{B}_3 & \text{ if } y_2 \notin V_i \\
(V_i \setminus \widebar{B}_3) \cup B_1 & \text{ if } y_2 \in V_i.
\end{cases}
\]
If $y_2 \in V_2$, then
\[
U_1 = V_1, \quad
U_{i>1} = 
\begin{cases}
V_i \setminus \widebar{B}_2 & \text{ if } y_2 \notin V_i \\
(V_i \setminus \widebar{B}_2) \cup B_1 & \text{ if } y_2 \in V_i.
\end{cases}
\]
In both cases, the subsets~$U_i$ are connected and form an open covering~$\mathcal{U}$ of~$X=Y/ \! \sim$ of multiplicity at most~$m$ with $i_*[\pi_1(U_0)] = \{ e \}$ and
\[
i_*[\pi_1(U_i)]  \simeq i_*[\pi_1(V_i)]
\]
by contractibility of~$\widebar{B}_i$.
Since $X$ satisfies the FNCA with constant at least~$h$, the same goes with~$Y$.
Repeating this process over and over with the remaining locally separating vertices, we obtain the simplicial complex~$X^\star$, which shows that $X^\star$ satisfies the FNCA with constant at least~$h$.
\end{proof}
\forgotten

\begin{proof}
Suppose that $X$ satisfies the FNCA with constant at least~$h$.
Without loss of generality, we can assume that $X$ is connected.

Let $x$ be a locally separating vertex of~$X$.
We can split~$X$ at~$x$ into $k$ connected simplicial complexes $\{ X_i \mid 1 \leq i \leq k \}$ with $k_i$ non locally separating vertices $\{ x_j^i \mid 1 \leq j \leq k_i \}$ in each~$X_i$ such that
\[
X = (X_1 \sqcup \dots \sqcup X_k)/ \! \sim
\]
where all the vertices~$x_j^i \in X_i$ are identified with~$x$.
By van Kampen's theorem, we have 
\[
\pi_1(X,x) \simeq \Asterisk_{i=1}^ k \left( \pi_1(X_i,x_1^i) * F_{k_i-1} \right)
\]
where $F_r$ is the free group of rank~$r$.

Let $\mathcal{V}_i = \{ V_{i,\alpha} \mid \alpha \in A_i \}$ be an open covering of~$X_i$ of multiplicity at most~$m$ with $V_{i,\alpha}$ connected.
Slightly perturbing the covering if necessary, we can assume that $x_j^i \notin \partial V_{i,\alpha}$ for all the indices.
In particular, we can fix three (small) contractible open metric balls $B^-_{i,j} \subsetneq B_{i,j} \subsetneq B^+_{i,j} \subseteq X_i$ around each vertex~$x_j^i \in X_i$ such that
\begin{enumerate}
\item the closures $\widebar{B}^-_{i,j}$, $\widebar{B}_{i,j}$ and~$\widebar{B}^+_{i,j}$ of these balls are still contractible;
\item the balls~$\widebar{B}^+_{i,j}$ are disjoint;
\item $\widebar{B}^+_{i,j}$ lies in~$V_{i,\alpha}$ if $x_j^i \in V_{i,\alpha}$;
\item $\widebar{B}^+_{i,j}$ is disjoint from~$V_{i,\alpha}$ if $x_j^i \notin V_{i,\alpha}$.
\end{enumerate}

Loosely speaking, for every vertex~$x_j^i$, we choose an open set~$V_{i,\alpha^i_j}$ containing~$x_j^i$ and remove from each open set~$V_{i,\alpha}$ a ball~$\widebar{B}^-_{i,j}$ or~$\widebar{B}^+_{i,j}$ around each vertex~$x_j^i$, where this ball is~$\widebar{B}^-_{i,j}$ if~$V_{i,\alpha}$ is the chosen open set~$V_{i,\alpha_j^i}$ containing~$x_j^i$ and is~$\widebar{B}^+_{i,j}$ otherwise.
Observe that the resulting open sets~$U_{i,\alpha} \subseteq X$ are connected and that removing the contractible balls~$\widebar{B}^-_{i,j}$ or~$\widebar{B}^+_{i,j}$ from the open sets~$V_{i,\alpha}$ does not change the images of their fundamental groups in~$\pi_1(X)$.
In particular, the images of the fundamental groups of~$U_{i,\alpha}$ and~$V_{i,\alpha}$ in~$\pi_1(X)$ are the same.
Now, the multiplicity of the~$U_{i,\alpha}$ is the same as the multiplicity of the~$V_{i,\alpha}$ at every point of~$X$, except in the neighborhood~$\bigcup_{i,j} \widebar{B}^-_{i,j}$ of~$x$, where it is equal to zero, and on the corona $\bigcup_{i,j} \widebar{B}^+_{i,j} \setminus \widebar{B}^-_{i,j}$, where it is equal to one.
To obtain an open covering of~$X$ with the desired properties, we add the contractible open neighborhood~$\bigcup_{i,j} B_{i,j}$ of~$x \in X$.

More formally, for every $1 \leq i \leq k$ and $1 \leq j \leq k_i$, fix $\alpha_j^i \in A_i$ such that $x_j^i \in V_{i,\alpha_j^i}$.
It may happen that $\alpha_j^i=\alpha_{j'}^i$ for $j \neq j'$.
Let 
\[
J_\alpha^i = \{ j \mid \alpha_j^i = \alpha \}.
\]
Define the open sets $U_{i,\alpha} \subseteq X_i \setminus \{ x_j^i \mid 1 \leq j \leq k_i \} \subseteq X$ with $\alpha \in A_i$ as follows:
\[
U_{i,\alpha} = V_{i,\alpha} \setminus \left[ \left( \bigcup_{j \in J_\alpha^i} \widebar{B}^-_{i,j} \right) \bigcup \left( \bigcup_{j \notin J_\alpha^i} \widebar{B}^+_{i,j'} \right) \right].
\]
Define also the open neighborhood~$U_0 \subseteq X$ of~$x$ as
\[
U_0 = \bigcup_{i,j} B_{i,j}.
\]
By construction, the subsets~$U_0$ and~$U_{i,\alpha}$ are connected and form an open covering~$\mathcal{U}$ of~$X$ of multiplicity at most~$m$ with $i_*[\pi_1(U_0)] = \{ e \}$ and
\[
i_*[\pi_1(U_{i,\alpha})]  \simeq i_*[\pi_1(V_{i,\alpha})]
\]
by contractibility of~$\widebar{B}_i$.
Since $X$ satisfies the FNCA with constant at least~$h$, one of the subgroups~$i_*[\pi_1(U_{i_0,\alpha_{i_0}})]$ has uniform exponential growth at least~$h$ and so does~$ i_*[\pi_1(V_{i_0,\alpha_{i_0}})]$.
Thus, the simplicial complex~$X_1 \sqcup \dots \sqcup X_k$ also satisfies the FNCA with constant at least~$h$.

Repeating this process over and over with the remaining locally separating vertices, we obtain the simplicial complex~$X^\star$, which shows that $X^\star$ satisfies the FNCA with constant at least~$h$.
\end{proof}

Splitting a simplicial complex at its locally separating vertices does not increase its volume entropy.

\begin{lemma} \label{lem:star2}
Let $X$ be a finite simplicial $m$-complex with a piecewise Riemannian metric.
Denote by~$X^\star$ the finite simplicial $m$-complex obtained by locally disconnecting~$X$ at its locally separating vertices.
Endow~$X^\star$ with the piecewise Riemannian metric pulled back by the simplicial map $j:X^\star \to X$.
Then every connected component~$Z$ of~$X^\star$ satisfies
\[
\ent(Z) \leq \ent(X).
\]
\end{lemma}

\begin{proof}
By construction, the $\pi_1$-injective map $j:Z \to X$ is $1$-Lipschitz and volume-preserving, and so is its lift $\tilde{j}: \tilde{Z} \to \tilde{X}$ to the universal covers of~$Z$ and~$X$.
Therefore,
\[
\tilde{j}(B_{\tilde{Z}}(R)) \subseteq B_{\tilde{X}}(R)
\]
and
\[
\vol \, B_{\tilde{Z}}(R) = \vol \, \tilde{j}(B_{\tilde{Z}}(R)) \leq \vol \, B_{\tilde{X}}(R)
\]
for some $R$-balls~$B_{\tilde{Z}}(R) \subseteq \tilde{Z}$ and~$B_{\tilde{X}}(R) \subseteq \tilde{X}$.
Hence,
\[
\ent(Z) \leq \ent(X).
\]
\end{proof}

\subsection{Fiber non-collapsing assumption and minimal volume entropy}
\mbox{ } 

\medskip

We can now prove the following result complementing Theorem~\ref{theo:A} when the fundamental group is thick.

\begin{theorem} \label{theo:B}
Let $X$ be a connected finite simplicial $m$-complex with thick fundamental group and $m \geq 3$.
If $X$ satisfies the fiber non-collapsing assumption, then $X$ has positive minimal volume entropy.

More precisely,
\[
\omega(X) \geq C_m' \, h(X)
\]
where $h(X)$ is the constant in the fiber non-collapsing assumption on~$X$ and $C'_m$ is an explicit positive constant depending only on~$m$.
\end{theorem}


\begin{proof}
Suppose that $X$ is equipped with a piecewise Riemannian metric.
This metric can be extended into a piecewise Riemannian metric on the $2$-regular simplicial complex~$\widehat{X}$ defined in~\eqref{eq:hatX} so that the inclusion~$X \hookrightarrow \widehat{X}$ is distance preserving with
\begin{equation} \label{eq:ent-estimate1}
\vol(\widehat{X}) \simeq \vol(X) \quad \text{ and } \quad \ent(\widehat{X}) \simeq \ent(X)
\end{equation}
by taking a suitable Riemannian metric on each $3$-simplex~$\Delta_i^3$ in~\eqref{eq:hatX} collapsing to the Riemannian metric of the edge~$\Delta_i^1$ or triangle~$\Delta_i^2$ of~$X$ to which the $3$-simplex~$\Delta_i^3$ is attached.
Here, the symbol~$\simeq$ means that the equality holds up to an arbitrarily small positive constant.
Endow the simplicial $m$-complex~$\widehat{X}^\star$ obtained by locally disconnecting~$\widehat{X}$ at its locally separating vertices with the piecewise Riemannian metric pulled back by the $\pi_1$-injective natural map $j:\widehat{X}^\star \to \widehat{X}$; see Definition~\ref{def:Xstar}.
By Lemma~\ref{lem:star2}, every connected component~$Z$ of~$\widehat{X}^\star$ satisfies
\begin{equation} \label{eq:ent-estimate2}
\vol(Z) \leq \vol(\widehat{X}) \quad \text{ and } \quad \ent(Z) \leq \ent(\widehat{X}).
\end{equation}

By Lemma~\ref{lem:hat} and Lemma~\ref{lem:star}, there exists a connected component~$Z_0$ of~$\widehat{X}^\star$ satisfying the fiber non-collapsing assumption with constant at least~$h(X)$.
Observe that the simplicial complex~$Z_0$ is of dimension~$m$, otherwise we would obtain a contradiction by taking for \mbox{$\pi:Z_0 \to P$} the identity map $Z_0 \to Z_0$ in the definition of the fiber non-collapsing assumption.

Now, since the simplicial complex~$\widehat{X}^\star$ is $2$-regular without locally separating vertices, see Section~\ref{sec:non-regular}, its connected component~$Z_0$ is also $2$-regular without locally separating vertices.
It follows from the estimates~\eqref{eq:ent-estimate1} and~\eqref{eq:ent-estimate2}, and Theorem~\ref{theo:B.bis} that
\[
\omega(X) \simeq \omega(\widehat{X}) \geq \omega(Z_0) \geq C_m' \, h(X)
\]
where $C_m'=\frac{1}{4} C_m^{\frac{1}{m}}$.
Hence, the minimal volume of~$X$ is positive.
\forget
By definition of the modified Urysohn width, see Definition~\ref{def:UW'}, for every $\varepsilon >0$, there exists a finite covering~$\mathcal{U} = \{ U_i \}$ of~$X$ of multiplicity at most~$m$ by connected open subsets with
\begin{equation} \label{eq:diam1}
\diam^+ (U_i) < \UW^+(Z_0) + \varepsilon.
\end{equation}

By the fiber non-collapsing assumption, which still holds on~$Z_0$ with constant~$h(X)$, one of the subgroups~$\Gamma_{i_0} = i_*[\pi_1(U_{i_0})] \leqslant \pi_1(Z_0)$ has uniform exponential growth at least~$h(X)$.
It follows from Proposition~\ref{prop:diam-ent} and the entropy estimates of~\eqref{eq:ent-estimate1} and~\eqref{eq:ent-estimate2} that 
\begin{equation} \label{eq:chain-diam}
\frac{1}{2} \, h(X) \leq \frac{1}{2} \, \ent(\Gamma_{i_0}) \leq \diam^+(U_{i_0}) \cdot \ent(Z_0) \lesssim \diam^+(U_{i_0}) \cdot \ent(X).
\end{equation}

On the other hand, by construction, the simplicial complex~$\widehat{X}^\star$ is $2$-regular without locally separating vertices, see Section~\ref{sec:non-regular}, and so is its connected component~$Z_0$.
It follows from Proposition~\ref{prop:UW+2UW} that
\[
\UW^+(Z_0) \leq 2 \, \UW(Z_0).
\]

Now, by Theorem~\ref{theo:width} and the volume estimates of~\eqref{eq:ent-estimate1} and~\eqref{eq:ent-estimate2}, we have
\[
\UW(Z_0) \leq C_m'' \, \vol(Z_0)^\frac{1}{m} \lesssim  C_m'' \, \vol(X)^\frac{1}{m}
\]
where $C_m'' = C_m^{-\frac{1}{m}}$.
Therefore, 
\begin{equation} \label{eq:UW2vol}
\UW^+(Z_0) \lesssim  2 \, C_m'' \, \vol(X)^\frac{1}{m}.
\end{equation}

Combining the inequalities~\eqref{eq:chain-diam}, \eqref{eq:diam1} and~\eqref{eq:UW2vol}, we derive
\[
\omega(X) \geq C_m' \, h(X) > 0
\]
where $C_m'=\frac{1}{4 C_m''} = \frac{1}{4} C_m^\frac{1}{m}$.
Hence, the minimal volume entropy of~$X$ is positive.
\forgotten
\end{proof}

\begin{remark} \label{rem:entUW2}
As in Remark~\ref{rem:entUW}, if the unit balls of a simplicial complex~$X$ in Theorem~\ref{theo:B} have small enough volume, the minimal volume entropy of~$X$ is bounded away from zero.
\end{remark}

\begin{remark}
By Proposition~\ref{prop:hyp}, Theorem~\ref{theo:B} applies to finite aspherical simplicial $m$-complexes~$X$ with a non-elementary word hyperbolic fundamental group and $H_m(X;\R)$ nontrivial.
Thus, these simplical complexes~$X$ have positive minimal volume entropy.
This result can also be obtained using filling techniques; see~\cite{minentB} and~\cite{sab17}.
\end{remark}




\forget
\begin{theorem} \label{theo:B}
Every connected finite simplicial $m$-complex~$X$ satisfying the fiber non-collapsing assumption has positive minimal volume entropy.
More generally (by contraposition), given a piecewise Riemannian metric~$g$ on~$X$, if for some $R>0$ every ball~$B(R) \subseteq X$ of radius~$R$ has volume at most~$C_m \, R^m$ then
\[
\ent(X,g) \geq \frac{h(X)}{2R}
\]
where $C_m$ is an explicit positive constant depending only on~$m$.
\end{theorem}

\begin{proof}
Let $g$ be a piecewise Riemannian metric on~$X$.
By Theorem~\ref{theo:width}  and Proposition~\ref{prop:width}, there exists a constant~$C'_m > 0$ and a simplicial map $\pi:X \to P$ to some simplicial $(m-1)$-complex~$P$, where every fiber~$F_p=\pi^{-1}(p)$ is connected, such that 
\begin{equation} \label{eq:diam1}
\diam_X (F_p) < C'_m \, \vol(X,g)^{\frac{1}{m}}.
\end{equation}
The constant $C'_m$ can be taken arbitrarily close to~$C_m^{-\frac{1}{m}}$, where $C_m$ is the constant in Theorem~\ref{theo:width}.
By the fiber non-collapsing assumption, one of the subgroups~$i_*[\pi_1(F_{p_0})]$ has uniform exponential growth at least~$h(X)$.
The point~$p_0 \in P$ has a small enough open neighborhood $B_{p_0} \subseteq  P$ whose preimage $U_{p_0}=\pi^{-1}(B_{p_0})$ is homotopy equivalent to the fiber~$F_{p_0}$ and has diameter arbitrarily close to the one of~$F_{p_0}$ so that 
\begin{equation} \label{eq:diam2}
\diam_X (U_{p_0}) < C'_m \, \vol(X,g)^{\frac{1}{m}}.
\end{equation}
The group~$\Gamma_{p_0}=i_*[\pi_1(U_{p_0})]$, which is isomorphic to the finitely generated subgroup~$i_*[\pi_1(F_{p_0})]$, has uniform exponential growth at least~$h(X)$.
That is, $\ent(\Gamma_{p_0}) \geq h(X)$.
It follows from Proposition~\ref{prop:diam-ent} and~\eqref{eq:diam2} that 
\begin{equation} \label{eq:chain-diam}
\frac{1}{2} \, h(X) \leq \frac{1}{2} \, \ent(\Gamma_{p_0}) \leq \diam_X(U) \cdot \ent(X,g) \leq C'_m \, \ent(X,g) \, \vol(X,g)^{\frac{1}{m}}.
\end{equation}
Hence, the minimal volume entropy of~$X$ is positive and satisfies 
\[
\omega(X) \geq \frac{1}{2 C'_m} \, h(X) > 0. 
\]

Suppose that every ball~$B(R)$ of radius~$R$ has volume at most~$C_m \, R^m$, where $C_m$ is the constant in Theorem~\ref{theo:width}.
By Theorem~\ref{theo:width}, we can replace the right-hand sides in the inequalities~\eqref{eq:diam1} and~\eqref{eq:diam2} by~$R$.
By making the appropriate change in the inequality chain~\eqref{eq:chain-diam}, we derive the desired lower bound for the volume entropy of~$X$.
\end{proof}
\forgotten



\subsection{Simplicial volume and minimal volume entropy} \label{subsec:ex}

\mbox{ } 
\medskip



We construct a sequence of simplicial complexes~$Z_m$ with zero simplicial volume and arbitrarily large minimal volume entropy. \\

Remove a ball from a closed manifold of dimension~$m=2k \geq 4$ with positive simplicial volume.
The resulting space~$\Sigma$ is a manifold with boundary $\partial \Sigma \simeq S^{2k-1}$.
Fix an integer $d \geq 3$.
Denote by~$Y$ the quotient of~$\Sigma$ by the natural free action of~$\Z_d$ on~$S^{2k-1}$ given by rotation of the Hopf fibration.
Observe that $\pi_1(Y) \simeq \pi_1(\Sigma) \ast \Z_d$ and $H_m(Y;\Z)=0$.
Define the simplicial $m$-complex
\[
X_n = \#_{i=1}^n Y_i
\]
by taking the connected sum of $n$ copies of~$Y$.
Note that $H_m(X_n;\Z)=0$.

The space~$X_n$ admits a $d$-sheeted cyclic cover which can be described as follows.
The connected sum $\#_{i=1}^n \Sigma_i$ of $n$ copies of~$\Sigma$ is a manifold whose boundary identifies with the disjoint union~$\sqcup S_i^{2k-1}$ of $n$ spheres.
Let $\widehat{X}_n$ be the space obtained by gluing $d$ copies of~$\#_{i=1}^n \Sigma_i$ along this disjoint union
\[
\widehat{X}_n = ( \sqcup S_i^{2k-1} ) \cup_{\psi_1} ( \#_{i=1}^n \Sigma_i ) \cdots \cup_{\psi_d} ( \#_{i=1}^n \Sigma_i )
\]
where the attaching maps~$\psi_j$ are given by the action of~$\alpha^j$ on the boundary components of~$\#_{i=1}^n \Sigma_i$ (for a fixed generator~$\alpha$ of~$\Z_d$).
The cover $\widehat{X}_n \to X_n$ is the natural map sending the $d$ copies $\#_{i=1}^n \Sigma_i$ to~$X_n$.
By the comparison principle, see~\cite[Lemma~4.1]{Brunnbauer08}, we have 
\begin{equation} \label{eq:d}
\omega(\widehat{X}_n) \leq d^\frac{1}{m} \, \omega(X_n).
\end{equation}

Now, take two copies $\#_{i=1}^n \Sigma_i$ and $\#_{i=1}^n \bar{\Sigma}_i$ in~$\widehat{X}_n$.
By construction, the boundaries~$\partial \Sigma_i$ and~$\partial \bar{\Sigma}_i$ agree and the union
\[
M_n = ( \#_{i=1}^n \Sigma_i ) \cup ( \#_{i=1}^n \bar{\Sigma}_i )
\]
is a closed $m$-manifold homeomorphic to
\[
M_n \simeq \#_{i=1}^n (  \Sigma_i \# \bar{\Sigma}_i ) \, \, \#_{i=1}^n ( S^1 \times S^{2k-1} ).
\]
Since the simplicial volume is additive under connected sums in dimension at least three, see~\cite{gro82}, we obtain
\[
\Vert M_n \Vert_\Delta = 2n \, \Vert \Sigma \Vert_\Delta > 0.
\]
Thus, by~\eqref{eq:gro}, the minimal volume entropy~$\omega(M_n)$ of~$M_n$ goes to infinity when $n$ tends to infinity.

To conclude, consider the simplicial $m$-complex~$Z_n$ defined as the connected sum
\[
Z_n= X_n \# \T^m.
\]
Clearly, $H_m(Z_n;\Z)=\Z$ and $\Vert Z_n \Vert_\Delta = 0$. 
Observe that $Z_n$ is a cellular $m$-complex with a single $m$-cell.
Note also that $Z_n$ is not aspherical since its fundamental group has torsion.
By the estimate $\omega(N_1)^m \leq \omega(N_1 \# N_2)^m$ established in~\cite[Theorem~2.12]{BS} for connected closed $m$-pseudomanifolds~$N_1$ and~$N_2$ with $m \geq 3$ and~$N_2$ orientable (which still holds when $N_1$, here~$X_n$, is a cellular $m$-complex with a single $m$-cell), we have $\omega(Z_n) \geq \omega(X_n)$.
Since $\pi_1(M_n)$ is a subgroup of~$\pi_1(\widehat{X}_n)$ and the manifold~$M_n$ contained in~$\widehat{X}_n$ has the same dimension~$m$ as~$\widehat{X}_n$, we deduce that $\omega(\widehat{X}_n) \geq \omega(M_n)$.
Thus, by~\eqref{eq:d}, the minimal volume entropy~$\omega(Z_n)$ of~$Z_n$ goes to infinity.

\begin{remark}
Similar examples exist in odd dimensions but their construction is more technical.
\end{remark}


\end{document}